\documentclass[12pt]{amsart}
\usepackage{amsmath, amssymb, latexsym, amsthm, amsfonts, mathtools, xparse}
\usepackage{mathrsfs}
\usepackage{url}
\usepackage{hyperref}
\usepackage{xcolor}  
\usepackage{geometry}
\mathtoolsset{showonlyrefs}
\usepackage{parskip}
\usepackage{paralist}
\usepackage[T1]{fontenc}
\usepackage[backend=biber, style=alphabetic, giveninits]{biblatex}

\numberwithin{equation}{section}
\hypersetup{
    colorlinks=true, 
    linktoc=all,     
    linkcolor=blue,  
}
\theoremstyle{plain}
\newtheorem{thm}{Theorem}[section]
\newtheorem{lem}[thm]{Lemma}
\newtheorem{prop}[thm]{Proposition}
\newtheorem{cor}[thm]{Corollary}
\newcommand{\thmref}[1]{Theorem~\ref{#1}}
\newcommand{\lemref}[1]{Lemma~\ref{#1}}
\newcommand{\propref}[1]{Proposition~\ref{#1}}
\newcommand{\corref}[1]{Corollary~\ref{#1}}

\theoremstyle{definition}
\newtheorem{rmk}[thm]{Remark}

\newtheorem{conjecture}[thm]{Conjecture}

\newcommand{\conjectureRef}[1]{Conjecture~\ref{#1}}

\newcommand{\psmb}{\left( \begin{smallmatrix}}
\newcommand{\psme}{ \end{smallmatrix} \right)}
\newcommand{\smat}[4]{\left( \begin{smallmatrix} #1 & #2 \\ #3 & #4 \\ \end{smallmatrix} \right)}
\newcommand{\mat}[4]{ \begin{smallmatrix} #1 & #2 \\ #3 & #4 \\ \end{smallmatrix} }

\DeclareMathOperator{\im}{Im}
\DeclareMathOperator{\tr}{tr}
\DeclareMathOperator{\rk}{rank}

\DeclareMathOperator{\sgn}{sgn}

\renewcommand*{\mod}{\operatorname{mod}}

\newcommand{\z}{\mathbb{Z}}
\newcommand{\h}{\mbb{H}}
\newcommand{\htwo}{\mbb{H}_2}
\newcommand{\hn}{\mbb H_n}
\newcommand{\lan}{\langle}
\newcommand{\ran}{\rangle}

\newcommand{\q}{\quad}

\newcommand{\prodd}{\prod \nolimits}

\newcommand*{\Q}{\mathbb{Q}}
\newcommand{\N}{\mbb{N}}
\newcommand{\R}{\mbb R}
\newcommand*{\complex}{\mathbb{C}}

\newcommand{\mbb}{\mathbb}
\newcommand{\mc}{\mathcal}
\newcommand{\mrm}{\mathrm}

\newcommand{\n}{\nonumber}
\newcommand{\mf}{\mathbf}

\newcommand{\bnk}{\mc B_k^n}

\newcommand*{\bkzz}[2]{\mbb{BK}_k(#1, #2)}
\newcommand{\snk}{S^n_k}
\newcommand{\spJn}{\sup(J^{cusp}_{k, S; (n,g)})}

\newcommand{\petf}{\lan F, F \ran}

\newcommand{\bksk}{ \sum \nolimits_{F \in \skkunit} \det(Y)^k |F(Z)|^2 }
\newcommand{\bkdegtwo}{ \sum \nolimits_{F \in \mc B_k^2} \det(Y)^k |F(Z)|^2 }
\newcommand{\sumn}{\sum \limits}

\newcommand*{\GL}[2]{\operatorname{GL}_{#1}(#2)}
\newcommand*{\SL}[2]{\operatorname{SL}_{#1}(#2)}
\newcommand*{\Sp}[2]{\operatorname{Sp}_{#1}(#2)}
\newcommand*{\GSp}[2]{\operatorname{GSp}^{+}_{#1}(#2)}
\newcommand*{\Sym}[2]{\operatorname{Sym}_{#1}(#2)}

\newcommand{\tp}[1]{#1^t}

\newcommand{\bkm}{ \mbb B_{k,m}(\tau,z)}

\newcommand{\saitobk}{B(\mrm{SK}_k)}
\newcommand{\sk}{\sup  (\mrm{SK}_k)}
\newcommand{\skk}{\mrm{SK}_k}
\newcommand{\skkunit}{\mc B^*_k}
\newcommand{\jk}{J_{k,1}^{cusp}}

\newcommand{\jkm}{J^{cusp}_{k,m}}
\newcommand{\jkmone}{J^{cusp, new1}_{k,m}}
\newcommand{\jkmtwo}{J^{cusp, new2}_{k,m}}
\newcommand{\jkp}{J_{k,p}^{cusp}}
\newcommand{\jkpn}{J_{k,p}^{cusp,new}}
\newcommand{\jkmn}{J_{k,m}^{cusp,new}}
\newcommand{\bksg}{\mathbb B_{k,S,g}(\tau,z)}

\newcommand{\ff}{F_f^\circ}

\newcommand{\ww}{\widetilde{W}}
\newcommand{\qtr}{Q_r(T)}

\newcommand*{\norm}[1]{\left\lVert#1\right\rVert}

\title[sup-norms on average]{Jacobi forms, Saito-Kurokawa lifts, their Pullbacks and sup-norms on average}

\keywords{Sup-norm, Bergman kernel, Jacobi forms, Saito-Kurokawa lifts, central values of twisted $L$-functions, Eichler-Zagier maps}

\author{Pramath Anamby }
\address{Department of Mathematics\\ 
Indian Institute of Science Education and Research\\ 
Pune -- 411008, India.}
\email{pramath.anamby@gmail.com}

\author{Soumya Das}
\address{Department of Mathematics\\ 
Indian Institute of Science\\ 
Bengaluru -- 560012, India.}
\email{soumya@iisc.ac.in}

\date{}
\subjclass[2020]{Primary 11F46, 11F50, 11F66, Secondary 11F11, 11F30, 11F37}

\begin{document}

\begin{abstract}
    We formulate a precise conjecture about the size of the $L^\infty$-mass of the space of Jacobi forms on $\h_n \times \mbb C^{g \times n}$ of matrix index $S$ of size $g$. This $L^\infty$-mass is measured by the size of the Bergman kernel of the space. We prove the conjectured lower bound for all such $n,g,S$ and prove the upper bound in the $k$ aspect when $n=1$, $g \ge 1$. 
    
    When $n=1$ and $g=1$, we make a more refined study of the sizes of the index-(old and) new spaces, the latter via the Waldspurger's formula. Towards this and with independent interest,
    we prove a power saving asymptotic formula for the averages of the twisted central $L$-values $L(1/2, f \otimes \chi_D)$ with $f$ varying over newforms of level a prime $p$ and even weight $k$ as $k,p \to \infty$ and $D$ being (explicitly) polynomially bounded by $k,p$. Here $\chi_D$ is a real quadratic Dirichlet character.
    
    We also prove that the size of the space of Saito-Kurokawa lifts (of even weight $k$) is $k^{5/2}$ by three different methods (with or without the use of central $L$-values), and show that the size of their pullbacks to the diagonally embedded $\h \times \h$ is $k^2$.
    In an appendix, the same question is answered for the pullbacks of the whole space $S^2_k$, the size here being $k^3$.
\end{abstract}

\maketitle
\section{Introduction}
This article addresses the classical sup-norm problem in the context of automorphic forms for the spaces of holomorphic Jacobi and Siegel modular forms -- these correspond to automorphic forms on the Jacobi group and the Siegel modular group respectively. While Jacobi theta functions are one of the oldest examples of automorphic forms, Siegel modular forms (SMF in short) play important roles in the study of quadratic forms and higher dimensional Abelian varieties among various other things. Let $S^n_k$ be the space of holomorphic cuspidal SMF of scalar weight $k$ on $\Gamma_n:= \Sp {n}{\z}$. When $k$ is even, one has inside $S^2_k$ the space of Saito-Kurokawa lifts, denoted by $\skk$ (SK lifts in short), which are functorial lifts from elliptic modular forms of weight $2k-2$ on $\SL{2}{\z}$. On this space, there are deep and fascinating
period/pullback formulas (cf. \cite{ichino}) and conjectures about the $L^2$-`mass' (cf. \cite{liu2014growth}, \cite{BKY}) for the pullback $F^\circ$ of an SK lift $F$, where $F^\circ(\tau,\tau') = F(\smat{\tau}{0}{0}{\tau'})$, ($\tau,\tau' \in \h$). Of high contemporary interest has been to understand the $L^p$-`mass' ($1 \le p \le \infty$) of these spaces and specifically the sup-norm bounds of the pullbacks of the eigenfunctions of all Hecke operators. Here and throughout this paper, the `size' or $L^\infty$-mass of a subspace of the space of automorphic forms on a symmetric space $S$ with respect to a suitable discrete group $\Gamma$ is measured by the supremum on $S$ of its Bergman kernel (BK in short) for the subspace which is invariant under the discrete group $\Gamma$, see e.g., \eqref{skbk-def}, where $S=\hn, \Gamma=\Gamma_n$ and the subspace is $\skk$.

Let $\hn$ be the Siegel's upper half-space of degree $n$.
Let us recall that the sup-norm problem for $F \in  S^n_k$ with Petersson norm $\norm{F}_2=1$, asks for a bound of the form
$\sup \nolimits_{Z \in \hn} \det(Y)^{k/2} |F(Z)| \ll P(k)$,
where $P$ is a suitable polynomial function of $k$. This problem has been studied extensively in the past few decades for almost all kinds of automorphic objects -- and has its origin in analysis, see e.g., the introductions in \cite{sd-hk}, \cite{blo}, \cite{blomer2016supnorm} etc. for more on the context and relevance of this old and classical problem. For SMF, however, the first general result seems to be that in \cite{sd-hk}, where a precise conjecture for $P(k)$ was also mentioned (viz. $P(k) = k^{n(n+1)/8}$). The two main ingredients required for successful results on this problem are the Fourier expansion and the geometric side of the Bergman kernel of the space.

However if one works with a (smaller) \textit{sub-family}, the above-mentioned feature concerning the BK is no longer present or expected to be tractable – and thus the problem has to be dealt with differently with new ideas, as we will discuss shortly. We only mention here that our sub-family will be the SK lifts -- and that these objects also serve as a very good testing ground for various open questions, in particular the very general (and hard) conjecture on sup-norm: for any $\epsilon>0$,
\begin{equation} \label{supnormconj}
    \norm{F}_\infty := \sup \nolimits_{Z \in \hn} \det(Y)^{k/2} |F(Z)| \ll_{\epsilon,n} k^{n(n+1)/8 +\epsilon} \q \q (F \in \bnk),
\end{equation}
where $\bnk$ denotes the $L^2$-normalised Hecke basis of $\snk$, and $k \to \infty$.

If one assumes the GLH for all twisted $\mrm{GL}(2)$ $L$-functions associated to holomorphic newforms, then a result of Blomer \cite{blo} demonstrates \eqref{supnormconj} when $n=2$ for SK lifts. By all means, the best one can do at present seems to be the bound $\norm{F}_\infty \ll k^{5/4 +\epsilon}$ (see \cite[Remark~2.5]{das-seng}) by adapting the method in \cite{blo} and using the non-trivial result of Young \cite{Y} on uniform bounds for $L(1/2,f \otimes \chi_D)$ ($\chi_D$ being the quadratic character of $\Q(\sqrt{D})$ and $f \in S_{2k-2}$, a Hecke eigenform). However, we bring to the reader's attention that \textit{Young's result is technically not enough for this purpose as it is only for odd square-free $D$ at the moment\footnote{There could be a chance of removing this assumption—from private communication with M. Young, but certainly is far from anything obvious.}} -- so strictly speaking, the above result was conditional. Interestingly, the bound $\norm{F}_\infty \ll k^{5/4 +\epsilon}$ will also follow from one of our main results on the size of the Bergman kernel for $\skk$ given below, without the use of $L$-functions.

Let $\skkunit$ be the set of $L^2$-normalised basis of Hecke eigenforms in $\skk$ which are obtained as lifts of \textsl{normalised} Hecke eigenforms in $S_{2k-2}$. This convention about lifts is understood throughout the paper. It follows from \lemref{ortholem} that $\skkunit$ is orthonormal. Put
\begin{equation} \label{skbk-def}
    \saitobk(Z) := \bksk; \q \sk := \sup \nolimits_{Z \in \htwo} \saitobk(Z).
\end{equation}
We refer to the quantity $\sk$ as the `size' or $L^\infty$-mass of $\mrm{SK}_k$.
We can now state one of the main results of the paper, which gives the correct size of $\sk$.

\begin{thm}\label{mainthm1}
Let $k$ be even\footnote{For odd weights we get non-holomorphic lifts, cf. \cite{arakawa-odd}. It would be interesting to look into these objects.}. Then for all such $k$ large enough,
\begin{equation} \label{mainthm1-dis}
  k^{5/2} \ll  \sk \ll_\epsilon k^{5/2 +\epsilon}.
\end{equation}
\end{thm}
We believe that the $\epsilon$ in the upper bound can be dropped, but can not see immediately how to do that. Dropping all but one term, one recovers the bound $\norm{F}_\infty \ll k^{5/4 +\epsilon}$ mentioned above. We give two proofs of \thmref{mainthm1}, neither of which use Young's result. In fact its usage gives worse results by our method for the second moment, i.e., for $\sk$ (since Young's result is on the 3rd moment) but gives good bounds for higher moments; this will be discussed later, cf. \eqref{f6}. A corollary of the upper bound in \thmref{mainthm1} can be stated as follows.
\begin{cor} \label{corsup}
Let $k$ be even. Then there exist $F \in \skkunit$ such that $\norm{F}_\infty \ll_\epsilon k^{3/4+\epsilon}$. Moreover, there exists an $L^2$-normalised $G \in \skk$ (which may not be a Hecke eigenform) such that $\norm{G}_\infty \gg k^{5/4}$.
\end{cor}
For the assertion about $G$, see the construction in \cite[Corollary~1.5]{sd-hk} which we do not repeat here.

We mention here the related recent result from \cite{sd-hk} which proves, among other things, that 
\begin{equation} \label{sups2k}
k^{9/2} \ll  \sup \nolimits_{Z \in \htwo} \bkdegtwo \ll k^{9/2+\epsilon},
\end{equation}
from which one gets the existence $F \text{ or } G \in S^2_k$ satisfying the above corollary (and evidence for \eqref{supnormconj}) -- but one can not however pin $F$ or $G$ down to $\skk$. But this is desirable, given that the SK-lifts contribute much smaller amount to the mass coming from $S^2_k$ as follows from \thmref{mainthm1}. So \corref{corsup} is new, and interesting because $F$ satisfies \eqref{supnormconj}. In contrast, from \cite{blo} we know the existence of an $H \in \skkunit$ with $\norm{H}_\infty \gg_\epsilon k^{3/4-\epsilon}$, which suggests that \eqref{supnormconj} is perhaps specific for Hecke eigenforms. Lastly, let us note that \thmref{mainthm1} is consistent with \eqref{supnormconj}, and along with \eqref{sups2k} provides (stronger) evidence towards \eqref{supnormconj}.

Our knowledge for the sup-norm of an individual $F  \in S^2_k$ is unfortunately very limited. Best is what one can obtain from the Bergman kernel, viz. $\norm{F}_\infty \ll_\epsilon k^{9/4+\epsilon}$ (cf. \eqref{sups2k}). Ideally, one would hope that for SK-lifts one could do better. For this one would require robust bounds on the geometric side of $\saitobk(Z)$.
However, as far as we could see, there is no `immediately useful' geometric side of the Bergman kernel for this space (cf. section~\ref{bkgeo}), as the definition of this space is via the Fourier expansion. To us, it seems that the best avenue for amelioration over the bound $k^{5/4+\epsilon}$ would be via an improvement of Young's bound \cite{Y} (\textit{for all fundamental discriminants}) in the twist aspect, this is demonstrated in section~\ref{subconvex}. We understand that this is a very hard problem in itself.

More precisely, if we define higher moments by $\mathscr F_{2r} := \sup_{Z \in \htwo}\sumn_{F \in \mc B_k^*} \det(Y)^{kr} |F(Z)|^{2r}$ (for $\mc B_k^*$ see \eqref{skbk-def}), and assume that (i.e. go beyond the Weyl-type bound in the twist aspect) for all fundamental discriminants $D$:
\begin{equation} \label{Ydemand}
 \sumn_{f\in  B_k} L(1/2, f \otimes \chi_D)^3\ll_\epsilon k^{1 +\epsilon} D^{1- \delta +\epsilon}
\end{equation} 
($B_k$ being the normalised Hecke basis) for some absolute positive constant $\delta$, it follows (see the end of in section~\ref{subconvex}) that
\begin{equation} \label{f6}
    \mathscr F_{6} \ll_\epsilon k^{15/2 - 2\delta +\epsilon}
\end{equation}
for all $\epsilon >0$. From this we would immediately get a saving over the exponent $5/4$ alluded to in the above. Note that one has $\mathscr F_{2r} \ll k^{5r/2}$ as an application of \thmref{mainthm1}.
Of course, the exponent of $k$ in \eqref{Ydemand} can not be improved (except may be the $\epsilon$) since by H\"older's inequality, we see that (see e.g., \cite[Lemma 3]{blo} for the first assertion below)
\begin{align} \label{holder-k}
  k \asymp_D  \sumn_{f \in  B_{2k-2}} L(k-1, f\otimes \chi_D) & \ll_\epsilon (\sumn_f 1)^{2/3} (\sumn_f L(k-1, f\otimes \chi_D)^{3})^{1/3} ,
\end{align}
from which the lower bound $k$ follows for the cubic moment in \eqref{Ydemand}. We call the proof of \eqref{f6} as the `third' method of proof of \thmref{mainthm1} which uses bounds on central $L$-values (with $\delta=0$), even though it is conditional on assuming \eqref{Ydemand} for all $D$.

We now say something about the `first' and `second' methods of proofs of \thmref{mainthm1}. To discuss the first method, we are led to the second main topic of this paper: the investigation of the sup-norm problem in the context of Jacobi forms. Our motivation here is two-fold – for one, this seems to be the first investigation of the problem for \textit{non-reductive} Lie groups viz. the Jacobi groups, and two, the results obtained thus have application to solving the same question for SK lifts. 

Let us recall some basic notation on Jacobi forms of matrix (or lattice) index. Let  $\Lambda_g^+$ be the set of all $g\times g$ symmetric, positive--definite, half-integral matrices  (i.e., $T \in \Lambda^+_g$ means $T=(t_{ij})$ with $ 2t_{ij},\;t_{ii}\in \z$) and $M_{g,n}(\mbb C)$ denote the set of $g \times n$ complex matrices. For any $S\in \Lambda_g^+$,  $J_{k, S; (n,g)}$ (resp. $J^{cusp}_{k, S; (n,g)}$) denotes the space of holomorphic Jacobi forms (resp. Jacobi cusp forms) of weight $k$ and index $S$ on $\hn \times M_{g,n}(\mbb C)$. Let $\mc B_{k,S;(n,g)}$ denote an orthonormal basis for $J^{cusp}_{k, S; (n,g)}$ and put (for $\tau (:=u+iv) \in \hn, z (:=x+iy) \in M_{g,n}(\mbb C)$)
\begin{align}
\mathbb{B}_{k,S;(n,g)}(\tau,z)&:=\sum \nolimits_{\phi\in\mathcal{B}_{k,S;(n,g)}} (\det v)^{k} e^{-4\pi \mrm{tr}( S v^{-1}[y] )} |\phi(\tau,z)|^2;\label{mBK} \\
\sup(J^{cusp}_{k, S; (n,g)})&:= \sup \nolimits_{(\tau, z)\in \hn \times M_{g,n}(\mbb C)}\mathbb{B}_{k,S;(n,g)}(\tau,z). \label{bkS2}
\end{align} 
It is easily checked that for any $u \in \GL{n}{\z}$ one has 
\begin{equation}
    \mathbb{B}_{k,S;(n,g)}(\tau,z) = \mathbb{B}_{k,S[u];(n,g)}(\tau, (u^t)^{-1} z),
\end{equation}
which implies that the sup-norms as defined in \eqref{bkS2} corresponding to $S$ and $S[u]$ are the same. Henceforth, we thus assume $2S$ to be reduced. The invariants \eqref{mBK} and their maximum values in \eqref{bkS2} are the other main protagonists of this work; and we propose a precise conjecture about their sizes.

\begin{conjecture} \label{jacobi-conj}
\begin{align}
 \sup(J^{cusp}_{k, S; (n,g)}) \asymp k^{\frac{3n(n+1)}{4} + \frac{gn}{2}} \det(2S)^{\frac{n}{2}}, \q \q \text{as } k \cdot \det(2S) \to \infty;
\end{align}

with the implied constant depending only on $g,n$. Practically the above conjecture should be thought of separately in the $k$ or $S$ aspect, fixing the other component.
\end{conjecture}
In fact, we go a step further and propose a conjecture about the sup-norm of an individual Jacobi form of scalar index. 

\begin{conjecture} \label{scalar-conj}
For an $L^2$-normalised Jacobi form $\phi \in J^{cusp}_{k,m, (n,1)}$, one has $\norm{\phi}_\infty \ll k^{(n^2+3n)/8} m^{-n/2}$ as $k \cdot m \to \infty$.
\end{conjecture}
One should note that \conjectureRef{scalar-conj} and \conjectureRef{jacobi-conj} are compatible with each other if we believe in Klingen's conjecture that $\dim (J^{cusp}_{k,m, (n,1)} ) \ll k^{n(n+1)/2} m^n$ (cf. \cite[Remark~p.~683]{klingen-jacobi}). One can propose a conjecture above for any $g \ge 1$, provided a good asymptotic formula for the dimension is known. This seems not available at the moment. \conjectureRef{scalar-conj} is expected especially for Jacobi Hecke eigenforms. 

To understand the lower bound, one natural way would be to understand the size of $\sum_\phi |c_\phi(l,r)|^2$, where $\phi $ runs over an orthonormal basis of $J^{cusp}_{k, S; (n,g)}$ as $k \to \infty$.
This is unfortunately not available quantitatively as of now when $n \ge 3$ (not even for SMF), but one can prove a qualitative result (without an error term) as a substitute. Using the interpretation via the Fourier coefficients of Poincar\'e series, we prove (slightly more generally) that

\begin{equation} \label{asymp-ortho-intro}
 \lim \nolimits_{k \to \infty}   C(l',r';P_{k,S}^{l,r}) = \delta_S(l,r; \, l',r'),
\end{equation}

where $C(l',r';P_{k,S}^{l,r})$ denotes the $(l',r')$-th Fourier coefficient of the Jacobi Poincar\'e series $P_{k,S}^{l,r}$ and $\delta_S(..)$ is a certain count of automorphisms, see section~\ref{asymp-ortho}. Note that $C(l,r;P_{k,S}^{l,r})$ is proportional to $\sum_\phi |c_\phi(l,r)|^2$, where $c_\phi(l,r)$ denotes the $(l,r)$-th Fourier coefficient of the Jacobi form $\phi$. For SMF of degree $n$ and even weights $k$, this was proved in \cite{kowalski2011note}; our proof is quite different from that, and has a simpler yet strong term-wise decay of the terms of Poincar\'e series, see \lemref{det>1} and subsequent arguments. Using all of these, we show in Proposition \ref{prop:low} that the lower bound indeed holds in \conjectureRef{jacobi-conj}.

When $n=1$, we prove the following result concerning the size of  $J^{cusp}_{k,S,g} $ where $J^{cusp}_{k,S,g} := J^{cusp}_{k,S,(1,g)}$. See \propref{prop:low} and \eqref{n1lbd}, which shows further that the lower bound below actually holds for all $n,S$ as $k \to \infty$.

\begin{thm} \label{mainthm2}
Let $k$ be even. Then for any $\epsilon>0$ there exist $k_0 \ge 1$ depending only on $g$ such that for all $k \ge k_0$,

\begin{align} \label{kS}
  k^{\frac{3}{2} + \frac{g}{2}} \det(2S)^{\frac{1}{2}} \ll  \sup(J^{cusp}_{k,S,g}) \ll_\epsilon k^{\frac{g+3}{2}+\epsilon} \det(2S)^{\frac{g+3}{2}+\epsilon}.
\end{align}
\end{thm}
One immediate application of \conjectureRef{jacobi-conj} or any result like \thmref{mainthm2} is towards estimating the dimension of the ambient spaces of automorphic forms (see e.g., \cite[Corollary~1]{zhou2010dimension} in our context). In the case of Jacobi forms, one gets very easily a polynomial bound for the dimension by using the Conjecture \ref{jacobi-conj}. Of course, good point-wise bounds for the BK are more useful--see e.g., \cite{sd-hk}.

Comparing with \conjectureRef{jacobi-conj} we see that in the $k$ aspect, the above theorem proves this conjecture. However, in the index aspect it is quite far away from the expectation. But still \eqref{kS} is not void in the $S$-aspect if we consider even unimodular lattices $2S$ (for $8|g$).
We note that the index has some similarities with the level of a modular form, in view of the Hecke equivariant isomorphism between $J_{k,m}$ and a certain subspace of $ M_{2k-2}(m)$. And then the above issue is not very surprising, as superior `hybrid' (weight along with level) aspect results are very rare on the sup-norm problem.

When $n=1,g=1$, we are in the setting of classical Jacobi forms $J_{k,m}$.
Then the `first' method of proof of \thmref{mainthm1} uses (in a small region) bounds for the geometric side of the Bergman kernel for $J^{cusp}_{k,m}$.
Moreover, for index $m=l^2$ ($l \ge 1$), we point out to the reader a peculiar situation: the contribution of `index-old' Jacobi forms $U_l (J^{cusp}_{k,1})$ inside $J^{cusp}_{k,l^2}$ (here $U_l(\phi) = \phi(\tau,l z)$) is of the same order as that of the space $J^{cusp}_{k,1}$(as per \conjectureRef{jacobi-conj}), see section~\ref{old-strange1}. The contribution of another index-old part, viz. $V_p(J^{cusp}_{k,1})$ ($p $ prime) inside $\jkp$ is shown to be at most $O(k^2p)$, which is better as compared to what one gets from \thmref{mainthm2}. We also show that the contribution of the index $p$-new space in $\jkp$ is at least as big as the conjectured bound \eqref{jacobi-conj} for the full space. These are summarized in \propref{old-prop} and \propref{new-prop}. For the treatment of newspace, we depend on an asymptotic result on the average of central $L$-values via Waldspurger/Baruch-Mao's (\cite{waldspurger1980correspondance}, \cite{waldspurger1981coefficients} and \cite{bar-mao}) formula. That the $k$ aspect size is the same for all the above spaces is expected, as all of them grow linearly in $k$, it is the index aspect that needs attention.

Namely, in section~\ref{l1/2} we prove an asymptotic formula for the first moment of the family of central $L$-values $L(1/2, f \otimes \chi_D)/L(1, \mrm{sym}^2 f)$ as $f$ varies in the space of newforms of level $p$ and weight $2k-2$, such that $k,p \to \infty$ with $D$ growing at most as polynomially with $k,p$. For the full level and fixed twist $\chi_D$, this was worked out in \cite[Lemma~3]{blo}. The level aspect is more delicate. Of course, this result should be of independent interest as well. 

For $f \in S^{new}_{2k-2}(p)$, put $A_p:=1+\frac{p}{p+1}(1+(-1)^{k-1}\sgn D)C(p)$ and $B_p:= 2+2C(p)$ where $C(p)$ is given as:
\begin{align} \label{cp}
C(p):= \frac{-1}{p+1}\sumn_{t\ge 0} \frac{p^t}{(p+1)^{2t}}\sumn_{d|p^t}c_{p^t}(d)^2 .
\end{align}
Here $c_\ell (d)$ are the Tchebyshev coefficients (see \cite{petrow2019generalized}) which occur in the expression of $\lambda_f(p^n)$ as a polynomial in $\lambda_f(p)$. Using the estimates for Tchebyshev's coefficients from \cite[Corollary 2]{petrow2019generalized} with $Y=1/2$ (loc. cit.), it is easy to see that $C(p)\ll 1/p$. Our main result is given below.
\begin{thm} \label{blo-p}
Let  $B_{2k-2}^*(p)$ denote the set of newforms for the space $S_{2k-2}^{new}(p)$. Then for any fundamental discriminant $D$, we have
\begin{enumerate}
    \item when $(p,D)=1$,
    \begin{equation} \label{blo1}
    \sumn_{f\in B_{2k-2} ^*(p)} \frac{L(1/2, f\otimes \chi_D)}{ L(1, \mrm{sym}^2 f) } = \frac{A_p(2k-2) p}{ 2 \pi^2} + O_\epsilon(D^{7/8+\epsilon} k^{19/24+\epsilon} p^{9/16+\epsilon}); 
\end{equation}
\item when $p|D$ and $k$ and $D$ be such that $(-1)^{k-1} \sgn(D)=1$, then 
\begin{equation} \label{blo2}
    \sumn_{f\in B_{2k-2} ^*(p)} \frac{L(1/2, f\otimes \chi_D)}{ L(1, \mrm{sym}^2 f) } = \frac{B_p(2k-2) p}{2 \pi^2} + O_\epsilon(D^{7/8+\epsilon} k^{-1/12} p^{-1/4+\epsilon}).
\end{equation}
\end{enumerate}
\end{thm}

When $(p,D)=1$, the above asymptotic is valid for $D\ll k^{10/21-\epsilon} p^{1/2-\epsilon}$ and when $p|D$, it is valid for $D\ll k^{26/21-\epsilon} p^{10/7-\epsilon}$.
Let us note that \thmref{blo-p} does not follow from \cite{Y} or \cite{petrow2019generalized}, partly because none is hybrid with respect to all the parameters in play: weight, level and the twist.

We must mention here about two more parts which play crucial role while discussing the size of the newspace: that of subsections~\ref{victory} and ~\ref{barmao}. Namely, in subsection~\ref{victory} we show how the choice of $D=-4p$ leads to the lower bound $k^2 p^{1/2}$ (which is the size of $\jkp$, if we believe in \conjectureRef{jacobi-conj}) for the size of the newspace. But for this we have to appeal to the classical newform theory for Jacobi forms, e.g., as considered in \cite{man-ram}. However, we must first reconcile the "two" notions of newspaces presented in loc. cit. and then derive the requisite Waldspurger type formula in the context of Jacobi forms when $p|D$ using Baruch-Mao's result (\cite{bar-mao}). Along the way, we write down the main results about Eichler-Zagier correspondence in a way which we believe would be useful in the future. This is the content of subsection~\ref{barmao}.

The `second' method is presented in section~\ref{method-2}. It does not use the geometric side of the Bergman kernel for Jacobi forms of index bigger than $1$, but uses full knowledge of Jacobi Poincar\'e series of index $1$. One of our motivations behind exploring more than one method is that one of them could have the potential to implement (some variant) the amplification method to improve the individual sup-norm bound for an SK lift. It is not immediately clear how this might work out. This is partly explained by the fact that the standard definition of the SK lifts is via their Fourier expansion—from which the geometric side of the Bergman kernel is hard to study, see section~\ref{bkgeo}.
This led us to make the remark above that the third method, with an improvement in sub-convexity results could be one of the best avenues in this regard.

In the last section, we investigate the size of `pullbacks' (to the diagonal $\h \times \h$) of an SMF. When the SMF is an SK lift, we have a deep result of Ichino expressing its support in $S_k \otimes S_k$ in terms of central $L$-values. In fact, this was the basis of the works \cite{liu2014growth} and \cite{BKY}, where the $L^2$-mass of such pullbacks were investigated. We are interested in the $L^\infty$-mass of the space consisting of the pullbacks (in the spirit of this paper) and show that the size of such pullbacks is either $k^3$ or $k^2$ depending on how one normalises the $L^2$-norm; it is $k^3$ if we sum over all the pullbacks (to the diagonal, see 1st paragraph of the introduction) $F^\circ$ of $F$ with $\norm{F}=1$, and it is $k^2$ if we sum the same with $\norm{F^\circ}=1$, provided it is non-zero. This suggests that the bound $\norm{F^\circ}_\infty \ll k^{1/2}$ might be expected for an $F \in \skk$, at least when $F$ is a Hecke eigenform. Also $\norm{F } :=\norm{F}_2$ throughout this paper.
\begin{conjecture}
Let $F \in \skk$ be a Hecke eigenform such that $F^\circ \ne 0$ and $\norm{F^\circ}_2=1$. Then $\norm{F^\circ}_\infty \ll k^{1/2 + o(1)}$.
\end{conjecture}
A curious corollary of our calculations is about the linear independence of the square-root of the central $L$-values $L(1/2, \mrm{sym}^2g \times f)$ as $g,f$ vary in $S_k,S_{2k-2}$ respectively, see \corref{maxrank}.

Moreover, in an Appendix to section~\ref{pullback-sec} (see section~\ref{appendix1}), we determine the size of pullbacks of the full space $S^2_k$ and find the answer to be $k^3$ -- which says that the pullbacks fill up all the mass in $S_k \otimes S_k$ (cf. \thmref{witt-thm}). We point out that \textit{the main idea here is the realization of space of pullbacks of SK lifts (or even for the full space $S^2_k$) as a certain simple subspace of $S_k \otimes S_k$, which allows us to find their sizes}; and also allows one to talk about the `preliminary' bounds for the size of these pullbacks. Merely using the Fourier expansions of the pullbacks (which involves sums of square-roots of central $L$-values or averages of Fourier coefficients of half-integral weight cusp forms, for which we do not have a Petersson formula or some other substitute) gives worse results, see \lemref{SK0tensor} in this regard.

We would like to remind the reader that the average results in this paper have their own value: note that merely squaring and adding the best unconditional sup-norm bounds do not usually give the expected average result. The goal of this paper is not to provide individual sup-norm bounds, but to study the BK. Individual sup-norm bounds will be taken up separately.

We end the introduction by mentioning several questions and remarks which can be pursued for further research on the topics considered in this paper.

\begin{inparaenum}[(i)]
\item
It is not clear to us how to improve the index aspect of \thmref{mainthm2}. One might use the fact that $J_{k,m} \cong V_{k,m} \otimes Th_m$, where $V_{k,m}$ is the space of certain vector-valued $\SL{2}{\z}$ modular forms, and $Th_m$ is the (space spanned by the) congruent theta tuple mod $m$. But we can't perceive this to be radically different from what has been done in this paper.

\item
\conjectureRef{jacobi-conj} perhaps could be proved for $n=2$ with efforts, by using the available results for SMF of degree $2$ (cf. \cite{sd-hk}). For higher degrees, one might get some result, but probably even the $k$ aspect will be weak.

\item
As the reader might have noticed, the `trivial' (but not easy) bound for an SK lift of a Hecke eigenform, stands bounded above by $k^{5/4}$. One would hope to improve this bound -- the best possible bound in this direction is $k^{3/4}$ (see the introduction in \cite{sd-hk}). Perhaps the `theta-lifting' perspective plays a role. But this will be specific to degree $2$. The same comment goes for pullbacks of Hecke eigenforms (cf. \corref{skocor}) where the bound which we get is $k ^{1+\epsilon}$ and the expected bound is $k^{1/2}$.

\item
One might be able to remove the presence of $\epsilon$ in the statement of \thmref{mainthm1}.

\item
Perhaps successful results on higher moments of twisted central $L$-values will also help improve our bounds, see section~\ref{subconvex}. One might try to use an asymptotic formula or a good upper bound for the higher moments in the setting of \thmref{blo-p}. See for instance \cite[(3.9)]{blo-har} for the second moment, in the level and twist aspects. We have already discussed \cite{Y} to this effect.

\item
It is desirable to understand the contributions of the index-old and index-new forms (inside $J_{k,m}$) in a more refined and accurate way -- also for higher degrees and matrix indices.

\item
It is natural to consider the level aspect version of the results in this paper. These are under consideration by the authors.
\end{inparaenum}

\subsection*{Acknowledgements} 
The authors thank the referee for comments on the article which helped in improving the presentation.

S.D. thanks IISc. Bangalore, UGC Centre for Advanced Studies, DST India and the Alexander von Humboldt Foundation for financial support. Parts of the paper were written while S.D. was enjoying the hospitality at the University of Bonn.

P.A. was a Postdoctoral Fellow at the Harish--Chandra Research Institute (HRI), Prayagraj, India during the preparation of the article and would like to thank HRI for funding and providing excellent facilities.
\section{Notation and setting} \label{prelim}

In this paper, we will mostly use standard notation, some of which are collected below, and the rest will be introduced as and when it is necessary. For standard facts about Siegel modular forms and Jacobi forms, we refer the reader to \cite{Fr}, \cite{klingen1990introductory}, \cite{ez}, \cite{ziegler1989jacobi}.

\begin{inparaenum}[(1)]
\item We use the standard conventions in analytic number theory. We note $A \ll B$ and $A=O(B)$ are the Vinogradov and Landau notations, respectively. By, $A\asymp B$ we mean that there exists a constant $c\ge 1$ such that $B/c\le A\le cB$. Any subscripts under them (e.g., $A\ll_n B$) indicates the dependence of any implicit constants on those parameters. Throughout, $\epsilon$ will denote a small positive number, which can vary from one line to another.

\item 
Let $\z, \Q, \mbb R$, and $\complex$ denote the integers,
rationals, reals and complex numbers respectively. For a commutative ring $R$ with unit, $M_{n,m}(R)$ and $M_n( R)$ denote the set of $n \times m$  and $n\times n$ matrices over $R$, respectively. $\mrm{GL}_n (R)$ denotes the group of invertible elements in $M_n(R)$. $\mrm{Sym}_n(R)$ denotes the set of all $n\times n$ symmetric matrices over $R$. $\mrm {Sp}_n(R)$ denotes the symplectic group of degree $n$ over $R$. $\Lambda_n$ (resp. $\Lambda_n^+$) denotes the set of all $n\times n$ symmetric, positive semi--definite (resp. positive--definite), half--integral matrices  (that is $T=(t_{ij})$ with  $ 2t_{ij},\;t_{ii}\in \z$). We will denote the transpose of $A$ by $\tp{A}$. For matrices $A$ and $B$ of appropriate size, we write $A[B]:=BAB^t$. Moreover, $1_n$ and $0_n$ will be the identity and zero matrices, respectively.

\item For $A=\smat{a}{b}{c}{d}\in \GSp{n}{\R}$ (the group of symplectic similitudes with positive similitude factor), we denote by $A^\uparrow$ to be the image of $A$ under the diagonal embedding of $\GSp{n}{\R} \hookrightarrow \GSp{n+m}{\R}$ given by $A\mapsto \smat{\mat{a}{0}{0}{\det(A)}}{\mat{b}{0}{0}{0_m}}{\mat{c}{0}{0}{0_m}}{\mat{d}{0}{0}{1_m}}$.

\item Let $\hn:=\{Z\in M_n(\mathbb C)\, :\, Z=Z^t, \im Z>0\}$ be the Siegel's upper half plane of degree $n$. The symplectic group $\Sp{n}{\mbb R}$ acts on $\hn$ by $M\lan Z\ran =(AZ+B)(CZ+D)^{-1}$, for $M=\smat{A}{B}{C}{D}\in \Sp{n}{\mbb R}$ and $Z\in \hn.$ Let $\mc F_n$ denote the Siegel's fundamental domain for the action of $\Sp{n}{\z}$ on $\hn$.
\end{inparaenum}
\subsection{Jacobi forms of type \texorpdfstring{$(n,g)$}{(n,g)}} Let $R$ be $\mbb Z$, $\mbb Q$ or $\mbb R$, then the Heisenberg group $H_R^{n,g}$ is given by
\begin{equation}
    H_R^{n,g}=\{[(\lambda,\mu),\kappa]\,:\, \lambda,\mu\in M_{g,n}(R), \kappa\in M_g(R), \kappa+\mu\lambda^t\in\mrm{Sym}_g(R)\}
\end{equation}
with the composition law
\begin{equation}
    [(\lambda_1,\mu_1),\kappa_1] \cdot [(\lambda_2,\mu_2),\kappa_2]:=[(\lambda_1+\lambda_2, \mu_1+\mu_2), \kappa_1+\kappa_2+\lambda_1\mu_2^t-\mu_1\lambda_2^t].
\end{equation}
The symplectic group $ \Sp{n}{ R}$ acts on $H_R^{n,g}$ from right by
\begin{equation}
    [(\lambda,\mu),\kappa]M:= [(\lambda,\mu)M,\kappa]; \q M\in \mrm{Sp}_n(R),\, [(\lambda,\mu),\kappa]\in H_R^{n,g}. 
\end{equation}
The Jacobi group of type $(n,g)$ over $R$ is defined as
\begin{equation}
    G_R^{n,g}:=\Sp{n}{R} \ltimes H_R^{n,g}
\end{equation}
with the group law given by
\begin{equation}
    (M_1,X_1)(M_2,X_2):=(M_1M_2, (X_1M_2)\cdot X_2), \text{ for } M_1,M_2\in \Sp{n}{R} \text{ and } X_1,X_2\in H_R^{n,g}.
\end{equation}
The Jacobi group $G_R^{n,g}$ acts on $\hn\times M_{g,n}(\mbb C)$ by
\begin{equation}
    (M,[(\lambda,\mu),\kappa])\cdot (\tau,z):=(M\lan \tau\ran,(z+\lambda\tau+\mu)(c\tau+d)^{-1}),
\end{equation}
where $M=\smat{a}{b}{c}{d}$ with $a,b,c,d\in M_n(R)$ and $M\lan \tau\ran=(a\tau+b)(c\tau+d)^{-1}$.

For any $k\in\z$ and $S\in \Lambda_g^+$, $G_{\mbb R}^{n,g}$ acts on functions on $\hn\times M_{g,n}(\mbb C)$ as follows:
\begin{equation}\label{Jacobiaction}
    (\phi|_{k,S} \gamma)(\tau,z):=J_{k,S}(\gamma, (\tau,z))\phi(\gamma(\tau,z)),
\end{equation}
where $\gamma=(M,[(\lambda,\mu),\kappa])$ and the automorphy factor $J_{k,S}$ is given by
\begin{equation}
    J_{k,S}(\gamma, (\tau,z)):=\frac{ e(\tr (S(-(z+\lambda\tau+\mu)(c\tau+d)^{-1}c(z+\lambda\tau+\mu)^t+ \tau[\lambda]+2\lambda z^t+\kappa+\mu\lambda^t))}{\det(c\tau+d)^{k}}.
\end{equation}
The Jacobi group of type $(n,g)$ is given by
\begin{equation}
    \Gamma^J_{n,g}:=\{(M,X)\in G_{\mbb Z}^{n,g}\, : \, M\in \mrm{Sp}_n(\mbb Z)\}.
\end{equation}
A holomorphic function $\phi$ on $\hn\times M_{g,n}(\mbb C)$ is called a Jacobi form of weight $k$ and index $S$ if
\begin{enumerate}
    \item $\phi|_{k,S}\gamma = \phi \q \forall \gamma\in \Gamma^J_{n,g}$.
    \item $\phi$ is bounded at the cusps of $\mrm{Sp}_n(\mbb Z)\backslash \hn$ (automatic for $n\ge 2$, due to the K\"{o}cher principle).
\end{enumerate}
Any such $\phi$ as a Fourier expansion given by
\begin{equation}
    \phi(\tau,z)=\underset{4T-S^{-1}[R]\ge 0}{\sumn_{T\in \Lambda_n}\sumn_{R\in M_{n,g}(\z)}}C_\phi(T,R) e(\tr (T\tau+Rz)).
\end{equation}
If the Fourier coefficients survive only for $4T-S^{-1}[R]>0$, then $\phi$ is called a Jacobi cusp form. $J_{k, S; (n,g)}$ (resp. $J^{cusp}_{k, S; (n,g)}$) denotes the space of holomorphic Jacobi forms (resp. Jacobi cusp forms) of weight $k$ and index $S$ on $\hn \times M_{g,n}(\mbb C)$.

The  space of Jacobi cusp forms, $J^{cusp}_{k, S; (n,g)}$ is equipped with an inner product and is given by
\begin{equation}
    \lan \phi,\psi\ran :=\int_{\Gamma^J_{n,g}\backslash \hn \times M_{g,n}(\mbb C)} \phi(\tau,z)\overline{\psi(\tau,z)} (\det v)^{k} \exp(-2\pi \tr (Sv^{-1}[y])) \frac{du\,dv\,dx\,dy}{(\det(v))^{n+g+1}},
\end{equation}
for $\phi,\psi\in J^{cusp}_{k, S; (n,g)}$ and $\tau=u+iv$, $z=x+iy$.

At various places we need the following operators $V_m$ defined on Jacobi forms of degree $1$ and index $1$ (even though they can be defined for any index). For $\phi \in \jk$ define (cf. \cite[\S~4 (2)]{ez})
\begin{align} \label{vmdef}
    V_m(\phi)(\tau,z)= m^{k-1} \sumn_{\gamma } (c \tau+d)^{-k} e(m c z^2/(c \tau+d)^2) \phi \big(\frac{a \tau +b}{c \tau+d}, \frac{m z}{c \tau+d} \big),
\end{align}
where $\gamma =\smat{a}{b}{c}{d} $ runs over a set of representatives $\Gamma_1 \backslash M_{2,m}(\z)$. Here $M_{2,m}(\z)$ denotes the set of size $2$ integral matrices with determinant $m$. Then $V_m $ maps $\jk$ to $\jkm$.
By choosing upper triangular representatives, one arrives at the following Fourier expansion of $V_m(\phi)$ (cf. \cite[Theorem~~4.2 (7)]{ez}):
\begin{align} \label{vmfe}
     V_m(\phi)(\tau,z)= \sumn_{n,r} \big( \sumn_{a |(n,r,m)} a^{k-1} c_{\phi} (nm/a^2, r/a) \big) \, e(n \tau+rz).
\end{align}
\subsection{SK lifts}
We now recall the formula for the Fourier coefficients of $F \in \skk$ in terms of those of the lifted $\phi \in \jk$. Throughout the remainder of this paper we shall put 
\begin{equation}
    F(Z) = \sumn_{m \ge 1} \phi_{m,F}(\tau,z) e(m \tau') , 
\end{equation}
where we write $Z=\smat{\tau}{z}{z^t}{\tau'}$ and $\phi_{m,F} = V_m(\phi_{1,F})$ with $V_m$ as defined in \eqref{vmdef}. By the results of \cite{maass1979spezialschar}, \cite{ez}, when $k$ is even $F \in \skk$ and this correspondence $\phi \mapsto F$ is a Hecke equivariant isomorphism from $\jk \to \skk$. $F$ is called the SK lift of $\phi$.
For any $\mc D \equiv 0,-1 \bmod{4}$, $\mc D>0$, $c_{\phi,F}(\mc D) := c_{\phi_{1,F}}(\mc D)$. Then for $T = \smat{n}{r/2}{r/2}{m}$ and $\mf c(T) = (n,r,m)$ -- the content of $T$, the following relation holds:
\begin{align} \label{fcreln}
    a_F(T) =  \sumn_{a|\mf c(T)} a^{k-1}  c_{\phi,F}\left( D /a^2 \right).
\end{align}
where $D= \det(2T)$.
We also often realize the SK lift to be from $S_{2k-2}$ via the Shintani and Eichler-Zagier maps, all of which are Hecke equivariant isomorphisms. We thus parametrize $F \in \skk$ as $F=F_f$ for a unique $f \in S_{2k-2}$. In particular, $F_f$ is a Hecke eigenform if and only if $f$ is.

\subsection{Bergman Kernel for Jacobi forms}
 Let $\mathcal{B}_{k,S;(n,g)}$ be an orthonormal basis for $J_{k,S;(n,g)}^{cusp}$. Then the Bergman kernel (a reproducing kernel for the space of Jacobi cusp forms $J_{k,S,(n,g)}^{cusp}$) is given by,
\begin{equation}\label{naiveBK}
    B_{k,S;(n,g)}(\tau,z;\tau_0,z_0)=\sumn_{\phi\in\mathcal{B}_{k,S;(n,g)}}\phi(\tau,z)\overline{\phi(\tau_0,z_0}).
\end{equation}
It can also be written explicitly as given below. Let $h_{k,S,(n,g)}$ be a function on $(\mathbb{H}_n\times M_{g,n}(\mbb C))^2$ defined as (see e.g., \cite{zhou2010dimension}, \cite{sko-zag}):
\begin{equation}
    h_{k,S,(n,g)}(\tau,z;\tau_0,z_0):=\det(\tau-\bar{\tau}_0)^{-k}e\left(-(\tau-\bar{\tau}_0)^{-1}S[z-\overline{z}_0]\right).
\end{equation}
Let $k> g+2n$. For $(\tau,z),(\tau_0,z_0)\in(\hn\times M_{g,n}(\mbb C))^2$ the Bergman Kernel for the space of Jacobi forms can be written as (see \cite[Proposition 2]{zhou2010dimension})
\begin{equation} \label{bkdef}
    B_{k,S;(n,g)}(\tau,z;\tau_0,z_0):=\lambda_{k,S;(n,g)}\sumn_{\xi\in\Gamma^{J}_{(n,g)}}(h_{k,S;(n,g)}|^{(1)}_{k,S}\xi)(\tau,z;\tau_0,z_0),
\end{equation}
where $\lambda_{k,S;(n,g)}= 2^{-n(n+3)/2}\pi^{-n(n+1)/2} (\det(2S))^n \prod_{t=0}^{n-1}\prod_{j=1}^{n-t}(k-\frac{g+t}{2}-j)$ and $|^{(1)}_{k,S}$ denotes the action of the Jacobi group $\Gamma^J_{(n,g)}$ with respect to the first pair of variables $(\tau,z)$.

\subsection{Siegel-Jacobi Poincar{\'e} series. }
For a positive definite matrix $l\in \Lambda_n^+$ and $r\in M_{n,g}(\mathbb Z)$, the $(l,r)$-th Jacobi Poincar{\'e} series of weight $k$ and index $S$ is defined as (see \cite{bouganis2020algebraicity}, we replace $(\tau,z)$ with $(Z,W)$)
\begin{equation}\label{PSdef}
    P_{k,S}^{l,r}(Z, W):=\sumn_{\gamma\in \Gamma^J_{(n,g);\infty} \backslash \Gamma^J_{n,g}} e(\mrm{tr}(l Z)) e(\mrm{tr}(r^t W))|_{k,S}\gamma.
\end{equation}

One knows that $P_{k,S}^{l,r}\in J^{cusp}_{k, S; (n,g)}$ for $k> 2n+g$ and for any $\phi\in J^{cusp}_{k, S; (n,g)}$ (see \cite[Theorem 6.4]{bouganis2020algebraicity})
\begin{equation}\label{PetPS}
    \lan \phi, P_{k,S}^{l,r}\ran= \lambda_{k,S; (n,g)}(l,r) C_\phi(l,r),
\end{equation}
where $\Gamma_n(s):=\pi^{n(n-1)/4}\prod\limits_{i=0}^{n-1}\Gamma(s-\frac{i}{2})$ and 
\begin{equation} \label{lambda-def}
    \lambda_{k,S; (n,g)}(l,r):=\frac{ \Gamma_n(k-\frac{n+g+1}{2}) (4l-S^{-1}[r])^{-k+(n+g+1)/2} }{ (\pi)^{nk-n(n+g+1)/2} (\det 2S)^{n/2}}.
\end{equation}

Let $C(l_1,r_1;P_{k,S}^{l_2,r_2})$ denote the $(l_1,r_1)$-th Fourier coefficient of $P_{k,S}^{l_2,r_2}$. Then, using \eqref{PetPS}, we get
\begin{equation}\label{PetTr}
    \sumn_{\phi\in \mathcal{B}_{k,S;(n,g)}} C_\phi(l_1,r_1)\overline{C_\phi(l_2,r_2)}=  \lambda_{k,S;(n,g)}(l_2,r_2)^{-1} C(l_1,r_1;P_{k,S}^{l_2,r_2}),
\end{equation}
which is the Petersson trace formula for the space $J^{cusp}_{k, S; (n,g)}$.

\section{Lower bound for the Bergman kernel for the space of Jacobi forms}
When $n>1$, a possible approach to obtain the best possible lower bound for the Bergman kernel for $J^{cusp}_{k, S; (n,g)}$ is via an `asymptotic orthogonality' of Fourier coefficients of cusp forms -- see below. This avoids the use of explicit knowledge of the Fourier coefficients of Siegel-(Jacobi) Poincar\'e series; something which is known to be quite difficult to deal with (cf. \cite{sd-hk}). When $n=1$, of course we can handle them better (cf. \cite{boecherer1993estimates}).

\subsection{Asymptotic orthogonality of the Fourier coefficients of Siegel--Jacobi Poincar\'e series} \label{asymp-ortho}
In this section, we prove the following asymptotic orthogonality result (cf. \cite{kowalski2011note} for Siegel cusp forms). 
\begin{prop}\label{prop:asymporth}
Let $k$ be even. Then
\begin{align}\label{asymporth}
 \lim \nolimits_{k \to \infty}   C(l',r';P_{k,S}^{l,r}) = \delta_S(l,r; \, l',r'),
\end{align}
where $\delta_S(l,r; \, l',r')=\# \{ (A,\lambda)  \in \GL{n}{\z} \times \z^{g}  \colon \smat{l}{r/2}{r^t/2}{S}[\smat{A}{0}{\lambda}{I_g}] = \smat{l'}{r'^t/2}{r'/2}{S} \}$.
\end{prop}
So from \eqref{PetTr}, one can say that the Fourier coefficients are `asymptotically orthogonal' on average over an orthonormal basis as $k \to \infty$, if we also use \eqref{lambda-def}. This result has application towards the sup-norm problem and other equidistribution problems (cf. \cite{kowalski2011note}, \cite[Lemma~4.1]{sd-hk} and below).

For any $y_0>0$, let us define the set
\begin{align} \label{compset}
    \mc F (y_0) = \{ Z= U+iV \in \hn \colon  V=y_0I_n \}.
\end{align}
In \cite{kowalski2011note}, \eqref{asymporth} was established for $n \ge 1, g=0$ via a technical lemma which states that for $y_0$ large enough, for all $Z \in \mc F(y_0)$ and all $M=\smat{a}{b}{c}{d}\in \Sp{n}{\z}$ with $c\neq 0$, one has $|\det(cZ+d)| >1$
(cf. \cite[Lemma~5]{kowalski2011note}). Further, in \cite[Remark~7]{kowalski2011note} it was speculated that the above result should hold for all $y_0>1$ based on Gottschiling's (see \cite{gottschling1959explizite}) explicit description of the finitely many boundary components of $\mc F_2$. Such description is not available for $n>3$ (see \cite[\S~5]{carine} for a partial result when $n=3$), as far as we know.

We prove the following result, which is an improvement over \cite[Lemma~5]{kowalski2011note}, in that it holds for all $n \ge 1$, all $y_0>1$ and does not require the explicit description of the boundary components of $\mc F_n$.

We write $c_M, d_M,\ldots$ etc. for the lower blocks of $M \in \Sp{n}{\z}$.
\begin{lem}\label{det>1}
Let $y_0>1$ and $Z\in\mc F(y_0)$. Then, for any $M=\smat{a_M}{b_M}{c_M}{d_M}\in \Sp{n}{\z}$ with $c_M \neq 0$,
\begin{equation}
    |\det(c_MZ+d_M)| \ge y_0.
\end{equation}
\end{lem}
\begin{proof}
Since $M$ is fixed, we omit the superscript in $c_M$ etc. Let $r=\text{rank}(c)$. Since $c\neq 0$, we have $r\ge 1$. Then, for any such co-prime symmetric pair $(c,d)$ as in the lemma, using Siegel's lemma (e.g., see \cite[Lemma 3.1]{das2015nonvanishing}) we can write $(c,d)=(\smat{c_1}{0}{0}{0} w^t, \smat{d_1}{0}{0}{1} w^{-1})$. Here $w\in\GL{n}{\z}$ and $c_1$ is a $r \times r$ matrix of rank $r$.

Next, we have 
\begin{align}
    |\det(cZ+d)| = |\det(c_1 (Z[w])_1 + d_1)| \ge |\det( (Z[w])_1 + c_{1}^{-1} d_1)| \ge \det (\im (Z[w])_1),
\end{align}
where $Z_1$ denotes the upper left corner of $Z$ of size $r$ and we have used the fact that $c_1^{-1} d_1$ is symmetric—this follows because $(c_1,d_1)$ is again a co-prime symmetric pair. For completeness, let us mention that we used the standard inequality $|\det(A+iB)| \ge \det B$, where $A,B \in \Sym{n}{\mbb R}$ and $B>0$.

Writing $Z=U+iy_0I_n$, we see that
\begin{equation}
    \im (Z[w])_1= (y_0I_n[w])_1=y_0 (w^tw)_1.
\end{equation}
Since $w$ has integer entries, it follows immediately that for any $y_0>1$, $|\det (cZ+d)|\ge y_0^r \ge y_0$.
This completes the proof of the lemma.
\end{proof}

From \cite[Prop.~6.1]{bouganis2020algebraicity}, $P_{k,S}^{l,r}$ converges absolutely and locally uniformly (uniformly on all compact sets) provided $k>n+g+1$. For us the compact set will be
\begin{align} \label{pint}
    \mc C(y_0) = \{ (Z,W) \in \hn\times M_{g,n}(\mbb C) \colon U \bmod{1}, V=y_0I_n; \, X \bmod{1}, Y=0  \};
\end{align}
where $y_0 >1$ has been chosen as in \lemref{det>1} -- so that for all $Z$ as above, one has $|\det(cZ+d)| >1$ for all $M \in  \Sp{n}{\z}$ with $c_M \neq 0$. 

We can write
\begin{align}\label{Clrint}
   C(l',r';P_{k,S}^{l,r}) = \int_{(Z,W) \in \mc C(y_0)} P_{k,S}^{l,r}(Z,W)\, e\big(- \tr (l' Z + r'^t W) \big) dZ \, dW.
\end{align}
For us, it would be sufficient to assume $l=I_n, r=0$, and $r=r', l=l'$; but we prove a somewhat general result. 

Write, $P_{k,S}^{l,r} = \sum_\gamma A_\gamma$ with $\gamma = (M, (\lambda, 0))\in \Gamma^J_{(n,g);\infty} \backslash \Gamma^J_{n,g} $. From \eqref{PSdef} and \eqref{Jacobiaction}, $A_\gamma$ is given by
\begin{align}
     \frac{e\big(- \tr (S \ww (c_MZ+d_M)^{-1} c \ww^t) \big) 
    e\big( \tr(S Z[\lambda] +2 S \lambda W^t  \big)  e \big( \tr(l M\lan Z \ran + r^t \ww(c_MZ+d_M)^{-1} \big)}{ \det(c_MZ+d_M)^{k} }, \label{trocean}
\end{align}
where we have put $\ww := W+\lambda Z$.

First, we show that for any $(Z,W) \in \mc C (y_0)$ and $\gamma$ as above with $c_M\neq 0$, $A_\gamma \to 0$ as $k \to \infty$ \textit{not depending on $S$}. To see this, let us put $\mc Z = \smat{Z}{W^t}{W}{Z'}$ and $\mc T = \smat{l}{r^t/2}{r/2}{S}$ (here $Z'$ will be chosen suitably so that $\mc Z \in \mbb H_{n+g}$). Denote by $\gamma^\uparrow$, the image of $\gamma$ under the embeddings of $\Sp{n}{\mathbb Z}$ and the Heisenberg group $H_{\mathbb Z}^{n,g}$ into $\Sp{n+g}{\z}$. Then $\gamma^\uparrow \lan \mc Z \ran  $ can be obtained from the action of $\Sp{n+g}{\z}$ on $\mathbb H_{n+g}$ (see section~\ref{prelim}~(3) and \cite{ziegler1989jacobi}). To be precise, for $\gamma = (M, (\lambda, 0))$, we have
\begin{equation}
   \gamma^\uparrow \lan \mc Z \ran =\smat{M\lan Z\ran}{M\lan Z \ran \lambda^t+a_M W^t}{\ww(c_MZ+d_M)^{-1}\q}{\ww(c_M Z+d_M)^{-1}(d_M\lambda^t-c_M W^t)+\lambda W^t+Z'}.
\end{equation}
Then a calculation shows, 
\begin{align} \label{agamma}
    A_\gamma = \det(c_MZ+d_M)^{-k} e\big(\tr (\mc T \cdot \mc \gamma^\uparrow \lan \mc Z \ran) \big) \, e \big( - \tr(SZ') \big).
\end{align}

It is clear that $|e\big(\tr (\mc T \cdot \mc \gamma^\uparrow \lan \mc Z \ran) \big)| \le 1$ and
\begin{align}\label{agammabound}
    |A_\gamma| \le |\det(cZ+d)|^{-k} \exp \big(2 \pi \tr(S Y') \big).
\end{align}
We choose $Y':=S^{-1}$. Clearly $\im (\mc Z) = \smat{V}{0}{0}{S^{-1}}$ since $Y=\im(W)=0$, and it is $>0$ since $V>0$ and $S>0$. 
Next, from \lemref{det>1}, we have
\begin{align}
    |\det(c Z+d)|^{-k} |\exp \big(2 \pi \tr(S Y') \big)| \le y_0^{-rk} e^{2 \pi g} \le y_0^{-k} e^{2 \pi g},
\end{align}
uniformly for all $(l,r;S)$ since $r \ge 1$ (as $c_M\neq 0$). This shows that $A_\gamma\to 0$ as $k\to \infty$ \textit{uniformly for all} $S$.

We also see from \eqref{agamma} and  \lemref{det>1}, that for $(Z,W)\in \mc C(y_0)$, the Poincar\'e series $P_{k,S}^{l,r}$ is dominated term wise by the series (since $k>2n+g$)
\begin{equation}
    \mc M_{S}^{l,r}:=  e^{2 \pi g}  \sumn_\gamma |\det(cZ+d)|^{-2n-g-1} \exp(-2\pi \tr(\mc T \cdot \im ( \gamma^\uparrow \lan \mc Z \ran) )\big),
\end{equation}
which converges absolutely and locally uniformly on $\hn\times M_{g,n}(\mbb C)$ and is independent of $k$. Combining this with the fact that $A_\gamma\to 0$ when $c_\gamma:=c_M\neq 0$ as $k \to \infty$, we get that $P_{k,S}^{l,r} \to \sum_{c_\gamma=0} A_\gamma$ by applying Lebesgue's dominated convergence theorem. 

Next, put
\begin{equation}
    \mc C_S := \{ \smat{Z}{W^t}{W}{i S} \in \mbb H_{n+g} | (Z,W) \in \mc C(y_0) \}.
\end{equation}
Note that any $ M\in \Sp{n,\infty}{\z} \backslash \Sp{n}{\z}$ with $c_M=0$ can be taken to be $M=\smat{u}{0}{0}{(u^t)^{-1}}$ with $u\in \GL{n}{\z}$. Thus, by using the dominated convergence theorem again, we see that (cf. \eqref{Clrint} and \eqref{agamma}) 
\begin{align}
    \lim_{k \to \infty}   C(l',r';P_{k,S}^{l,r}) &= \sumn_{\gamma, \, c_\gamma=0} \int_{\mc C_S} A_\gamma \, e\big(- \tr (l' Z + r'^t W) \big) dZ \, dW \n \\
    & = \sumn_{u\in \GL{n}{\z}, \lambda} \det(u)^k \int_{\mc C_S} 
    e\big(\tr (\mc T \cdot \mc \gamma^\uparrow \lan \mc Z  \ran - \mc{T}'\cdot\mc Z) \big)  dZ \, dW ,
\end{align}
where $\mc T' = \smat{l'}{r'^t/2}{r'/2}{S}$. Expanding out $\gamma^\uparrow \lan \mc Z  \ran$, we see that the integral is non--zero only when $\mc T [A] =\mc T'$, where $A=\smat{u}{0}{\lambda}{I_g}$. Thus, we obtain Proposition \ref{prop:asymporth}.

\subsection{Lower bound for Bergman kernel}\label{seclbd}
Using the asymptotic orthogonality property for the Jacobi Poincaré series from the previous section, we obtain the conjectured lower bound for $\sup(J^{cusp}_{k, S; (n,g)})$.

First, note that  the quantity $(\det v)^{k} e^{-4\pi \mrm{tr}( S v^{-1}[y] )} |\phi(\tau,z)|^2$ is invariant under $\Gamma^J_g$. Thus, using the Cauchy--Schwarz and basic integral inequalities, we have for any $v_0>0$ and $y_0\in M_{g,n}(\mbb R)$ (see \cite[section 3.1]{sd-hk} for example)
\begin{equation}
\begin{split}
\spJn & \gg (\det v_0)^{k} e^{-4\pi \mrm{tr}( S v_0^{-1}[y_0] )} e^{-\mrm{tr}(4\pi v_0)}\sumn_{\phi\in\mathcal{B}_{k,S;(n,g)}}|c_\phi(1_n,0)|^2.
\end{split}
\end{equation}

We also have that (from \eqref{PetTr})
\begin{equation}\label{LBPFC}
\sumn_{\phi\in\mathcal{B}_{k,S;(n,g)}}|c_\phi(1_n,0)|^2= \frac{ (4\pi)^{nk-n(n+g+1)/2} (\det 2S)^{n/2} }{\Gamma_n(k-\frac{n+g+1}{2})}C(1_n,0;P_{k,S}^{1_n,0}),
\end{equation}
From the asymptotic orthogonality, we have for $k\to \infty$,
\begin{equation}\label{c10P}
    C(1_n,0;P_{k,S}^{1_n,0})\gg 1.
\end{equation}
Combining this with \eqref{LBPFC} we see that
\begin{equation}
\begin{split}
    \spJn &\gg_{n,g} (\det v_0)^{k} e^{-4\pi \mrm{tr}( S v_0^{-1}[y_0] )} e^{-\mrm{tr}(4\pi v_0)} \frac{(4\pi)^{nk}(\det 2S)^{\frac{n}{2}}}{\Gamma_n(k-\frac{n+g+1}{2})}\\
    &\gg_{n,g} (\det v_0)^{k} e^{-4\pi \mrm{tr}( S v_0^{-1}[y_0] )} e^{-\mrm{tr}(4\pi v_0)} \frac{(4e\pi)^{nk} k^{\frac{3n(n+1)}{4} + \frac{gn}{2}} \det(2S)^{\frac{n}{2}}}{k^{nk}},
\end{split}
\end{equation}
by Stirling's formula for the Gamma function.

Then choosing  $v_0=\frac{k}{4\pi}1_n$ and $y_0=0$ we get  the following lower bound.
\begin{prop}\label{prop:low}
Let $k$ be even. Then
\begin{equation}
  \spJn \gg_{n,g} k^{\frac{3n(n+1)}{4} + \frac{gn}{2}} \det(2S)^{\frac{n}{2}} \q\text{ as } k\rightarrow\infty.
\end{equation}
\end{prop}

When $n=1$, the behaviour of Fourier coefficients $C(1_n,0;P_{k,S}^{1_n,0})$ is explicitly known (see \propref{proppkmnr} below and also \cite{das2010nonvanishing}). To be precise, we know that $C(1,0;P_{k,S}^{1,0})\gg 1$ for $k\ge k_0$ for some explicit $k_0$ depending only on $g$. Thus for such a $k_0$,
\begin{equation} \label{n1lbd}
  \sup(J^{cusp}_{k,S,g}) \gg_{g} k^{\frac{3}{2} + \frac{g}{2}} \det(2S)^{\frac{1}{2}} \q\text{ for }\, k\ge k_0;
\end{equation}

\section{Upper bounds for the sup-norms of Jacobi forms on average} \label{jacobi-supnorm}
\subsection{Bounds from the Fourier expansion.}\label{BFE}
First, we show (cf. \cite{sd-hk}) that the size of the Bergman kernel can be understood via the Fourier coefficients of the Siegel-Jacobi Poincar{\'e} series. Let $P_{k,S}^{n,r}\in J_{k,S,g}^{cusp}$ denote the $(n,r)$-th Jacobi Poincaré series and denote their Fourier coefficients by $C(n',r';P_{k,S}^{n,r})$. Then we have 
\begin{equation}\label{PTF}
    \sumn_{\phi\in\mathcal{B}_{k,S,g}}|c_{\phi}(n,r)|^2=C(n,r;P_{k,S}^{n,r})\lambda_{k,S,D}^{-1}=:p_{k,S}(n,r),
\end{equation}
where we write $D=\det\psmb 2n& r\\ r^t&2S \psme$ and 
\begin{equation} \label{lk}
    \lambda_{k,S,D}= \Gamma(\ell)(\det 2S)^{\ell-1/2} 2^{-g/2}(2\pi D)^{-\ell} \,\text{ with } \ell=k-g/2-1.
\end{equation}
Then same arguments as in \cite[Lemma~4.4]{sd-hk} give us that
\begin{equation*}
    \mathbb{B}_{k,S,g}(\tau,z)\le v^{k}e^{-4\pi S[y^t]/v}\big( \sum_{n,r; \, 4n> S^{-1}[r]}p_{k,S}(n,r)^{\frac{1}{2}}e^{-2\pi nv}e^{-2\pi ry}\big)^2=: q_{k,S}(y,v)^2.
\end{equation*}
Since $\mathbb{B}_{k,S,g}(\tau,z)$ is invariant under the Jacobi group, we may and will restrict $(\tau,z)$ in the standard Jacobi fundamental domain given by (here $\mc F$ is the standard fundamental domain for the action of $\SL{2}{\z}$ on $\mbb H$)
\begin{align} \label{jacobi-funddom}
    \mc F^J = \mc F^J_g := \{ \tau=u+iv, z=x+iy \mid \tau \in \mc F, y \bmod v,\, x\bmod 1\}.
\end{align}

Using bounds for the Salie-Kloosterman sums (cf. \cite{boecherer1993estimates}) appearing in the expression for $p_{k,S}(n,r)$ and that on the Bessel functions, we get the following asymptotic for the $(n,r)$-th Fourier coefficient of the Jacobi Poincar{\'e} series $P_{k,S}^{n,r}$ (see \cite[(5.1),(5.2)]{das2010nonvanishing} for the bounds).

\begin{prop}\label{proppkmnr}
Let $p_{k,S}(n,r)$ be as above. Then for, $k\ge g+2$ (and $\ell$ as in \eqref{lk})
\begin{equation}\label{p(T)bound}
 p_{k,S}(n,r) = \frac{ 2^{g/2} (2\pi D)^\ell}{\Gamma(\ell)(\det 2S)^{\ell-1/2}}\big(2+ O\big( \frac{D^{g/2+\epsilon}}{\ell^{g/2+1/3}(\det 2S)^{g/2+1/2+\epsilon}} \big) \big).
\end{equation}
\end{prop}

For convenience, let us define the quantities $v'$ and $\widetilde{S}>0$, a matrix with integral entries, by
\begin{align} \label{v'Stilde}
v':=v(2\det(2S))^{-1}, \q \widetilde{S}:= \det(2S)(2S)^{-1}. 
\end{align} 
Let us note that $D=\det(2S)(2n -(2S)^{-1}[r])= 2n \det(2S) - \widetilde{S}[r] $. Now consider the following sum over $D$ and $r$:
\begin{equation}\label{sumQ}
    \mc Q(\alpha,\beta):=\sumn_{D,r}  ( 4\pi v'D)^{\ell/2} \big/ \Gamma(\ell)^{1/2} \exp\left(-2\pi v'D\right)) \ell^{\alpha} D^{-\beta} \exp(-2\pi(ry+\widetilde{S}[r]v')),
\end{equation}
where $r$ runs over $\mbb Z^g$ and $D>0$ varies such that $D\equiv - \widetilde{S}[r]  \bmod{\det(2S)}$. 

Corresponding to the two terms in \eqref{p(T)bound} we can then write
\begin{equation}\label{qkS}
q_{k,S}(y,v)\ll  q_{k,S}^{(1)}(y,v)+q_{k,S}^{(2)}(y,v) ,
\end{equation}
where
\begin{align}
q_{k,S}^{(1)}(y,v)&=  v^{\frac{g}{4}+\frac{1}{2}}(\det (2S))^{\frac{1}{4} } e^{-2\pi S[y]/v} \mc Q(0, 0), \label{q1} \\
q_{k,S}^{(2)}(y,v)&=v^{\frac{g}{4}+\frac{1}{2}} (\det (2S))^{-\frac{g}{4}-\epsilon} e^{-2\pi S[y]/v} \mc Q\big(-\frac{g}{4}-\frac{1}{6}, -\frac{g}{4}-\epsilon\big). \label{q2}
\end{align}
In \eqref{sumQ}, denote the sum over $D>0$ by $\mc Q_1$. Using Stirling's formula, we have 
\begin{equation} \label{Q1Dsum}
    \mc Q_1(\alpha,\beta) \asymp \sumn_D \big( 4\pi v'D \big/ \ell \big)^{\ell/2}\exp\left(\ell/2-2\pi v'D\right) \ell^{\alpha+1/4} D^{-\beta}.
\end{equation}

\begin{lem}\label{lemQ1}
Let $\mc Q_1(\alpha,\beta)$ be as above. Then (with $\ell$ as in \eqref{lk}),
\begin{equation} 
\mc Q_1(\alpha,\beta)\ll \begin{cases} v^{\beta} \ell^{\alpha-\beta+1/4} (\det(2S))^{-\beta} & \text{ for } v\gg \ell^{1/2+\epsilon};\\
\ell^{\alpha-\beta+3/4+\epsilon} v^{\beta-1}  (\det (2S))^{- \beta} \left(1+v\ell^{-1/2}\right) & \text{ for }  v\ll \ell^{1/2+\epsilon}.
\end{cases}
\end{equation}
\end{lem}
\begin{proof}
First, note that the function $h(x):= (4\pi x/\ell)^{\ell/2}\exp(\ell/2-2\pi x)$ attains maximum at $x=\ell/(4\pi)$. Moreover, when $|x-\ell / (4\pi)| \gg \ell^{1/2+\epsilon}$, $h(x)$ has exponential decay in $\ell$. Below, we take some care with respect to the parameter $S$, which we want to control uniformly.

    When $v'D\not \in \ell / (4\pi) +O(\ell^{1/2+\epsilon})$, we divide the range of $D$ into $4\pi v'D\le 2\ell$ and the dyadic intervals $2^t\ell< 4\pi v'D \le 2^{t+1}\ell$ for $t\ge 1$. Next, we count the $D$ satisfying the congruence condition $D \equiv - \widetilde{S}[r]  \bmod{\det(2S)}$. For each of the $D$ above, the corresponding summands decay sub-exponentially (cf. \cite[(4.26)]{sd-hk}). Finally, summing over these sub-intervals, we see that the tail of $\mc Q_1(\alpha,\beta)$ is bounded by $\exp(-c\ell^\epsilon)(v'\det(2S))^{-1}=\exp(-c\ell^\epsilon)v^{-1}$, for some constant $c>0$ (see \cite{sd-hk}, especially \cite[Lemma~4.10, Lemma~4.11]{sd-hk}).
This reduces the sum \eqref{Q1Dsum} into a finite sum over $D$ satisfying $v'D=\frac{\ell}{4\pi}+O(\ell^{1/2+\epsilon})$.

We put the relation $v'D \asymp \ell $ in the resulting finite sum and thus end up with the following bound:
$\mc Q_1(\alpha,\beta) \ll \mc C_{v} v'^{\beta} \ell^{\alpha-\beta+1/4}$ taking into account the congruence condition;
where 
\begin{equation} \label{cv}
    \mc C_{v}:=\# \{D>0\, |\, D \equiv - \widetilde{S}[r]  \bmod{\det(2S)} , v'D=\frac{\ell}{ 4\pi}+O(\ell^{1/2+\epsilon})\}\ll \frac{\ell^{1/2+\epsilon}}{\det(2S) v'} +1.
\end{equation}
Recall that $v'\det(2S)=v/2$ from \eqref{v'Stilde}.
From \eqref{cv}, $\mc C_{v}\ll 1$ if $\ell^{1/2+\epsilon}v^{-1}\ll 1$. In this case,
\begin{equation}
    \mc Q_1(\alpha,\beta)\ll v^{\beta} \ell^{\alpha-\beta+1/4} (\det(2S))^{-\beta}.
\end{equation}
In the other case, i.e., when $\ell^{1/2+\epsilon}v^{-1}\gg 1$, we have 
\begin{equation}\label{Q1bound}
    \mc Q_1(\alpha,\beta)\ll \ell^{\alpha-\beta+3/4+\epsilon}v^{\beta-1}  (\det (2S))^{-\beta} \,  \left(1+v\ell^{-1/2}\right).\qedhere
\end{equation}
\end{proof}

Since the bound for $\mc Q_1$ is independent of $r\in\z^g$, we have
\begin{equation} \label{rmyv1}
\mc Q(\alpha,\beta)\ll \mc Q_1(\alpha,\beta) R_S(y,v); \q R_S(y,v):=\sumn_{r\in \mbb Z^g}\exp(-2\pi(ry+\widetilde{S}[r]v')).
\end{equation}
To estimate $R_S(y,v)$, we simplify things by invoking the fact that $S$ can be assumed to be Minkowski-reduced. 
If we denote by $s_1,s_2,\ldots, s_g$ be the diagonal elements of $2S$ and put $r=(r_1,r_2,\ldots, r_g)$, then we know that
\begin{align}
    \widetilde{S}[r] \asymp_g (r_1^2\det(2S)/s_1 +  r_2^2\det(2S)/s_2 +\ldots  r_g^2\det(2S)/s_g).
\end{align}
Then the quantity $R_S(y,v)$ from \eqref{rmyv1} can be estimated as a product of $g$ one-dimensional sums:
\begin{align}
    R_S(y,v) &\ll \prodd_j \sumn_{r_j \in \z} \exp(-2 \pi (r_jy_j + r_j^2 v  s_j^{-1} )).
\end{align}
The one dimensional sums $R_{s_j}(y_j,v)$ are estimated using the Gaussian integrals (see e.g., \cite[No.~ 3.322]{gradshteyn2014table}), and we get
\begin{equation}\label{bndRm}
    R_{s_j}(y_j,v)\le e^{2\pi s_jy_j^2/v}\left(1+\sqrt{s_j v^{-1}}\right),
\end{equation}
so that
\begin{equation}\label{bndRM}
   R_S(y,v) \ll \prodd_j e^{2\pi s_j y_j^2/v}\Big(1+\sqrt{s_j v^{-1}}\Big) =  e^{2 \pi S[y^t]/v} \prodd_j \left(1+\sqrt{s_j v^{-1}}\right).
\end{equation}
Thus from \lemref{lemQ1}, \eqref{bndRm} and \eqref{qkS}, \eqref{q1}, \eqref{q2} along with the fact that $v \gg 1$ from \eqref{jacobi-funddom}, we arrive at the following bound for $q_{k,S}$.

\begin{lem}\label{lemqKS}
Let $q_{k,S}(y,v)$ be as in \eqref{qkS}. Then (with $\ell$ as in \eqref{lk})
\begin{equation}
   q_{k,S}(y,v) \ll \begin{cases} \ell^{1/4}\ (\det (2S))^{\frac{1}{4}}\ v^{\frac{g}{4}+\frac{1}{2}}\ \prod_j \left(1+\sqrt{s_j v^{-1}}\right) & \text{ for } v\gg \ell^{1/2+\epsilon};\\
   \ell^{3/4+\epsilon}\ (\det (2S))^{\frac{1}{4}}\ v^{\frac{g}{4}-\frac{1}{2}}\ \left(1+v\ell^{-1/2}\right) \prod_j \left(1+\sqrt{s_j v^{-1}}\right) & \text{ otherwise}.
   \end{cases}
\end{equation}
\end{lem}

\subsection{Bounds from the Bergman kernel.}\label{bkjkmsec}
We have from \cite[Lemma 3.2]{arakawa1992selberg} and \eqref{bkdef},
\begin{equation}\label{BKGeometric}
    B_{k,S,g}(\tau,z;\tau_0,z_0)= 2^{k-2}\ \pi^{-1}\ \ell \det (2S)^{1/2}\  i^{\ell+1}\ \sum_{\gamma\in\mrm{SL}_2(\mbb{Z})}(\Theta|^{(1)}_{k,S}\gamma)(\tau,z;\tau_0,z_0),
\end{equation}
where $\Theta |^{(1)} \gamma$ indicates the actions of $\gamma$ with respect to the first set of variables, and
\begin{align} \label{theta-def}
    \Theta=\Theta_{k,S}(\tau,z;\tau_0,z_0) &:=(\tau-\bar{\tau}_0)^{\frac{g}{2}-k}\sum_{\eta\in (2S)^{-1}\mathbb Z^g/\mathbb Z^g}\theta_{S,\eta}(\tau,z)\overline{\theta_{S,\eta}(\tau_0,z_0)}, \q \q \text{and}  \\
\theta_{S,\eta}(\tau,z)& :=\sumn_{r\in \mbb Z^g}e(\tau S[r+\eta]+2S(\eta+r,z)).
\end{align}

Let $\gamma=\begin{psmallmatrix}
a&b\\c&d
\end{psmallmatrix}\in \mrm{SL}_2(\mbb Z)$, then
\begin{equation}\label{thetatrans}
\theta_{S,\mu}\left(\gamma(\tau,z)\right)(c \tau + d)^{-\frac{g}{2}} e\left(-c S[z] (c \tau+d)^{-1}\right)=\sum_{\eta\in(2S)^{-1}\mathbb Z^g/\mathbb Z^g} \varepsilon_{S}(\eta, \mu;\gamma)\theta_{S,\eta}(\tau, z),
\end{equation}
where $\varepsilon_{S}(\eta, \mu;\gamma)$ are complex numbers of absolute value $1$ as in e.g. \cite{arakawa1992selberg}.
From \eqref{naiveBK} and \eqref{mBK} we have that
\begin{equation}
   \mathbb B_{k,S,g}(\tau,z)=v^{k}e^{-4\pi S[y^t]/v}B_{k,S,g}(\tau,z;\tau,z).
\end{equation}
Thus the modified Bergman kernel $\mathbb B_{k,S,g}(\tau,z)$ can now be written as
\begin{equation}\label{bkmtheta}
C_{k,S,g}(v,y)\sum_{\gamma\in\mrm{SL}_2(\mbb{Z})}v^k\  \left(j(\gamma,\tau)(\gamma(\tau)-\overline{\tau})\right)^{-k+\frac{g}{2}}\sum_{\eta,\mu}\varepsilon_{S}(\eta, \mu;\gamma)\theta_{S,\eta}(\tau, z)\overline{\theta_{S,\mu}(\tau,z)},
\end{equation}
where $C_{k,S,g}(v,y)=\ell \det (2S)^{1/2}\ e^{-4\pi S[y^t]/v}$.

From \eqref{bkmtheta}, we can bound $\mathbb B_{k,S,g}(\tau,z)$ by
\begin{equation}\label{BKbydef}
    \ll C_{k,S,g}(v,y) v^k\sum_{\gamma}\left|j(\gamma,\tau)(\gamma(\tau)-\overline{\tau})/2i\right|^{-k+\frac{g}{2}}\big|\sum_{\nu,\mu}\varepsilon_{S}(\nu, \mu;\gamma)\theta_{S,\nu}(\tau, z)\overline{\theta_{S,\mu}(\tau,z)}\big|.
\end{equation}

Let $\rho_S(\gamma)=\left(\varepsilon_{S}(\nu, \mu;\gamma)\right)_{\nu,\mu}$ and $\Theta=(\theta_{S,\mu})_{\mu}$. Then we know that $\rho_S(\gamma)$ is unitary (see e.g., \cite{arakawa1992selberg}). Thus the sum over $\nu,\mu$ is 
\begin{align}
    =\lan \rho_S(\gamma)\Theta,\Theta\ran {\ll}\lan \rho_S(\gamma)\Theta,\rho_S(\gamma)\Theta\ran^{1/2}\lan \Theta,\Theta\ran^{1/2} =\lan \Theta,\Theta\ran=\sumn_\mu |\theta_{S,\mu}|^2,
\end{align}
where we use Cauchy-Schwartz for the first inequality, and the unitary property of $\rho_S(\gamma)$ for the second. Here $\lan \, , \, \ran$ denotes the standard inner product on Euclidean space.

Note that from \cite[Lemma 3.2]{arakawa1992selberg}, we have (see also \cite[p.~182, (11)]{sko-zag} for $n=1$)
\begin{equation}
   \sumn_\mu |\theta_{S,\mu}(\tau,z)|^2 =(\det S)^{1/2}\ v^{-g/2}\ \Theta^*(\tau,z),
\end{equation}
where
\begin{align} \label{theta*}
   \Theta^*(\tau,z)&=  \sumn_{\lambda, \mu\in \mathbb Z^g} e\left(  -S [z  - \bar{z} -\lambda \bar{\tau} -\mu] (\tau - \bar{\tau})^{-1}   - \bar{\tau} S[\lambda] -  2 \lambda^t S \bar{z}  \right).
\end{align}
We want good upper bounds for $\Theta^*(\tau,z)$. In \eqref{theta*} we first consider the sum over $\mu$ and see that it is bounded by
\begin{align}
    \le \exp(4 \pi S[y]/v) \sumn_{ \mu\in \mathbb Z^g} \exp \Big(  \frac{\pi(-u^2+v^2)}{v} S[\lambda] - \frac{\pi}{v} S[\mu] -\frac{2 \pi}{v} u\lambda^t S\mu  +  4\pi y^t S \lambda \Big).
\end{align}
Now considering the sum over $\lambda$ as well and bounding it absolutely, we are left with
\begin{align}
     \Theta^*(\tau,z)& \ll \exp(4 \pi S[y^t]/v) \sumn_{\lambda, \mu}\exp \Big(  \frac{\pi(-u^2+v^2)}{v} S[\lambda] - \frac{\pi}{v} S[\mu] \Big)\n \\
     &\times \exp \Big( -\frac{2 \pi}{v} u\lambda^t S\mu  + 4 \pi y^t S \lambda - 2 \pi vS[\lambda] - 4 \pi \lambda^t S y \Big)\n \\
     & \ll \exp(4 \pi S[y^t]/v) \sumn_{\lambda, \mu} \exp \Big(  \frac{- \pi(u^2+v^2)}{v} S[\lambda] - \frac{\pi}{v} S[\mu] -\frac{2 \pi}{v} u\lambda^t S\mu \Big).
\end{align}
We then notice that 
\begin{align*}
    \sum_{\lambda, \mu} \exp \Big(  \frac{- \pi(u^2+v^2)}{v} S[\lambda] - \frac{\pi}{v} S[\mu] -\frac{2 \pi}{v} u\lambda^t S\mu \Big) = \sum_{\lambda, \mu} \exp \Big( - \pi v S[\lambda] - \frac{\pi}{v} S[u \lambda + \mu] \Big). \label{g-dimest}
\end{align*}
At this point, we can assume that $S$ is diagonal, with $s_j$ denoting the diagonal elements. Therefore, we can bound \eqref{g-dimest} as
\begin{align}
    \prod \nolimits_j \sumn_{\lambda_j, \mu_j} \exp \left( - a  s_j \lambda_j^2 - b s_j (u \lambda_j + \mu_j)^2 \right);
\end{align}
where $a \asymp_g v$ and $b \asymp_g v^{-1}$. For the one-dimensional sums, using the well-known bounds on Gaussian exponential integrals (see e.g., \cite[No.~ 3.322]{gradshteyn2014table}), we get
\begin{equation}
    \sumn_{p,q}  \exp (-  m (a p^2 + b (up+q)^2)\ll \left(1+(am)^{-1/2}\right)\left(1+(bm)^{-1/2}\right).
\end{equation}

Thus, we have
\begin{align}
   \Theta^*(\tau,z)\ll e^{4 \pi S [y^t]/v}\prod \nolimits_j \Big(1+ \sqrt{vs_j^{-1}}\Big) \Big(1+ \sqrt{(s_jv)^{-1}}\Big)\ll e^{4 \pi S [y^t]/v}\prod \nolimits_j \Big(1+ \sqrt{vs_j^{-1}}\Big).
\end{align}
Now we bound the sum over $\gamma$. Write $r=k-\frac{g}{2}\,(=\ell +1)$ and define
\begin{equation}
    M(\tau):=\sumn_{ \gamma} v^r \, \big |(\frac{\gamma(\tau)-\overline{\tau}}{2i})j(\gamma,\tau)\big |^{-r}.
\end{equation}
From Corollary 2.16 and Proposition 3.2 in \cite{RSSreal}, we have for any $\eta>0$ and $A>0$
\begin{equation}
    M(\tau)\ll 1+vr^{-\frac{1}{2}+\eta}+vr^{-A}.
\end{equation}
Then the sum over $\gamma$ in \eqref{BKbydef} is bounded by $M(\tau)$. Thus, we have
\begin{lem}\label{lemBKGeo}
For any $(\tau,z) \in\mc F^J$ and $\ell$ as in \eqref{lk},
\begin{align}\label{BKGeo}
    \mathbb B_{k,S,g}(\tau,z)\ll \ell \det (2S) \left( 1+v\ell^{-\frac{1}{2}+\eta}+v\ell^{-A}\right)\prod \nolimits_j \Big(1+ \sqrt{vs_j^{-1}}\Big).
\end{align}
\end{lem}

\subsection{Final Bounds} With the bounds on BK from the Fourier expansion obtained in subsection \ref{BFE} and the bounds from the geometric side of the BK in subsection \ref{bkjkmsec} at hand, we are now ready to prove the upper bound in \thmref{mainthm2}. To do this, we split the fundamental domain $\mc F^J_g$ ($\mc F_1$ to be precise) into suitable regions and use the bounds from the aforementioned subsections. In particular, we consider the following regions. Recall the definition of $\ell$ from \eqref{lk}.

\subsubsection{$v\gg \ell \det(2S)^{1+\epsilon}$:} In this region, we show that $\mathbb B_{k,S,g}(\tau,z)$ has exponential decay both in $\ell$ and $\det(2S)$. To do this, first note that from \eqref{qkS} and \eqref{sumQ}, $\mathbb B_{k,S,g}(\tau,z)$ is bounded by the quantity
\begin{equation}
   v^{g/2+1} \det(2S)^{1/2} e^{-4\pi S[y]/v} R_S(y,v)^2 \Big( \mc Q_1(0,0)^2+\det(2S)^{-g/2-1/2-\epsilon} \mc Q_1(-\frac{g}{4}-\frac{1}{6},-\frac{g}{4}-\epsilon)^2\Big).
\end{equation}
Next, using \cite[Lemma 4.11]{sd-hk}, for any finite $\alpha,\beta$, we see that the quantity $\mc Q_1(\alpha,\beta)$ has exponential type decay in the region $v'\gg \ell$. To be precise, 
\begin{equation}
    \mc Q_1(\alpha,\beta)\ll \exp(-c_0 v').
\end{equation}
Thus when $v\gg \ell (\det(2S))^{1+\epsilon}$, the quantity $\ \mathbb B_{k,S,g}(\tau,z)$ is (with some absolute constant $c>0$)
\begin{align}
  \ll  v'^{g/2+1} \det(2S)^{g/2+3/2}\ \prodd_j \left(1+ \sqrt{s_j v^{-1}}\right) \exp(-c_0\, v') \ll \exp(-c\, k (\det(2S))^{\epsilon}).
\end{align}

\subsubsection{$\ell \ll v \ll \ell (\det(2S))^{1+\epsilon}$:} In this region,  from \lemref{lemBKGeo} we get
\begin{align*}
    \bksg &\ll \ell^{3/2+ \epsilon} (\det(2S))^{2+\epsilon} \prodd_j \big(1+\sqrt{\frac{\ell (\det(2S))^{1+\epsilon}}{ s_j}} \big) 
    \ll \ell^{3/2+g/2+\epsilon} (\det(2S))^{g/2+3/2+\epsilon}.
\end{align*}

\subsubsection{$ v\ll \ell $:}
First, consider the case $v \gg \ell^{1/2+\epsilon} \det(2S)$. In this region, from the first bound in \lemref{lemqKS}, we have
\begin{equation}
    \bksg \ll \ell^{3/2+g/2} \det(2S)^{1/2}.
\end{equation}
Note here that the condition on $v$ implies that $\det (2S)\ll \ell^{1/2-\epsilon}$.

Next, we consider the intersection of regions $v\ll \ell$ and $v\ll \ell^{1/2+\epsilon} \det(2S)$. Thus, we have to worry only about the case $\det(2S)\ll \ell^{1/2-\epsilon}$. In this region, from Lemma \ref{lemBKGeo}, we have
\begin{equation}
    \mathbb B_{k,S,g}(\tau,z) \ll \ell^{3/2+\epsilon} \det (2S)  \prod \nolimits_j \Big(1+ \ell^{1/2} s_j^{-1/2} \Big)\ll \ell^{3/2+g/2+\epsilon} \det(2S)^{1/2}.
\end{equation}
Thus we have the following bound on $\bksg$. 
\begin{prop} \label{jacobi-BK}
For $(\tau,z)\in \mc F^J$, we have
\begin{equation*}
    \bksg\ll k^{3/2+g/2+\epsilon} (\det(2S))^{g/2+3/2+\epsilon}.
\end{equation*}
\end{prop}

\section{Contribution from index old and new forms.} \label{old-strange1}
In this section, we restrict ourselves to $n=1$ and $g=1$. We demonstrate that for certain indices $m$, there are `\textit{large}' natural subspaces of $J_{k,m}^{cusp}$. This is done by (unconditionally) computing lower bounds or even the sizes of their respective BK.
Here, by \textit{index old subspaces} (cf. \cite[p.~138]{skoruppa1988jacobi}) we mean the image of lower index Jacobi forms in $J_{k,m}^{cusp}$ under the (combination of the) Hecke operators $U_l$ and $V_m$. We also show that for BK of the \textit{index newspace}, that is the orthogonal complement of the index old subspace in $J_{k,m}^{cusp}$, the same is true as mentioned above. Unconditionally, we also show that it is possible to improve the upper bound from Proposition \ref{jacobi-BK} for some of these old subspaces, which indicate that the bound on BK is closer to the conjectured bound than that implied by \propref{jacobi-BK}.

\subsection{Size of some index-old spaces.}
In this section we demonstrate our assertions made above about the sizes of certain index-old subspaces.

\subsubsection{The old-space $U_l(\jk)$}
Consider the operator $U_l:J_{k,1}\rightarrow J_{k,l^2}$. We estimate the contribution of the old-space $U_l (\jk)$ towards the BK for $J_{k,l^2}^{cusp}$. That is, we estimate the quantity
\begin{align}
    \sup(U_l( \jk)) := \sup \nolimits_{\h \times \mbb C}\sumn_\phi \frac{ v^k e^{- 4 \pi l^2 y^2/v}  |\phi(\tau,lz)|^2 }{ \lan U_l \phi , U_l \phi \ran}.
\end{align}

For $U_l \phi = \phi(\tau,lz)$, first we  compute $\lan U_l \phi , U_l \phi \ran $ in terms of $\lan \phi, \phi \ran$. We have an explicit formula for $U_l^*$ but not for $ U_l^* U_l$, which we would need, so we compute things directly. Write $d\mu(\tau, z)= dxdydudv/v^3$. Then $\lan U_l \phi , U_l \phi \ran$ equals
\begin{align}
     \underset{y \bmod v, x\bmod 1}{\int_{\tau \in \mc F_1}} & v^k e^{-4\pi l^2y^2/v} |\phi(\tau,lz) |^2 d\mu(\tau, z)
     = l^{-2} \underset{y \bmod v, x\bmod l}{\int_{\tau \in \mc F_1}} v^k e^{-4\pi y^2v} |\phi(\tau,z) |^2 d\mu(\tau, z) \n \\
  & = l^{-2} \sum_{j=1}^l \int_{\tau \in \mc F_1, (j-1)v \le y \le  jv, x\bmod l} v^k e^{-4\pi y^2/v} |\phi(\tau,z ) |^2 d\mu(\tau, z). \n \\
\end{align}
Making a change of variable $z\mapsto z + (j-1) \tau $, we get
\begin{align}
    &=  l^{-1} \sum_{j=1}^l \int_{\tau \in \mc F_1, 0 \le y \le  v,x\bmod 1} v^k e^{-4\pi (y + (j-1) v)^2/v} |\phi(\tau,z + (j-1) \tau ) |^2 d\mu(\tau, z) \n \\
    &=  l^{-1} \sum_{j=1}^l \int_{\tau \in \mc F_1, 0 \le y \le  v, x\bmod 1} v^k e^{-4\pi (y + (j-1) v)^2/v} |e(- (j-1)^2 \tau-2 (j-1) z)|^2 |\phi(\tau,z  ) |^2 d\mu(\tau, z) \n \\
    & = l^{-1} \cdot l \, \lan \phi, \phi \ran =  \lan \phi, \phi \ran. 
\end{align}
The same calculation holds for two forms $\phi, \phi'$. In particular, orthogonal forms go to orthogonal forms.  Also, note that $U_l$ is injective. Thus we see that $ \sup(U_l( \jk))$ equals

\begin{equation*}
\sup_{\h \times \mbb C}\sum_\phi \frac{ v^k e^{- 4 \pi l^2 y^2/v}  |\phi(\tau,lz)|^2 }{ \lan U_l \phi , U_l \phi \ran} 
    = \, \sup_{\h \times \mbb C}\sum_\phi \frac{ v^k e^{- 4 \pi  y^2/v}  |\phi(\tau,z)|^2 }{ \lan  \phi ,  \phi \ran} 
    =\sup(\jk) \asymp k^2.
\end{equation*}
Thus the size of $U_l (\jk)$ is same as that of $\jk$.

\subsubsection{The old-space $V_p(\jk)$}
Next, we consider the operators $V_m:J_{k,1}\rightarrow J_{k,m}$ defined for each $m \ge 1$ in \eqref{vmdef}. When $m=p$, a prime, the image $V_p(J_{k,1})$ is part of the index-old subspace inside $\jkp$. We will discuss the properties of this operator in more detail in section~\ref{compactregbkm}. Now we only borrow a fact from \lemref{ortholem}. Namely, $\lan V_p \phi,  V_p \phi \ran= (\lambda_f(p) + (p+1) p^{k-2} \lambda_f(1))\lan \phi, \phi \ran \gg p^{k-1} \lan \phi, \phi \ran$, and that $V_p(\phi), V_p( \psi)$ are orthogonal if $\phi,\psi$ are orthogonal (or distinct) Hecke eigenforms. Since $\jk$ admits a Hecke basis, it follows from these two facts that $V_p$ is injective (but this may not be true for non-prime indices) and thus for a Hecke basis $\{\phi\}$ of $J_{k,1}$, we can write
\begin{equation}
   \sup(V_p(J_{k,1}^{cusp})):= \sup\nolimits_{(\tau,z)\in\h\times \mbb C} \sumn_{\phi  } (\widetilde{ V_p(\phi)}/\norm{V_p \phi})^2,
\end{equation}
where, for $\phi \in \jkm$ and $m \ge 1$, we define $\widetilde{\phi}(\tau,z)^2 : =v^{k}e^{-4\pi my^2/v}|\phi(\tau,z)|^2,$
and note that it is $\Gamma^J$ invariant.
From its definition, notice that $\widetilde{V_m(\phi)}^2$ equals
\begin{align}
     & v^{k}e^{-4\pi my^2/v}|V_m(\phi)|^2
     = m^{-2} v^{k}e^{-4\pi my^2/v} \big|\sum_{ad=m}\sum_{b\mod d}a^{k}\phi(\frac{a\tau+b}{d},az)\big|^2 \label{vm-def} \\
    &= m^{-2} \big|\sum_{ad=m}\sum_{b\mod d}(a d)^{k/2} \widetilde{\phi}(\frac{a\tau+b}{d},az)\big|^2 = m^{k-2} \big|\sum_{d|m}\sum_{b\mod d} \widetilde{\phi}(\frac{a\tau+b}{d},az)\big|^2.   \n
\end{align}
Next, we can use the Cauchy-Schwartz inequality and the lower bound on $\norm{V_p(\phi)}$ to write
\begin{align*}
    \sup(V_p(J_{k,1}^{cusp})) &\ll p^{-1} \sup \nolimits_{(\tau,z)\in \mbb H \times \mbb C} \sum_{ \phi } ( \sum_{ad=p} d ) \cdot \sum_{ad=p} \sum_{b\mod d} \widetilde{\phi}\big(\frac{a\tau+b}{d},az\big)^2  \n\\
    &\ll p^{-1}(p+1) \sup \nolimits_{(\tau,z)\in \mbb H \times \mbb C}  \sum_{ad=p} \sum_{b\mod d} \sum_{\phi} \widetilde{ \phi }\big(\frac{a\tau+b}{d},az \big)^2  \n\\
    &\ll p^{-1}(p+1)^2 \sup \nolimits_{(\tau,z)\in \mbb H \times \mbb C} \sumn_\phi \widetilde{ \phi } (\tau,z)^2
    \ll  p \ k^2.
\end{align*}
Therefore this bound is better than what we would have obtained from \propref{jacobi-BK}.

Now we compute a lower bound for $\sup(V_m(J_{k,1}^{cusp}))$. For any $m \ge 1$ and any fixed $v_0 >0,y_0$, using the Fourier expansion of $V_m\phi$ from \eqref{vmfe} and $\norm{V_m\phi}^2\ll m^{k-1+\epsilon}\norm{\phi}^2$ (see section \ref{compactregbkm}), we get that

\begin{align*}
    \sup(V_m(J_{k,1}^{cusp}))&\ge \int_{u \mod 1}\int_{x\mod 1}\sum_{\phi}v_0^{k}e^{-\frac{4\pi my_0^2}{v_0}}(|V_m\phi(u+iv_0,x+iy_0)|/\norm{V_m\phi})^2du dx\\
    &\ge v_0^{k} e^{-\frac{4\pi my_0^2}{v_0}} e^{-4\pi v_0} \sum_\phi\frac{1}{\norm{V_m\phi}^2} \left| \int\int \frac{V_m\phi(u+iv_0,x+iy_0)}{ e((u+iv_0))}du dx\right|^2\\
    &=v_0^{k} e^{-\frac{4\pi my_0^2}{v_0}} e^{-4\pi v_0} \sum_\phi\frac{\big| c_{V_m\phi}(1,0)\big|^2}{\norm{V_m\phi}^2}\\ 
    &\gg m^{-k+1-\epsilon} v_0^{k} e^{-\frac{4\pi my_0^2}{v_0}} e^{-4\pi v_0}\sum_\phi \left(\big| c_{\phi}(m,0)\big|/\norm{\phi}\right)^2\\
   &= m^{-k+1-\epsilon} v_0^{k} e^{-\frac{4\pi my_0^2}{v_0}} e^{-4\pi v_0} (\lambda_{k,1,4m})^{-1}C(m,0;P_{k,1}^{m,0}).
\end{align*}
where $\lambda_{k,1,4m}\sim \Gamma(k-3/2)(4\pi m)^{-k+3/2}$ is as in \eqref{lk} if we notice that $S=1$ and $D =- 4m$ in loc. cit.
We know that (cf. \cite[Prop.~6.3]{das2010nonvanishing}) $C(t,0;P_{k,1}^{t,0})> 1/2$ for $t\ll k^{2-\epsilon}$. Thus, choosing $y_0=0$, $v_0=k/4\pi$, we conclude that for any $\epsilon>0$,
\begin{equation}
    \sup(V_m(J_{k,1}^{cusp}))\gg_\epsilon k^{2} m^{-1/2 - \epsilon}.
\end{equation}
We summarize the results of this subsection in the following proposition.
\begin{prop} \label{old-prop}
For any $l \ge 1$, we have $\sup(U_l(J_{k,1}^{cusp})) \asymp k^2 l$. For any prime $p$, $k^2 p^{-1/2 -\epsilon} \ll \sup(V_p(J_{k,1}^{cusp}))\ll k^{2} p$, the lower bound being actually valid for any $m \ge 1$. The subspace $U_l(J_{k,1}^{cusp})$ has the same size, up to absolute constants, as that of the space $\jk$.
\end{prop}

\subsection{Size of index-new space.} \label{victory}
We first define two classical notions of a `new-space' in the literature (which we denote by ''new1'' and ''new2'') in the context of $\jkm$, and show that they are actually equal. Apparently these two definitions are in use in \cite{man-ram}, but we couldn't find a relation between them in \cite{man-ram}.
Even though we work here only with $m$, a prime, for which subsection~\ref{barmao} would have been enough, but for future applications in mind and to fill this seeming gap in the literature, we present \lemref{jgap}. We believe that it will be useful for the community. This subsection also explains our need to write subsection~\ref{barmao} in the first place, something we believe should be communicated.

We have the following decomposition for $\jkm$ into index-new and index-old spaces according to \cite[p.~138, (4)]{skoruppa1988jacobi} (see also \cite[p.~2613, (41)]{man-ram}). The first definition states that it is just the orthogonal complement of the (explicitly defined) space of index-old forms.
\begin{equation} \label{new1def}
    \jkm= \jkmone \bigoplus^{\perp} \nolimits_{ld^2|m,ld^2>1} J_{k,m/ld^2}|U_d\circ V_l =\jkmone \bigoplus \nolimits_{(l,d)} J^{cusp, new1}_{k,m/ld^2}|U_d\circ V_l.
\end{equation} 
The second definition runs via eigen-packets of elliptic newforms. Let $\{f_i\}$ denote an orthogonal basis of eigenforms for $S_{2k-2}^{new}(p)$ and define,
\begin{equation} \label{new2def}
    \jkmn(f_i):=\text{span}\{\phi\in J_{k,m}^{cusp} \, | \,  T_J(q)\phi=\lambda_{f_i}(q)\phi, \text{ for primes } q\nmid m\},
\end{equation}
and define $\jkmtwo:= \bigoplus_{f_i} \jkmn(f_i)$ (see \cite[sec.~5.1]{man-ram}). 

\begin{lem} \label{jgap}
With the above notations, one has $\jkmone=\jkmtwo$.
\end{lem}
\begin{proof}
We use the fact that the Hecke operators $T_J(n)$ for $(n,m)=1$ are Hermitian on $\jkm$ and that $T_J(n)$ commutes with $V_l,U_d$ when $(n,lmd)=1$ (cf. \cite[\S~4, Corollary~1]{ez}). By abuse of notation, we denote by the same symbol $T_J(n)$ the Hecke operators on index $m$ and for lower index forms. Moreover, $T_J(n)$ leaves $\jkmone$ invariant (\cite[p.~138]{skoruppa1988jacobi}).

We first show that $\jkmone \subseteq \jkmtwo$. Since $\jkmone$ is invariant under the Hermitian $T_J(n)$ as noted above, we can decompose it into simultaneous eigenspaces for all the $T_J(n)$. From \cite[p.~138, (ii)]{skoruppa1988jacobi} we know that there exists a Hecke equivariant isomorphism $\mathscr{S}$ between $\jkmone$ and $M_{2k-2}^{-,new}(m)$ for all $m \ge 1$. By looking at the spanning set of newforms in $\jkmone$, therefore the eigenvalues occurring in the eigenspaces under $T_J(n)$ above must come (only) from those of $M_{2k-2}^{-,new}$ and we are done since this is precisely the definition of $\jkmtwo$.

Next we show $\jkmtwo \subseteq \jkmone$ by induction on $m$. This is clear when $m=1$. In fact, we already have equality  by \cite[Thm.~5.4]{ez} and standard results on the Shimura correspondence (cf. e.g., \cite[\S~5]{ez}). Next, assume that $J^{cusp,new2}_{k,m'}\subseteq J^{cusp,new1}_{k,m'}$ for any $m'<m$. Since $\jkmone \subseteq \jkmtwo$ for any $m$, we have  $J^{cusp,new2}_{k,m'}= J^{cusp,new1}_{k,m'}$ as our induction hypothesis.

Now consider a newform $f \in S_{2k-2}^{new}(m)$. Write $0 \neq \phi \in \jkmtwo(f)$ as (with the tuple $(l,d)$ as mentioned above)
\begin{align}
\phi = \sumn_t \phi_t + \sumn_{s=(l,d)} \psi_s |U_d\circ V_l,
\end{align}
where $\phi_t$ and $\psi_s$ run through a basis consisting of simultaneous eigenfunctions of $T_J(n)$, for all $n$ such that $(n,m)=1$ on $\jkmone$ and $J^{cusp,new1}_{k,m'}$ respectively, where $m'=m/ld^2$ as $(l,d) $ varies over as in \eqref{new1def}. If we apply $T_J(n)$ on both sides and equate coefficients, we would get for all $(n,m)=1$,
\begin{align} \label{p=1}
    \lambda_{f}(n) = \lambda_t(n) = \lambda_s(n)
\end{align}
for all $s,t$ for which $\psi_s, \phi_t$ are non-zero. Here obviously the quantities $\lambda_*(n)$ denote the eigenvalue of the relevant objects under $T_J(n)$.
But then, by induction hypothesis, $\lambda_s(n)$ is the eigenvalue of a level $m'$ ($<m$) elliptic newform (here we use multiplicity-one to get a unique such newform) so \eqref{p=1} is not possible again by multiplicity-one. Therefore, all the $\psi_s$ must be $0$ and $\jkmtwo \subseteq \jkmone $. 

Thus we have $\jkmone=\jkmtwo$ and  we call either of the spaces as $\jkmn$.
\end{proof}
For a fixed $i$, let $\{\phi_{ij}\}_j$ be an orthogonal basis for $\jkmn(f_i)$. In general, note that the multiplicity in the space of newforms in the case of Jacobi forms might be greater than one—as is the case sometimes with half-integral weight forms of level $4N$ with $N$ not square-free, and in many cases the aforementioned newspaces are isomorphic as Hecke modules. Then we define our object of interest in this subsection—the contribution of the newspace towards the size of the BK of $\jkmn$ -- by
\begin{equation}
    B_{k,m}^{new}(\tau,z;\tau_0,z_0):=\sumn_{i,j}\frac{\phi_{ij}(\tau,z)\overline{\phi_{ij}(\tau_0,z_0)}}{\lan \phi_{ij},\phi_{ij}\ran}.
\end{equation}

For any negative fundamental discriminant $D$ with $(D,m)=1$ and $D\equiv r^2\mod 4m$, we have a Waldspurger-type formula (cf. \cite{man-ram} \footnote{ There seems to be an extraneous index factor $i_{mM}$ in \cite[(29)]{man-ram}. The factor $R^2_D$ is also missing. }, cf. \eqref{eq:waldJacfin} for $m$ prime, probably \eqref{wald-p1} is true for any $m \ge 1$, but have not verified this. We will only use $m$ to be a prime.)
\begin{equation} \label{wald-p1}
    \sumn_{j}\frac{|c_{\phi_{ij}}(|D|)|^2}{\lan \phi_{ij},\phi_{ij}\ran} =\frac{\Gamma(k-1)|D|^{k-3/2}}{   R_D^2   2^{2k-3}\pi^{k-1}m^{k-2}} \frac{L(k-1,f_i\otimes \chi_D)}{ \lan f_i , f_i \ran}.
\end{equation}
Here $R_D=\# \{ t \bmod{2m} \mid  t^2 \equiv D \bmod{4m} \}$. E.g., when $m$ is a prime, $R_D = 1$ or $2$ depending on whether $D  \equiv 0,m^2 \bmod{4m}$ or otherwise.
If we choose $(n,r)$ such that $D=r^2-4mn$ is a fundamental discriminant with $(D,m)=1$, we quote from \cite[Thm.~5.7]{man-ram} that (by an argument verbatim same as in section \ref{seclbd}) 
\begin{align}
    \sup(J_{k,m}^{cusp,new})
    \gg  v_0^k e^{-4\pi m y^2/v} e^{-4\pi \frac{r^2-D } {4m} v_0}  \frac{ |D|^{k-3/2} \Gamma(k-1) }{(4\pi)^{k} m^{k-2}}\sumn_{f_i} \frac{L(k-1,f_i\otimes \chi_D)}{  \lan f_i , f_i \ran}.
\end{align}
Choosing $v_0=\frac{k}{4\pi n}$, $y=0$, we see that with $(D,m)=1$,
\begin{align}\label{eq:lowLval}
   \sup(J_{k,m}^{cusp,new})\gg \frac{|D|^{k-3/2} \Gamma(k-1) }{(4\pi)^{2k} m^{k-2} } \cdot \left(\frac{k}{e}\right)^k \sumn_{f_i \in S^{-,new}_{2k-2}(m) } \frac{L(k-1,f_i\otimes \chi_D)}{  \lan f_i , f_i \ran}.
\end{align}

\textit{From now on we will assume that $m =p$, a prime.}
On the one hand we have to maximize the lower bound and so should choose $D$ as large as possible in \eqref{eq:lowLval}. On the other hand, choosing $D$ very large compared to $p$ would destroy the asymptotic formula of the average of $L$-values in \eqref{eq:lowLval}. Moreover, the analytic conductor of the modular forms $f_i \otimes \chi_D$ is $k^2D^2p$ if $(D,p)=1$ whereas it is $k^2p^2$ if, say $D=-4p$. The former choice of larger conductor leads to some difficulty with the asymptotic analysis, see below.
So choosing $D=-4p$ seems to be a good option. This is also supported by the fact that we can always arrange $(n,r)$ such that $0<|r^2-4np| \le 4p$ by the division algorithm, so $D=-4p$ maximizes our choice from this point of view as well. 

However, there is a technical problem now – the Waldspurger-type formula as in \eqref{wald-p1} is not written down for index $p$ Jacobi forms when $p|D$. Fortunately, it is possible to put together various resources to overcome this, which we do below. The final result is that \eqref{eq:lowLval} holds true when $p|D$ as well. We will choose $D=-4p$, but have to distinguish between the cases $p \equiv1,3 \bmod{4}$.
This leads to the final lower bound, which is discussed in subsections~\ref{1mod4} and \ref{3mod4} respectively.

\subsection{Detour via Eichler-Zagier map, Baruch-Mao's formula when \texorpdfstring{$p|D$}{p|D}} \label{barmao}
Let $k$ be even. We use the following results.

(i) Let us put $S_{k-1/2}^{+,p^{\pm}}(4p):=\{g=\sum a_g(n)q^n \in S_{k-1/2}^{+}(4p) \colon a_g(n)=0 \text{ for } \left(\frac{-n}{p}\right)= \mp 1\}$.
Further, put $S_{k-1/2}^{+,p^\pm, new}(4p):= S_{k-1/2}^{+,p^\pm}(4p)\cap S_{k-1/2}^{+, new}(4p)$.
The generalized Eichler-Zagier (EZ) map $\mf Z_p \colon \jkp \rightarrow S_{k-1/2}^{+,p^+}(4p)$ is defined by means of the Fourier expansion as
\begin{align}
     \sum_{D,r}c_\phi(D,r)e(\frac{r^2-D}{4p}\tau + rz) \mapsto \sum_{D<0} \big( \sum_{r \bmod{2p} , r^2 \equiv D \bmod{4p}} c_\phi(D,r) \big) e(|D|\tau)
\end{align}
Note that $c_\phi(D,r)$ does not depend on $r$ as $\phi \in \jkpn$ (cf. \cite[Cor.~5.3]{man-ram}). Thus when restricted to the newspace,
$\mf Z_p \colon \jkpn \to S_{k-1/2}^{+,p^+,new}(4p)$ becomes a Hecke equivariant isomorphism (cf. \cite[Theorem~5.6]{ez},
\cite[Theorem~5.4]{man-ram}) and reads as\footnote{We point out an incorrectly stated definition given in \cite[Theorem~5.4]{man-ram} wherein the quantities $R_D$ were omitted.}
\begin{equation} \label{ezmap}
    \sumn_{D,r}c_\phi(D,r)e(\frac{r^2-D}{4p}\tau + rz) \mapsto \sumn_{D} R_Dc_\phi(D)e(|D|\tau),
\end{equation}
where the pair $(D,r)$ runs over integers such that $D<0$ and $D \equiv r^2 \bmod{4p}$ and in the second sum $D<0$, $D \equiv \square \bmod{4p}$ and $R_D$ is as before.

This implies that a strong multiplicity-one result holds for both of these spaces in view of (another) Hecke-equivariant isomorphism (via the trace formula) between $\jkmn$ and a $S^{new,-}_{2k-2}(m)$ (cf. \cite{sko-zag}, the minus refers to the root number being $-1$) for any $m \ge 1$.

(ii) Further from Kohnen's work \cite[Sec.~5, Theorem~2~(iii)]{ko1} the spaces $S_{k-1/2}^{+,p^\pm,new}(4p)$ (which can also be described as the $+1$ eigenspace under a suitable involution) and $S^{p^\pm,new}_{2k-2}(p)$ (consisting of the $+1$ eigenspace of the Fricke involution) are isomorphic as Hecke modules (under some linear combination of Shimura maps, say $\mathscr{S}$). In particular, multiplicity-one holds for these spaces.

(iii) (Baruch-Mao) We only invoke the details pertinent to our case at hand, for more details the reader is referred to \cite[Theorem~1.1]{bar-mao} and the example following it. We take some care in presenting the argument, as it is not very commonly available in the setting we need.

Let $g \in S_{k-1/2}^{+,p^+,new}(4p)$, and let $\mathscr S_D$ be the $D$-th Shimura map (see e.g., \cite{ko1}), one for each fundamental discriminant $D$, including $D=1$. If we put $f= \mathscr{S}_1(g)$, then by multiplicity-one on $S^{new}_{2k-2}(p) $ and the commutation of the maps $\mathscr{S}_D$ with all Hecke operators, one gets that $\mathscr{S}_D (g)=a_f(|D|) f$ for all $D$; and moreover under the isomorphism mentioned in (ii) above, $g$ corresponds to $f$ since the image of $g$ under the linear combination of Shimura maps leading to an isomorphism--as mentioned above--is still a constant multiple of $f$, if we re-scale this map. With this convention, the result of Baruch-Mao (which removes the assumption $(p,D)=1$ from \cite{ko2}) states that for all $D$ with $(-1)^{k-1}D>0$ (i.e., $D<0$ in our case)
\begin{equation}\label{allFD}
    \frac{|c_g(|D|)|^2}{\lan g, g\ran} = \frac{2 \Gamma(k-1)|D|^{k-3/2}}{3 \pi^{k-1}} \frac{L(k-1, f\otimes \chi_D)}{ \lan f , f \ran}. 
\end{equation}
The factor $3= [\Gamma_0(4p) \colon \Gamma_0(p)]$ comes because our Petersson norms are not normalised by volume.
The reader might notice that we have not mentioned any local conditions, viz. $(D/p)=\eta_p(f)$ because certainly the result holds for these discriminants (even when $p|D$) and it is also known that both sides of \eqref{allFD} are $0$ when $(D/p)=-\eta_p(f)$, see \cite[Remark after Corollary~1]{ko2} and the second part of the example following \cite[Theorem~1.1]{bar-mao}. Here we use that $D<0$ and $k$ is even. Moreover, $\eta_p(f)$ is the eigenvalue of the Fricke involution at $p$.

\subsection{Relation between Petersson norms} \label{pet-relns}
Let $g\in  S_{k-1/2}^{+,p, new}(4p)$ be the image of $\phi \in \jkpn$: $g= \mf Z_p(\phi)$ as in \eqref{ezmap}. Put $f = \mathscr{S}_1(g) \in S^{p^+,new}_{2k-2}(p)$. Then $g,f$ are uniquely determined, and  $R_D c_\phi(D) = c_g(|D|)$. Thus \eqref{allFD} holds with $c_g(|D|)$ replaced by $R_Dc_\phi(|D|)$. The only thing that is missing now is a relation between $\lan \phi, \phi \ran$ and $\lan g, g \ran $. This deep result is available from the thesis of Skoruppa \cite[Satz 4.1]{skoruppa}, but perhaps is not well-known, so we give a brief account by translating it to our situation. Let us first state the result and then explain the notation. In the passing, note that it is not enough to quote a similar inner product relation from \cite{br-ag} since they deal with index $1$ Jacobi forms of level $p$, whereas we have index $p$ to start with. Also it seems difficult to generalize the method in loc. cit. via residues and Rankin-Selberg method to our situation because of complicated congruence relations on the Fourier expansions. We therefore take the path via Eichler-Zagier maps.

Namely, in the notation of \cite[p.~94,~(4)]{skoruppa}, if $\phi,\psi \in J^{1,f}_{k,m}$ (a certain subspace of $J_{k,m}$, see below) and $\chi \bmod F$ is a Dirichlet character, then
\begin{align}\label{eq:petrelsk}
    \lan \phi, \psi \ran =  [\SL{2}{\z} \colon \Gamma_0(4m)]^{-1} \, c^\chi_{k,m} \lan \mf Z^\chi_{k,m} (\phi), \mf Z^\chi_{k,m} (\psi) \ran  . 
\end{align}
Here  $c^\chi_{k,m} = (4m)^{k-1} \omega^f_m(1)\cdot$ $|\{\rho\mod 2m\ | \ (\rho, FQ) = 1\} |^{-1}$, where $F|2m$ and $Q$ is the largest square dividing $m$ and $\mf Z^\chi_{k,m}$ is the generalized Eichler-Zagier map. For the constant $\omega^f_m(1)$, when $m=p$ a prime, see below.

Now note that for us the Dirichlet character $\chi \bmod F$ is $1$, $F=1$, $m=p$ (and $Q=1$ in loc. cit.). $J_{k,p}$ (more generally for any index $m$) can be decomposed as a direct sum of certain canonical subspaces $J^{1,f}_{k,p}$ ($f|p$, $f$ square-free). Now from \cite[Satz 2.3]{skoruppa}
it follows that (since $k$ is even) $J_{k,p}=J_{k,p}^{1,1}$ and the other component is $0$. Thus for any $\phi,\psi \in J_{k,p}$, from \eqref{eq:petrelsk} we get the following identity relating the Petersson norms if we note that the quantity $\omega^f_p(1)$ occurring in the definition of $c^1_{k,m}$ equals $p+1$ (cf. \cite[p.~23,~(13) and p.~39,~(29)]{skoruppa}). Note also that the Petersson norms in \cite{skoruppa} are normalized by volume, whereas ours are not. The factor $p+1$ above thus cancels out and we get:
\begin{equation} \label{phig}
   \lan \phi, \phi \ran = 3^{-1} \cdot 2^{2k-2} p^{k-2} \lan g, g\ran. 
\end{equation}
In conclusion, we can write the following Waldspurger's formula for any fundamental discriminant $D<0$ and any $\phi \in \jkpn$, with $f,g$ as in this subsection.

\begin{equation}\label{eq:waldJacfin}
    \frac{|c_\phi(|D|)|^2}{\lan \phi, \phi\ran} = \frac{ \Gamma(k-1)|D|^{k-3/2}}{   R_D^2  2^{2k-3}\pi^{k-1} p^{k-2}} \frac{L(k-1, f\otimes \chi_D)}{ \lan f , f \ran}.
\end{equation}
Thus summing \eqref{eq:waldJacfin} over the corresponding (isomorphic) newspaces, \eqref{eq:lowLval} holds for all fundamental discriminants $D<0$. For the lower bound of $\sup(J_{k,p}^{cusp,new}$) we will choose $D=-4p$ but for the upper bound we need \eqref{eq:waldJacfin} for all $D<0$. This leads us to study an asymptotic formula of the averages of central $L$-values, which is the content of the next section.

\subsection{Back to the size of the new-space} \label{new-concl}
Recall the following relation between $\lan f, f\ran$ and $L(1,\mrm{sym}^2)$.
If $f\in S_{2k-2}(p)$ be a newform, then we have (cf. \cite[(2.36),(3.14)]{ILS})
\begin{equation} \label{pet-sym}
12  \zeta(2) \lan f,f\ran =  \Gamma(2k-2) p (4 \pi)^{3-2k}   L(1, \mrm{sym}^2 f).
\end{equation}
Moreover the estimate in \eqref{eq:lowLval} is now valid for $p|D$ also. We now distinguish two cases.

\subsubsection{$p \equiv 1 \bmod 4$:} \label{1mod4}
We now choose $D=-4p$ when $p \equiv 1 \bmod 4$ so that $D$ is a fundamental discriminant. Note that in our notation, $S^{p^+,new}_{2k-2}(p) = S^{new,-}_{2k-2}(p)$ when $k$ is even.
With these at hand, we now invoke \thmref{blo-p} and \corref{*-}, use the conversion in \eqref{pet-sym} (see \eqref{eq:asymnormp|D} for the asymptotic formula) and substitute back in \eqref{eq:lowLval} to see that
\begin{align} \label{p1lbd}
   \sup(J_{k,p}^{cusp,new})\gg \frac{(4p)^{k-3/2} \Gamma(k-1) }{(4\pi)^{2k} p^{k-2} }\cdot \left(\frac{k}{e}\right)^k  \frac{(4\pi)^{2k-3}}{\Gamma(2k-3)} \gg
   k^2 p^{1/2}.
\end{align}

\subsubsection{$p \equiv 3 \bmod 4$:} \label{3mod4}
Recall that $D<0$ and $D \equiv \square \bmod{4p}$.
When $p \equiv 3 \bmod 4$ we still take $D=-4p$.
In this case $D$ is not fundamental, but recall that
for any discriminant $D$, if we uniquely write $D= n^2 D_0$, where $D_0$ is a fundamental discriminant, then for any newform $f$ under our consideration, we have (cf. e.g., \cite[p.~241,~(11)]{ko2})
\begin{equation}\label{nonfund}
    c_g(|D|)= c_g(|D_0|)\sumn_{d|n, (d,p)=1}\mu(d) \left(\frac{D}{d}\right) d^{k-1} \lambda_f(n/d).
\end{equation}
For us $D_0=-p$ and $n=2$. It is easy to see that only $d=1$ survives in the sum over $d|n$ above, and the sum equals $\lambda_f(4)$.
Thus here we would require an asymptotic formula for
\begin{align}
    \sumn_f  \lambda_f(4) \, L(1/2, f \otimes \chi_{-p}) \lan f,f \ran^{-1}
\end{align}
with the extra eigenvalue included. The result is given below, whose proof is however deferred to until the end of the next section, for better exposition. See subsection~\ref{moretwist-proof} and \propref{moretwist}. In particular, the same lower bound as in \eqref{p1lbd} holds true for $p \equiv 3 \bmod{4}$ as well.

\subsubsection{Bounding the size of newspace via central $L$-values}
For any discriminant $D$, from \eqref{nonfund}, we have
    $|c_g(|D|)|^2\ll n^{2k-3+\epsilon}|c_g(|D_0|)|^2$.
Now using the Waldspurger/Baruch-Mao's formula for $|c_g(|D_0|)|^2$, we get the following bound on the size of index-newspace which can be treated as in section~\ref{jacobi-supnorm}. We believe that the bound below could be useful for future investigations.

\begin{align}
    \sup(J_{k,p}^{cusp,new})\ll& v^{k}e^{-4\pi py^2/v}\frac{\Gamma(k-1)}{2^{2k-3} \pi^{k-1}p^{k-2}} \Big(  \sum_{D,r, D=n^2D_0} D^{\frac{k}{2}-\frac{3}{4}+\epsilon} \big( \sum_{f}  \frac{L(k-1, f\otimes \chi_{D_0})}{ \lan f , f \ran} \big)^{\frac{1}{2}} \n\\ 
    & \times\exp\left(-\frac{\pi v}{2m}D\right)\exp(-2\pi(ry+\frac{r^2v}{4m}))\Big)^2.
\end{align}
This relation produces the conjectured upper bound for the new-space in certain regions (viz. $v\ll k$ and $v\gg kp^{1+\epsilon}$). But we would still require the arguments from the geometric side of the BK to obtain the conjectured bounds in the remaining regions, which at the moment gives the same bound $k^2p^2$. Thus we do not carry these out and only mention the global upper bound.
For convenience, we state the results on the new-space as follows.
\begin{prop} \label{new-prop}
For $k$ even we have $k^2 p^{1/2} \ll \sup(\jkpn) \ll k^2p^2 $. Assuming \conjectureRef{jacobi-conj}, one has $\sup(\jkpn) \asymp \sup(\jkp)$.
\end{prop}

\section{A first moment of central \texorpdfstring{$L$}{L}-values} \label{l1/2}
In this section we prove \thmref{blo-p}. We first discuss the root numbers of twists  and then appeal to a variant of the Petersson formula on the space of newforms.
Let us note that our formula (see \thmref{blo-p}) is uniform with respect to the parameters $D,p,k$, and this is important for the sup-norm problem. We have restricted to level $p$ because the part of the argument from the Jacobi forms side seems reasonably tractable only in index $p$. For higher indices (cf. e.g., \eqref{phig}) much more complicated scenario ensues.
Perhaps with more care, one would be able to generalise \thmref{blo-p} to any square-free level.


\subsection{Root number} For a newform $f$ of weight $2k-2$ and level $p$, the root number of $f$ is given by $\epsilon_f = \eta_p(f) (-1)^{k-1}$ where $\eta_p(f)$ denotes the Atkin-Lehner eigenvalue of $f$ at $p$ (see \cite[Section 5.11]{iwaniec2004analytic}). Moreover, one has the relation (see e.g., \cite[eq.~(2.23)]{ILS}) $\eta_p(f)= -\lambda_f(p) p^{1/2}$.
The root number $\epsilon_{f \otimes \chi_D}$ of the L-function attached to $f \otimes \chi_D$ can then be calculated as (see \cite[Section 8.1]{petrow2019generalized}): 
\begin{equation}
    \epsilon_{f \otimes \chi_D}= \chi_D(-r) \mu(q') q'^{1/2} \lambda_f(q') \epsilon_f.
\end{equation}
Here we write $p=rq'$ with $(r,D)=1$ and $q'|\mrm{rad}(D)$, with $\mrm{rad}(D)$ being the product of the primes dividing $D$. Therefore, when $(p,D)=1$, we have 
\begin{equation}\label{pDcoprime}
    \epsilon_{f \otimes \chi_D}= \chi_D(-p) \epsilon_f = \sgn(D) \chi_D(p) \epsilon_f = -(-1)^{k-1} \sgn(D) \chi_D(p) \lambda_f(p) p^{1/2};
\end{equation} 
and when $p|D$, 
\begin{equation}
    \epsilon_{f \otimes \chi_D}= -\chi_D(-1) \lambda_f(p) p^{1/2} \epsilon_f =  \sgn(D) \eta_p(f) \epsilon(f) =(-1)^{k-1} \sgn(D) .
\end{equation}

In the case of our interest, i.e., $k$ even and $D<0$ with $D\equiv \square\mod 4p$, we have $\chi_D(-1)=-1, \chi_D(p)=1$ and $\epsilon_f= \lambda_f(p) p^{1/2}$. Thus for the sup-norm problem we would be interested in,
\begin{equation}\label{eq:rootno}
    \epsilon_{f \otimes \chi_D}=
    \begin{cases}
    -\lambda_f(p) p^{1/2} & \text{ if } (p,D)=1 \text{ and } D\equiv \square\mod 4p;\\
    1 & \text { if } p|D.
    \end{cases}
\end{equation}

\begin{rmk}
For $f \in B^*_{2k-2}(p)$, the twisted central $L$-value
$L(1/2, f\otimes \chi_D)=0$ (for sign reasons) in the following cases:
\begin{enumerate}
\item $(-1)^{k-1} \sgn(D)=-1$ such that $p|D$,
\item $\sgn(D) \chi_D(p) \epsilon_f=-1$ and $(p,D)=1$.
\end{enumerate}
In the cases where the root number does not enforce vanishing of all the central $L$-values, our asymptotic formula holds true.
\end{rmk}
\begin{cor} \label{*-}
 The sums $\sum_{f\in B_{2k-2} ^*(p)}$ in \eqref{blo1}, \eqref{blo2} of \thmref{blo-p} are actually supported on the set $B_{2k-2}^{-}(p) \subset B_{2k-2}^{*}(p)$ consisting of newforms with root number $-1$.
\end{cor}
This is clear from the above discussion – the $L$-values $L(1/2, f\otimes \chi_D)$ vanish if the root number of $f$ is $+1$, given the other conditions.

\subsection{Proof of \thmref{blo-p}}
For better understanding, the proof is subdivided into various subsections below.

From the relation \eqref{pet-sym}, it is enough to get an asymptotic for $\sum_{f_i} L(1/2, f_i\otimes \chi_D)  \lan f_i,f_i\ran^{-1} $. Using the approximate functional equation for $L(1/2, f_i\otimes \chi_D)$, we see that $\sum_{f_i} L(1/2, f_i\otimes \chi_D)  \lan f_i,f_i\ran^{-1} $
\begin{equation} \label{levelp-approx}
 = \sum_{n} \frac{\chi_D(n)}{ n^{1/2} } V(n/q^{1/2})\sum_{f_i}\frac{\lambda_{f_i}(n)}{ \lan f_i,f_i\ran } + \sum_{n} \frac{ \chi_D(n) }{ n^{1/2} } V(n/q^{1/2})\sum_{f_i} \epsilon_{f_i \otimes \chi_D} \frac{\overline{\lambda_{f_i}(n)}}{ \lan f_i,f_i\ran },
\end{equation}
where the level $q$ of the twists satisfy $q= pD^2$ if $(p,D)=1$ and $q=D^2$ if $p|D$ (see e.g., \cite[p.~311, para~3]{petrow2019generalized}). We denote by $\mf q:=k^2q$, the analytic conductor of $L(s, f\otimes \chi_D)$. Further note that $\lambda_{f_i}(n)$ is real for all $n$.

Denote the two sums in \eqref{levelp-approx} as $S$ and $T$ respectively.
When $(p,D)=1$, the analysis of the sums $S$ and $T$ will be somewhat different from each other since the root number is not constant (cf. \eqref{eq:rootno}). We will consider $S$ in subsection~\ref{S-sum} and $T$ in subsection~\ref{T-sum} below.

For both $S,T$ we will use a variant of `Petersson formula' for the newspace of level $p$, see \cite{ILS}, \cite[Theorem 3]{petrow2019generalized}. In the notation of loc. cit., we put
\begin{align}
\sumn_{f_i}\lambda_{f_i}(n)\lan f_i,f_i\ran^{-1}=\frac{(4\pi)^{2k-3}}{\Gamma(2k-3)} \Delta^*_p(1,n).
\end{align}
From \cite[Theorem 3]{petrow2019generalized}, we have the following formula for the quantity $\Delta^*_p(1,n)$.
\begin{equation}\label{petnew}
    \Delta^*_p(1,n)= \sum_{LM=p}\frac{\mu(L)}{\nu(L)}\sum_{\ell|L^\infty} \frac{\ell}{\nu(\ell)^2}\sum_{d_1, d_2|\ell}c_\ell(d_1)c_\ell(d_1) \sum_{v|(n,L)}\frac{v\ \mu(v)}{\nu(v)}\sum_{b|(\frac{n}{v}, v)}\sum_{e|(d_2, \frac{n}{b^2})}\Delta_M(d_1, \tfrac{nd_2}{e^2b^2}),
\end{equation}
where the notations are as in \cite{petrow2019generalized}. In particular $\mu(*)$ is the M\"obius function, $\nu(p)=p+1, \nu(1)=1$ and
$\Delta_M(s,t)=\delta(s, t) + 2 \pi i^{-2k+2} \sum_{M|c} c^{-1} S(s,t;c) J_{2k-3}(4 \pi \sqrt{st}/c) $, where $S(*,*;c)$ and $J_{2k-3}(*)$ denote the classical Kloosterman sum and Bessel function respectively.

\subsubsection{Analysis of $S$} \label{S-sum}
\begin{equation}
  \sum_n \frac{\chi_D(n)}{n^{1/2}} V(n/q^{1/2}) \sum_{f_i}\lambda_{f_i}(n)\lan f_i,f_i\ran^{-1} = \frac{(4\pi)^{2k-3}}{\Gamma(2k-3)} \sum_n \frac{\chi_D(n)}{n^{1/2}} V(n/q^{1/2}) \Delta^*_p(1,n).
\end{equation}
Denote by $\mc S_1(n)$ and $\mc S_2(n)$ the terms corresponding to $L=1$ and $L=p$ in \eqref{petnew} respectively. Then
\begin{equation} \label{sn-def}
    S= \frac{(4\pi)^{2k-3}}{\Gamma(2k-3)}\sumn_{n} \frac{\chi_D(n)}{n^{1/2}} V(n/q^{1/2}) (\mc S_1(n) +\mc S_2(n)).
\end{equation}
Using the decay of $V(y)$ (see \cite[Proposition 5.4]{iwaniec2004analytic}),  the sum over $n$ can be truncated at $n\ll q^{1/2+\epsilon} k^{1+\epsilon}=\mf q^{1/2+\epsilon}$ with negligible error in $k$, $p$ and $D$.

From \eqref{petnew} and from the fact that $c_1(1)=1$, we see that $\mc S_1(n)= \Delta_p(1,n)$,
where $\Delta_p(1,n)$ is as defined above.
From \cite[Corollary 2.2]{ILS}, we have the following asymptotic result:
\begin{equation}\label{eq:S1}
   \mc S_1(n)= \Delta_p(1,n) = \delta(1,n) +O\Big(\frac{n^{3/8}}{(n,p)^{1/2}p^{5/4-\epsilon} k^{13/12}}\Big).
\end{equation}
Write $\mc S_1(n)=  \mc M_1(n) +\mc E_1^{k}(n)$ corresponding to the main term and the error term in \eqref{eq:S1}. Here we write $\mc E_1^{k}(n)$ to stress the dependence on $k$. It is understood that all the quantities depend on $p$.
Then we find that
\begin{align}\label{eq:M1n}
    \sumn_n \frac{\chi_D(n)}{n^{1/2}} V(n/q^{1/2}) \mc M_1(n) = V(1/q^{1/2})= 1+O(\mf q^{-\alpha/2}),
\end{align}
where we have used the asymptotic $V(y)=1+O((y k^{-1})^\alpha)$ with $0< \alpha \le k/6$ (see \cite[Proposition 5.4]{iwaniec2004analytic}).

For the error term $\mc E_1^k(n)$ we use that $V(y)$ is bounded and thus in both the cases, we get
\begin{align}\label{eq:E1n}
    \sum_n \frac{\chi_D(n)}{n^{1/2}} V(n/q^{1/2})\mc E_1^k(n) &\ll \sum_{n\ll \mf q^{1/2+\epsilon}} \frac{1}{n^{1/2}} \frac{n^{3/8}}{(n,p)^{1/2}p^{5/4-\epsilon} k^{13/12}} \ll p^{-\frac{5}{4}+\epsilon} k^{-\frac{5}{24}+\epsilon} q^{\frac{7}{16}+\epsilon}.
\end{align}

Now consider $\mc S_2(n)$. Again from \eqref{petnew} we have
\begin{equation}\label{eq:S2Del}
    \mc S_2(n)=  \frac{-1}{p+1}\sum_{t\ge 0} \frac{p^t}{(p+1)^{2t}}\sum_{d_1,d_2|p^t}c_{p^t}(d_1)c_{p^t}(d_2) \sum_{v|(n,p)}\frac{v\ \mu(v)}{\nu(v)}\sum_{b|(\frac{n}{v}, v)}\sum_{e|(d_2, \frac{n}{b^2})}\Delta_1(d_1, \tfrac{nd_2}{e^2b^2}).
\end{equation}
We use the following asymptotic for $\Delta_1(m,n)$ (see \cite[Corollary 2.2]{ILS}).
\begin{equation}\label{eq:Petbndmn}
    \Delta_1(m,n)= \delta(m,n) +O((mn)^{3/8+\epsilon} k^{-13/12}).
\end{equation}
Corresponding to the main and error terms, we write $\mc S_2(n)=  \mc M_2(n) +\mc E_2^{k}(n)$.
We first evaluate the main term $\mc M_2(n)$.
\begin{equation}\label{eq:mp}
\mc M_2(n)=\frac{-1}{p+1}\sum_{t\ge 0} \frac{p^t}{(p+1)^{2t}}\sum_{d_1,d_2|p^t}c_{p^t}(d_1)c_{p^t}(d_2)  \sum_{v|(n,p)}\frac{v\ \mu(v)}{\nu(v)}\sum_{b|(\frac{n}{v}, v)}\sum_{e|(d_2, \frac{n}{b^2})}\delta(d_1, \tfrac{nd_2}{e^2b^2}).
\end{equation}
Note here that, since $d_1, d_2$ are powers of $p$ in \eqref{eq:mp}, we have \begin{equation}\label{eq:M2bnd}
\mc M_2(n)\begin{cases}
    =0 & \text{ if } n \text{ is not a power of } p;\\
    \ll 1/p & \text{ if } n \text{ is a power of } p.
    \end{cases}
\end{equation}

\textit{$\bullet$ The case $p|D$:} In this case, the only contribution comes when $(n,p)=1$, since $\chi_D(n)=0$ for $(n,p)>1$. When $(n,p)=1$, the contribution from the sum over $v$ is $\delta(d_1,nd_2)$. Since both $d_1$ and $d_2$ are divisors of $p$, $\delta(d_1,nd_2)=\delta(1,n)\delta(d_1,d_2)$. Thus when $p|D$, the contribution from $\mc M_2(n)$ is
\begin{equation}\label{eq:M2np|D}
    \sumn_n \frac{\chi_D(n)}{n^{1/2}} V(n/q^{1/2}) \mc M_2(n) = C(p) + O(p^{-1}\mf q^{-\alpha/2}),
\end{equation}
if we recall the definition (and a bound) of $C(p)$ from \eqref{cp}.

\textit{$\bullet$ The case $(p,D)=1$:} When $(n,p)=1$, the contribution from the sum over $v$ is $\delta(d_1,nd_2)=\delta(1,n)\delta(d_1,d_2)$.
Next, consider the case $(n,p)=p$. In this case, the contribution from the $v$ sum is
\begin{align}\label{vsumM2}
    (1-\frac{p}{p+1})\sum_{e|(d_2, n)} \delta(d_1, nd_2/e^2) -\frac{p}{p+1}\sum_{e|(d_2, n/p^2)}\delta(d_1, nd_2/e^2p^2).
\end{align}
Since $d_1=p^{j_1}$, $d_2=p^{j_2}$, for some $j_1,j_2\le t$,  the contribution comes only when $n$ is a power of $p$. Also, $c_{p^t}(p^i)=0$ when $t\not \equiv i \mod 2$. Therefore, the sum survives only when $n$ is an even power of $p$, say $n=p^{2j}$ for some $j\ge 1$. Now, multiplying by $n^{-1/2}$ and summing over $n$ (that is $j\ge 1$), we get the following telescoping sum (noting that $n$ sum is truncated at $\ll\mf q^{1/2+\epsilon}$):
\begin{align}\label{eq:M2analysis}
    &\sum_{1\le j\ll \log_p\mf q^{\frac{1}{4}+\epsilon}}\Big(\frac{1}{p^j}\sum_{e|(d_2, p^{2j})} \delta(d_1,\frac{p^{2j}d_2}{e^2})-\frac{1}{p^{j-1}} \sum_{e|(d_2, p^{2(j-1)})}\delta(d_1, \frac{p^{2(j-1)}d_2}{e^2})\Big)\n \\
    & =-\delta(d_1, d_2)+O(\mf q^{-\frac{1}{4}-\epsilon}).
\end{align}
Now observe that $\chi_D(n)=1$ for all $n=p^{2j}$ and $V(y) = 1+O((yk^{-1})^\alpha)$.
Thus for any $\alpha<1/2$, using \eqref{eq:M2bnd} and combining it with the contribution from $(n, p)=1$ (cf. \eqref{eq:M2analysis}), we get 
\begin{equation}\label{eq:M2ncoprime}
    \sumn_n \frac{\chi_D(n)}{n^{1/2}} V(n/q^{1/2}) \mc M_2(n) =   \frac{p}{p+1}C(p)  + O(p^{-1}\mf q^{-\alpha/2}+ p^{-1} \mf q^{-\frac{1}{4}-\epsilon}).
\end{equation}
Next, we estimate the error term $\mc E_2^k(n)$ of $\mc S_2(n)$. 
\begin{align} \label{peter-finiteS}
    \mc E_2^k(n) &\ll   \frac{k^{-\frac{13}{12}}}{p+1} \sum_{t\ge 0} \frac{p^t}{(p+1)^{2t}} \sum_{d_1,d_2|p^t} c_{p^t} (d_1) c_{p^t}(d_2) \sum_{v|(n,p)} \sum_{b | (n/v,v)} \sum_{e_2| (d_2,n/b^2)} \Big(\frac{n d_1 d_2}{b^2e_2^2}\Big)^{3/8+\epsilon}\n \\
    &\ll \frac{n^{\frac{3}{8}+\epsilon}}{k^{\frac{13}{12}}\ p} \sum_{t\ge 0} \frac{1}{p^{t}} \sum_{d_1,d_2|p^t} c_{p^t} (d_1) c_{p^t}(d_2) (d_1d_2)^{3/8+\epsilon}\ll \frac{n^{\frac{3}{8}+\epsilon}}{k^{\frac{13}{12}}\ p} \sum_{t\ge 0} \frac{1}{p^{t(1/4-\epsilon)}}\ll \frac{n^{\frac{3}{8}+\epsilon}}{k^{\frac{13}{12}}\ p} .
\end{align}
Here we use the bound from \cite[Lemma 2]{petrow2019generalized}. Therefore, the total contribution from the errors $\mc E_2^k(n)$ as $n$ is summed is:
\begin{align}\label{eq:E2n}
   \ll \sum_n \frac{\chi_D(n)}{n^{1/2}} V(n/q^{1/2})\mc E_2^k(n) &\ll k^{-13/12} p^{-1} \sum_{n\ll \mf q^{1/2+\epsilon}} \frac{1}{n^{1/8-\epsilon}} \ll p^{-1} k^{-5/24+\epsilon}  q^{7/16+\epsilon}.
\end{align}
Combining the contributions from $\mc S_1(n)$ and $\mc S_2(n)$, we get
\begin{enumerate}
    \item When $(D,p)=1$, putting together the contributions from \eqref{eq:M1n}, \eqref{eq:E1n}, \eqref{eq:M2ncoprime} and \eqref{eq:E2n},  for any $\alpha<1/2$ we get,
\begin{equation} \label{firstsum}
    S=\frac{(4\pi)^{2k-3}}{\Gamma(2k-3)}\Big(1+ \frac{p}{p+1}C(p)+ O( k^{-\alpha}p^{-\alpha/2}D^{-\alpha}+p^{-9/16+\epsilon} k^{-5/24} D^{7/8+\epsilon})\Big).
\end{equation}
\item When $p|D$, from \eqref{eq:M1n}, \eqref{eq:E1n} \eqref{eq:M2np|D} and \eqref{eq:E2n} we get
\begin{equation} \label{firstsump|D}
    S= \frac{(4\pi)^{2k-3}}{\Gamma(2k-3)}\left(1 + C(p) + O(k^{-\alpha}D^{-\alpha}+p^{-5/4+\epsilon} k^{-13/12} D^{7/8+\epsilon})\right).
\end{equation}
\end{enumerate}

\subsubsection{Analysis of $T$} \label{T-sum}
Now we estimate the dual sum for the case $(p,D)=1$. Since $\lambda(p)\lambda(n)= \lambda(pn)$ and from \eqref{pDcoprime}, we get
\begin{align} \label{tn-def}
     T= &- (-1)^{k-1}\sgn(D) \chi_D(p)  p^{1/2} \frac{(4\pi)^{2k-3}}{\Gamma(2k-3)}\sumn_n \frac{\chi_D(n)}{n^{1/2}} V(n/q^{1/2}) \Delta^*_p(1,pn) \n\\
     &= - (-1)^{k-1}\sgn(D) \chi_D(p)  p^{1/2}\frac{(4\pi)^{2k-3}}{\Gamma(2k-3)}\sumn_n \frac{\chi_D(n)}{n^{1/2}} V(n/q^{1/2})(\mc T_1(n) + \mc T_2(n)),
\end{align}
where $\mc T_1(n)$ and $\mc T_2(n)$ are the terms corresponding to $L=1$ and $L=p$ in $\Delta^*_p(1,pn)$ respectively. We now treat $\mc T_1(n)$ and $\mc T_2(n)$ one by one. 

We clearly have $\mc T_i(n)= \mc S_i(pn)$ for $i=1,2$. Thus, from \eqref{eq:S1}, we have
\begin{equation} \label{t1n}
    \sumn_n \frac{\chi_D(n)}{n^{1/2}} V(n/q^{1/2}) \mc T_1(n)\ll p^{-\frac{15}{8}+\epsilon} k^{-\frac{5}{24}+\epsilon} q^{\frac{7}{16}+\epsilon}.
\end{equation}
For $\mc T_2(n)$, first note from \eqref{eq:S2Del} (replacing $n$ by $np$) that
\begin{equation}\label{eq:S2*Del}
    \mc T_2(n)=  \frac{-1}{p+1}\sum_{t\ge 0} \frac{p^t}{(p+1)^{2t}}\sum_{d_1,d_2|p^t}c_{p^t}(d_1)c_{p^t}(d_2) \sum_{v|p}\frac{v\ \mu(v)}{\nu(v)}\sum_{b|(\frac{np}{v}, v)}\sum_{e|(d_2, \frac{np}{b^2})}\Delta_1(d_1, \tfrac{npd_2}{e^2b^2}).
\end{equation}
The contribution from $v$ sum is
\begin{align}
    &\sum_{e|(d_2, np)}\Delta_1(d_1, \tfrac{npd_2}{e^2}) -\frac{p}{p+1} \Big( \sum_{e|(d_2, np)}\Delta_1(d_1, \tfrac{npd_2}{e^2}) + \sum_{e|(d_2, n/p)} \Delta_1(d_1, \tfrac{npd_2}{pe^2} ) \Big)\n\\
    &= \frac{1}{p+1} \sum_{e|(d_2, np)}\Delta_1(d_1, \tfrac{npd_2}{e^2}) - \frac{p}{p+1} \sum_{e|(d_2, n/p)} \Delta_1(d_1, \tfrac{nd_2}{pe^2}). \label{ppower}
\end{align}
Now in \eqref{eq:S2*Del} write $\mc T_2(n)=\mc M_2^*(n)+\mc E_2^{k, *}(n)$ corresponding to main and error terms of \eqref{eq:Petbndmn}. First, recall that $\mc M_2^*(n)= \mc M_2(pn)$. Using the same analysis as in \eqref{vsumM2} and \eqref{eq:M2analysis}, we see that the only contribution of the $n$ sum, where we have put $np=p^{2j}$, is from the telescoping sum  up to an error $O(\mf q^{-\frac{1}{4}-\epsilon})$ as we observed after \eqref{eq:M2analysis}:
\begin{align}\label{eq:M2*analysis}
    &\chi_D(p)p^{1/2}\sum_{1\le j\ll \log_p\mf q^{\frac{1}{4}+\epsilon}+\frac{1}{2}}\Big(\frac{1}{p^j}\sum_{e|(d_2, p^{2j})} \delta(d_1,\frac{p^{2j}d_2}{e^2})-\frac{1}{p^{j-1}} \sum_{e|(d_2, p^{2(j-1)})}\delta(d_1, \frac{p^{2(j-1)}d_2}{e^2})\Big)\n\\
    &=-\chi_D(p)p^{1/2}\delta(d_1, d_2)+O(\mf q^{-\frac{1}{4}-\epsilon}).
\end{align}
Therefore, we see that (for $\alpha<1/2$) $\sum_n \frac{\chi_D(n)}{n^{1/2}} V(n/q^{1/2}) \mc M_2^*(n) $ is (recall the definition of $C(p)$ from \eqref{cp})
\begin{align}
    &=\frac{  \chi_D(p) p^{1/2} }{(p+1)^2}\sum_{t\ge 0} \frac{p^t}{(p+1)^{2t}}\sum_{d|p^t}c_{p^t}(d)^2 + O(p^{-1}\mf q^{-\alpha/2}+ \mf q^{-\frac{1}{4}-\epsilon})\n \\
    &=  \frac{ -\chi_D(p) p^{1/2} }{(p+1)}C(p) + O(p^{-1}\mf q^{-\alpha/2} + p^{-1} \mf q^{-\frac{1}{4}-\epsilon}). \label{m2t2}
\end{align}
For the error term of $\mc T_2(n)$ similar estimates as in \eqref{peter-finiteS} follow and we get that $\mc E_2^{k, *}(n) = \mc E_2^{k}(pn) $ is
\begin{align} \label{peter-finiteT}
   &\ll   \frac{k^{-\frac{13}{12}}}{p+1 } \sum_{t\ge 0} \frac{p^t}{(p+1)^{2t}} \sum_{d_1,d_2|p^t} c_{p^t} (d_1) c_{p^t}(d_2) \Big(\frac{1}{p} \sum_{e|(d_2, np)} \big(\tfrac{np d_1d_2}{e^2}\big)^{\frac{3}{8}+\epsilon} + \sum_{e|(d_2, n/p)}  \big(\tfrac{nd_1d_2}{pe^2}\big)^{\frac{3}{8}+\epsilon}\Big) \n \\
    &\ll \frac{n^{\frac{3}{8}+\epsilon}}{k^{\frac{13}{12}}\ p^{\frac{11}{8}-\epsilon}} \sum_{t\ge 0} \frac{1}{p^{t}} \sum_{d_1,d_2|p^t} c_{p^t} (d_1) c_{p^t}(d_2) (d_1d_2)^{\frac{3}{8}+\epsilon}\ll \frac{n^{\frac{3}{8}+\epsilon}}{k^{\frac{13}{12}}\ p^{\frac{11}{8}-\epsilon}} \sum_{t\ge 0} \frac{1}{p^{t(\frac{1}{4}-\epsilon)}}\ll \frac{n^{\frac{3}{8}+\epsilon}}{k^{\frac{13}{12}}\ p^{\frac{11}{8}-\epsilon}} .
\end{align}
Thus
\begin{equation} \label{e2t2}
    \sum_n \frac{\chi_D(n)}{n^{1/2}} V(n/q^{1/2})\mc E_2^{k,*}(n) \ll k^{-\frac{13}{12}} p^{-\frac{11}{8}+\epsilon} \sum_{n\ll \mf q^{1/2+\epsilon}} \frac{1}{n^{1/8-\epsilon}} \ll p^{-11/8+\epsilon} k^{-5/24+\epsilon}  q^{7/16+\epsilon}.
\end{equation}
Putting together \eqref{m2t2}, \eqref{e2t2} we see that (for $\alpha<1/2$)
\begin{equation}\label{eq:T}
    T= \frac{(4\pi)^{2k-3}}{\Gamma(2k-3)}\left(\frac{ (-1)^{k-1}\sgn(D)  p}{p+1}C(p) + O(p^{-1/2}\mf q^{-\alpha/2}+ p^{-7/8+\epsilon} k^{-5/24+\epsilon}  q^{7/16+\epsilon})\right).
\end{equation}
Gathering everything together, we see that
\begin{enumerate}
    \item  When $(p,D)=1$, we choose $\alpha= 1/2-\epsilon$ and add \eqref{firstsum} with \eqref{eq:T} to get
\begin{equation}\label{eq:asympnorm}
    \sumn_{f\in B_{2k-2} ^*(p)} \frac{L(1/2, f\otimes \chi_D)}{\lan f,f\ran } = \frac{(4\pi)^{2k-3}}{\Gamma(2k-3)}\left( A_p + O(D^{7/8+\epsilon} k^{-5/24+\epsilon} p^{-7/16+\epsilon})\right).
\end{equation}
    \item When $p|D$ we choose $\alpha= 1/2-\epsilon$ and add \eqref{firstsump|D} with \eqref{eq:T} to get
\begin{equation}\label{eq:asymnormp|D}
    \sumn_{f\in B_{2k-2} ^*(p)} \frac{L(1/2, f\otimes \chi_D)}{ \lan f,f\ran } = \frac{(4\pi)^{2k-3}}{\Gamma(2k-3)}\left(B_p + O(D^{7/8+\epsilon} k^{-13/12} p^{-5/4+\epsilon})\right). \qedhere
\end{equation} 
\end{enumerate}

\begin{rmk}
Note that in the above proof, we have used the asymptotic formula from \cite[Corollary 2.2]{ILS} for $\Delta_N(s,t)$. One could have also used \cite[Theorem 16.7]{iwaniec2004analytic}, which has better exponents of $s,t$. This is of help when we wish to improve the level aspect, but only so in the error terms of asymptotic formulae in \thmref{blo-p}, for a fixed $k$. However, in this process the power of $k$ in the error term will be larger than in the main term, which does not help us.
\end{rmk}

\subsection{More twists}   \label{moretwist-proof}
In this subsection we provide the justification behind the choice $D=-4p$ when $p \equiv 3 \bmod4$ in subsection~\ref{3mod4}.
\begin{prop} \label{moretwist} 
Keep the same conditions as in \thmref{blo-p}. Then,
\begin{equation} \label{blo23mod4}
    \sum_{f\in B_{2k-2} ^*(p)} \lambda_f(4) \frac{L(1/2, f\otimes \chi_D)}{ L(1, \mrm{sym}^2 f) } = \frac{B_p(2k-2) p}{4 \pi^2} + O(D^{7/8+\epsilon} k^{-1/12} p^{-1/4+\epsilon}).
\end{equation}
\end{prop}
\begin{proof}
We will only indicate the analogous steps required to handle this situation. Quite analogously to \eqref{sn-def}, \eqref{tn-def} in the proof of \thmref{blo-p}, we put
\begin{align*}
    S'(n):= \frac{\chi_{-p}(n)}{n^{\frac{1}{2}}} V(\frac{n}{q^{\frac{1}{2}}})\sum_{f_i}\frac{ \lambda_{f_i}(4) \lambda_{f_i}(n)}{ \lan f_i,f_i\ran }, \q 
    T'(n):= \frac{\chi_{-p}(n)}{n^{\frac{1}{2}}}  V(\frac{n}{q^{\frac{1}{2}}}) \sum_{f_i} \epsilon_{f_i \otimes \chi_D} \frac{ \lambda_{f_i}(4) \lambda_{f_i}(n)}{ \lan f_i,f_i\ran },
\end{align*}
Thus we need to understand 
\begin{align}
S'+T', \q \text{   where   } \q S':=\sumn_n S'(n), \q T':=\sumn_n T'(n).
\end{align}
We recall the following well-known Hecke relation. For any $A \ge 1$,
\begin{align}
    \lambda_{f}(4A) \lambda_{f}(n) = \sumn_{d|(4A,n), (d,p)=1}  \lambda_f(4An/d^2).
\end{align}
Now notice that the sum over $n$ is supported only on $n$ such that $(n,p)=1$. 
When we consider $S'(n)$, $A=1$ and when we consider $T'(n)$, $A=p$. By the above remark, for us both the sums over $d$ are just over $d|(4,n)$. 
Thus

\begin{equation}
\begin{split}
  S'= &\sum_{d|4} \sum_{n \equiv 0 \bmod{d}} \frac{\chi_{-p}(n)}{n^{1/2}} V(\frac{n}{q^{1/2}})\sum_{f_i}\frac{  \lambda_{f_i}(4n/d^2)}{ \lan f_i,f_i\ran } \\
   & = \sumn_{d|4} \frac{\chi_{-p}(4/d) d^{1/2}}{2} \sumn_{n} \frac{\chi_{-p}(n)}{n^{1/2}} V(\frac{4n}{dq^{1/2}})\sumn_{f_i}\frac{  \lambda_{f_i}(dn)}{ \lan f_i,f_i\ran }. \label{dsum}
\end{split}
\end{equation}

Let us call the sums over the three values of $d$ as $S'_d$. Clearly $S'_1 = S/2$ up to the error term in \thmref{blo-p} (cf. \eqref{sn-def}  -- with $V(y)$ replaced by $V(4y)$). We claim that the other terms $S'_2,S'_4$ go into error terms. Indeed a moment's reflection shows that the main terms 
in the case of $p|D$ in the analysis of the quantity $S=\sum_n S(n)$ came only from $n=1$ (cf. \eqref{eq:M1n}, \eqref{eq:E1n}). So the claim is justified. The error terms remain the same.

When it comes to $T'=\sum_n T(n)$ from above we have to consider terms $T'_d$ analogous to \eqref{dsum}, but now $n$ replaced by $pn$. Comparing with \eqref{t1n} in the proof of \thmref{blo-p}, we see that there is no main term "from $L=1$" in our cases, and further no main terms arise in subsequent analysis "from $L=p$", since $2np, 4np$ are never perfect powers of $p$, see e.g., \eqref{ppower}. This means $T'$ itself is relegated to the error term, which is the same as that in \thmref{blo-p}. 
\end{proof}

\section{Sup-norms of SK lifts on average}
In this section, we establish bounds on the size of Bergman kernel for the space of SK lifts in order to prove \thmref{mainthm1} i.e., $k^{5/2}\ll \sk\ll_\epsilon k^{5/2+\epsilon}$. It is natural to utilize the `geometric side' of the BK (as an average of a suitable function over a suitable group) for this space, but it is not clear how to satisfactorily utilize the same, see the next subsection for more remarks. However to avoid this, we will transform the question to one on Jacobi forms of index $1$, which are better understood. We then use the bounds on Bergman kernel for these spaces to obtain the required bounds for SK lifts. 

\subsection{Geometric side of the Bergman kernel} \label{bkgeo} 
For completeness and with a view towards future investigations, we briefly indicate what expression one gets for the geometric side of the BK of SK lifts. We give two approaches, of course the resulting expression must be the same from both of them.

\subsubsection{From the Fourier Expansion} We start with the Fourier expansion of $\bkzz{Z}{W}$ in the $W$ variable. Below, all the sums over $F$ are over an orthonormal basis of $\skk$.
\begin{align} \label{bkseek}
    \bkzz{Z}{W} &=  \sumn_F \det(Y)^k  F(Z) \overline{F(W)} = \sumn_F \det(Y)^k F(Z) \sumn_T \overline{A_F(T)} e(- T \overline{W}).
\end{align}
For a discriminant $D<0$, let $P_D(\tau,z)$ be the $D$-th Poincaré series in $\jk$. Let $\mc P_D$ be the SK lift of $P_D$. Write $\mc P_D = \sum_{F } c_F F(Z)$,
then one can compute that $c_F = \lan \mc P_D, F \ran  = b_k^{-1} \overline{a_F(Q)} $, where $Q \in \Lambda_2$ is any matrix with $\mrm{disc}(Q)=D$ with (2,2)-th entry $1$ and $b_k = 6 (4 \pi)^k \Gamma(k)^{-1}$.
This follows from e.g., \cite[Theorem~1,~3]{K-S}. Therefore, we get
$\mc P_D = \sum_{F}  b_k^{-1}  \overline{a_F(Q)} F(Z)$, 
However, notice that in \eqref{bkseek} we need to account for all $T$ and not only those which represent $1$. This will be done by appealing to the Maass-relations. For any $G \in \skk$, one knows (see e.g., \cite{maass1979spezialschar})
\begin{align}
    a_G(\smat{a}{b/2}{b/2}{d}) = \sumn_{r \mid (a,b,d)} r^{k-1} a_G(\smat{ad/r^2}{b/2r}{b/2r}{1}).
\end{align}
We will call the matrices in the RHS above as $Q_r(T)$. With this notation then,
\begin{align}
    \sumn_F \overline{A_F(T)} F(Z) & = \sumn_{r \mid \mf c(T)} r^{k-1} \sumn_F \overline{a_F(\qtr)} F(Z) 
     = \sumn_{r \mid \mf c(T)} r^{k-1} b_k \mc P_{\det(T)/r^2}(Z).
\end{align}
 Therefore, we can write
\begin{align} \label{fe-bk-expr}
    \bkzz{Z}{W} = \det(Y)^k \sumn_T \sumn_{r \mid \mf c(T)} r^{k-1} b_k \mc P_{\det(T)/r^2}(Z) e(- T \overline{W}).
\end{align}
Now $P_D(\tau,z) = \sum_{\gamma \in \Gamma^J_\infty \backslash \Gamma^J} e(\frac{s^2-D}{4}\tau + sz) \mid_{k,1} \gamma$
which implies that for $D \equiv - s^2 \bmod{4}$
\begin{align}
    \mc P_D(Z) = \sum_{m \ge 1} \sum_{g \in \SL{2}{\z} \backslash \Delta_0(m)} \sum_{\gamma \in \Gamma^J_\infty \backslash \Gamma^J} e(\frac{s^2-D}{4}\tau + sz) \mid_{k,1} \gamma \mid_{k,1} g \ e(m \tau'),
\end{align}
which allows one to write \eqref{fe-bk-expr} in terms of functions as an average over the relevant groupoids.

\subsubsection{From the BK of Jacobi forms of index $1$}
Let $B^*_k$ denote a Hecke basis for $\skk$ and consider
\begin{equation}
   B^*(Z_1,Z_2):= \sumn_{F_f \in B_k^*} \frac{\det(Y)^k  L(k,f) F_f(Z_1) \overline{F_f(Z_2)}  }{\lan F_f, F_f \ran}.
\end{equation}
The reason to start from the weighted Bergman kernel is that the factors $L(k,f)$ show up during the conversion of $\lan F_f,F_f \ran$ to $\lan \phi, \phi \ran$ of Hecke eigenforms $F \leftrightarrow \phi$. For the size of the BK, if we use the trivial bound $L(k,f) \asymp 1$, this only affects the answer by an absolute constant:
\begin{align}
    \saitobk(Z) \asymp B^*(Z,Z).
\end{align}
Using the Fourier Jacobi expansion  of $F_f$ and the relation between Petersson norms (see \eqref{petrelns} below) we see that
\begin{align}
     B^*(Z_1,Z_2) & =   \pi^k c_k \sumn_{\phi}\frac{ 1 }{\lan \phi,\phi\ran} \sumn_{r,s} \phi(\tau_1,z_1)|V_r\cdot \overline{\phi(\tau_2,z_2)|V_s} \,e(r \tau'_1 - s \tau'_2)\n \\
     =&\pi^k c_k  \sumn_{r,s} \mbb B_{k,1,1}(\tau_1,z_1; \tau_2,z_2)|^{(1)} V_r |^{(2)}V_s \,e(r \tau'_1 - s \tau'_2),
\end{align}
where for $i=1,2$ we write $Z_i=\smat{\tau_i}{z_i}{z_i^t}{\tau_i'}$ and $\mbb B_{k,1,1}(\tau_1,z_1; \tau_2,z_2)$ is the BK for the Jacobi forms of index $1$. 

Using the geometric definition of $\mbb B_{k,1,1}$ from \eqref{BKGeometric} and the definition of $V_r$ operators, we see that $ \pi^{1-k} c^{-1}_k  (2k-3)^{-1} i^{1/2-k} 2^{k-5/2} B^*(Z_1,Z_2) $ equals

\begin{align} \label{b*}
      \sum_{r,s} \sum_{ \gamma_s\in \SL{2}{\z}  \backslash M^+_2(s) } \sum_{\gamma_r \in \SL{2}{\z} \backslash M^+_2(r)} \sum_{\gamma \in \SL{2}{\z}} \big( \Theta|^{(1)}_{k,1}\gamma_r |^{(2)}_{k,s}(\gamma \cdot \gamma_s) \big)(\tau_1,z_1;\tau_2,z_2) e(r \tau'_1 - s \tau'_2) ,
\end{align}
where $\Theta$ is as in \eqref{theta-def}. Moreover, $M^+_2(t)$ denotes the set of $2 \times 2$ integral matrices of determinant $t$.

Define $\widetilde{\Theta}(Z_1,Z_2):= \Theta(\tau_1,z_1;\tau_2,z_2)e(\tau'_1 - \tau'_2)$. Recall the action of $M_2^+(\z)$ on $\mbb H \times \mbb C^2$ from \cite[p.~144]{ibukiyama2012saito}. This allows us to rewrite the sums in \eqref{b*} as follows.
\begin{equation}
   \sumn_{ \gamma_2\in \mrm{SL}_2(\mbb{Z})\backslash M^+_2(\mbb{Z}) } \, \sumn_{\gamma_1\in M^+_2(\mbb{Z})} (\widetilde{\Theta} \big|^{(1)}_{k,1}\gamma_1^{\uparrow} \, \big|^{(2)}_{k,1}\gamma_2^{\uparrow})(Z_1,Z_2) ,
\end{equation}
where $\gamma_j^\uparrow$ denote the standard diagonal embedding of the Jacobi group in the Siegel modular group,(cf. section~\ref{prelim}~(3) and also \cite[p.~144]{ibukiyama2012saito}) but with the $\mrm{GL}_2$ component having determinant at least $1$.

\subsection{Lower bound for \texorpdfstring{$\sk$}{sup}} \label{lbd-sk}
We start by quoting from \cite[(3.1)]{sd-hk}
\begin{equation} \label{sk-lbd}
  \sup\nolimits_{Z \in \mbb H_2}  \sumn_{f \in {B}_{2k-2}} \det(Y)^k (|F_f(Z)|/\norm{F_f})^2 \ge \det(Y_0)^k e^{\tr (- 4 \pi Y_0))} \sumn_f |a_{F_f}(I_2)/\norm{F_f}|^2,
\end{equation}
where ${B}_{2k-2}$ denotes the Hecke basis for $S_{2k-2}$. Our argument here is similar in spirit similar to that used by Blomer in \cite{blo}.

Recall that with $F=F_f$, one has $a_{F}(I_2) = c(4)$, where $h(\tau)=\sum_n c(n)q^n$ is the half-integral cusp form in Kohnen's plus space. Further, 
by Waldspurger's formula (cf. \cite{ko2}), for $k$ even and $D$ a fundamental discriminant such that $D<0$,
\begin{align}
    |c(|D|)|^2 = \frac{\Gamma(k-1)}{\pi^{k-1}} |D|^{k-3/2}L(k-1, f\otimes \chi_D) \frac{\lan h, h \ran}{6 \lan f , f \ran}.
\end{align}
Moreover, the following relations hold amongst the Petersson norms:
\begin{align} \label{petrelns}
    \frac{\lan F, F \ran}{\lan \phi_1, \phi_1 \ran} = L(k,f) \pi^{-k} c_k^{-1} ; \q \frac{\lan \phi_1, \phi_1 \ran}{\lan h, h \ran} = 2^{2k-3} \q \q (c_k= \frac{3 \cdot 2^{2k+1}}{\Gamma(k)}).
\end{align}
These together give the relation
\begin{align}
    48 \pi^k \lan F, F \ran = L(k,f)  \Gamma(k)\lan h, h \ran.
\end{align}
We normalise $F$ with $\norm{F}_2=1$. This implies that
\begin{align} \label{wald}
     |c(|D|)|^2 \asymp \frac{D^{k-3/2}}{k} L(k-1, f\otimes \chi_D) \frac{(4 \pi)^{2k}}{\Gamma(2k-2) \, L(k, \mrm{sym}^2(f))}.
\end{align}
Putting together \eqref{wald} and \eqref{sk-lbd}, we finally see that (with $D=-4$ and $Y_0= \frac{k I_2}{4 \pi}$)
\begin{align}
  \sup_{Z \in \mbb H_2}  \sum_{f \in {B}_{2k-2}} \det(Y)^k |F_f(Z)|^2 
  &\gg \frac{ \det(4 \pi Y_0)^k \exp(\tr (- 4 \pi Y_0)) 4^{k-3/2} }{k \, \Gamma(2k-2)}  \sum_f \frac{L(k-1, f\otimes \chi_{-4})}{ L(k, \mrm{sym}^2(f))} \n\\
  & \gg \frac{ k^{2k} \exp(-2k) 4^{k-3/2}}{k^{2k - 5/2} 4^k \exp(-2k)} \label{firstmom} 
 \gg k^{5/2}.
\end{align}
where in \eqref{firstmom} we have used the fact that $\sum_f \frac{L(k-1, f\otimes \chi_{-4})}{ L(k,  \mrm{sym}^2(f))} = 2k/\pi^2 + O(1)$ from \cite[Lemma 3]{blo}.

\subsection{Upper bound - first method.}
\label{method1}
We now prove the upper bound for the quantity $\sk$. Our first method is a combination of bounds for certain Fourier coefficients of Jacobi Poincar{\'e} series  (cf. \propref{proppkmnr}) and some of the methods in \cite{sd-hk} and \cite{blo}. The underlying idea, as mentioned in the introduction, is that the $L^2$-norm of the Fourier coefficients $a_F(T)$ in a basis $\{F\}$ controls the sup-norm of the Bergman kernel efficiently (cf. \cite{sd-hk}). We will see below that this idea essentially gives us the desired upper bound for the Bergman kernel for SK lifts on the bulk of $\mc F_2$. In the remaining small region of $\mc F_2$ we will use the Fourier-Jacobi expansion of the $F$ and embed the index-old forms $\phi_{1,F}|V_m$ into $\jkm$ and look at their contribution to the Bergman kernel for $\jkm$. This will be our first method of bounding the Bergman kernel for SK lifts.

We begin by estimating the quantity $p_*(T):= \sum_{F \in \skkunit} |a_{F}(T)|^2$ . Then using the relation \eqref{fcreln},
\begin{equation}
    \begin{split}
      p_*(T)&= \sum_F \big| \sum_{a|(n,r,m)} a^{k-1}  c_{\phi,F}\left( \frac{D }{a^2} \right) \big|^2 = \sum_{a,a'} (a a')^{k-1} \sum_F c_{\phi,F}\left( \frac{D }{a^2} \right) \overline{ c_{\phi,F}\left( \frac{D }{a'^2} \right) }\\
      &\le \Big( \sumn_a a^{k-1} \big( \sum_F |c_{\phi,F}\left( \frac{D }{a^2} \right) |^2 \big)^{1/2} \Big)^2.
    \end{split}
\end{equation}
Now using the bound for the Fourier coefficients of Poincar{\'e} series from Proposition \ref{proppkmnr} we see that
\begin{align}
p_*(T)& \ll \frac{ 4^k \pi^{2k} D^k}{\Gamma(k) \Gamma(k-3/2) } \left( \sum_a a^{k-1} \left( a^{2k} (D^{-3/2}a^{3}+k^{-5/6} D^{-1+\epsilon} a^{2-\epsilon} ) \right)^{1/2}  \right)^2 \n\\
& \ll_\epsilon \frac{ 4^k \pi^{2k} D^k}{\Gamma(k) \Gamma(k-3/2) }
\left( D^{-3/4} \mf c(T)^{1/2}+ D^{-1/2+\epsilon} k^{-5/12} \mf c(T)^\epsilon \right)^2.
\end{align}
Next, note that we can write by \cite[Lem.~4.4]{sd-hk} (follows from Cauchy-Schwartz inequality) that
\begin{align} \label{lem41}
     \sumn_{F \in \skkunit}  \det(Y)^k |F(Z)|^2 \le q_*(Y)^2 
\end{align}
where we have put
\begin{align} \label{qp}
   q_*(Y) := \sumn_T p_*(T)^{1/2} \det(Y)^{k/2} \exp(- 2 \pi TY).
\end{align}
Plugging in the bound for $p_*(T)$ in \eqref{qp} we get that $ q_*(Y) $ is
\begin{align}
    \ll  \frac{ (4\pi)^{k} }{ \sqrt{ \Gamma(k) \Gamma(k-3/2)} } & \sum_T \left( \frac{\mf c(T)^{1/2} \det(T)^{k/2}}{\det(T)^{3/4} }  +  \frac{\mf c(T)^\epsilon \det(T)^{k/2}}{\det(T)^{1/2-\epsilon} k^{5/12 +\epsilon}} \right) \det(Y)^{k/2} \exp(- 2 \pi TY) \n\\
& \ll k^{3/4} Q(1/2,3/4-\epsilon;Y)+k^{1/3-\epsilon}Q(\epsilon,1/2; Y),  \label{qabreln}
\end{align}
where 
\begin{align}\label{QabY}
    Q(\alpha, \beta; Y): = \frac{ (4\pi)^k}{ \Gamma(k) } \sumn_T \mf c(T)^\alpha \det(Y)^{k/2} \det(T)^{k/2-\beta +\epsilon} e^{-2\pi Tr(T Y)} .
\end{align}
\begin{lem}\label{lemQabY}
One has 
    $Q(\alpha, \beta; Y)\ll  k^{2-2\beta }  \, \sum_{d \ll k} d^{\alpha - 2\beta - \epsilon} \det(Y)^{\beta-3/4} $.
\end{lem}
\begin{proof} 
We refer to part of the calculations done in \cite{blo}. Especially looking at  \cite[p.~346]{blo} equation~(4.1) onwards we can arrive at the bounds given below. As in loc. cit. put $Y_d = dY$.
\begin{align}
     Q(\alpha, & \beta; Y)  \ll \sumn_d d^{\alpha - 2\beta - \epsilon} \sumn_T \frac{ (4\pi)^k \det(T Y_d)^{k/2} e^{-2\pi Tr(T Y_d)} }{ \det(T)^{\beta - \epsilon} \Gamma(k) } \n \\
     &\ll   k^{1/2-2\beta} \, \sumn_{d \ll k} d^{\alpha - 2\beta - \epsilon} \det(Y_d)^{\beta} \sumn_{T Y_d \in \mc X} 1,
\end{align}     
where $\mc X$ is the set of diagonalizable matrices whose eigenvalues are both of size $k/(4 \pi) + O(\sqrt{k} \log(k))$. Then from \cite[Lemma~4]{blo}, we know that $\sum_{T Y_d \in \mc X} 1 \ll k^{3/2+ \epsilon} \det(Y)^{-3/4}$.
\begin{equation}
       \ll   k^{1/2-2\beta +3/2}  \, \sumn_{d \ll k} d^{\alpha - 2\beta - \epsilon} \det(Y)^{\beta-3/4}. \qedhere
\end{equation}
\end{proof}

Plugging in the values of $(\alpha,  \beta)$ from \eqref{qabreln} into \lemref{lemQabY} we get
\begin{align}
    q_*(Y) \ll k^{5/4+\epsilon} + k^{5/4 +1/12} \det(Y)^{-1/4}.
\end{align}
So here we need to focus on the region $\det(Y) \ll k^{1/3}$ from methods pertaining to the Bergman kernel for a suitable space. Further, note that since $\det Y \ll k^{1/3}$ and $v \le v'$ and $vv' \asymp \det Y$, we also have
\begin{equation} \label{remreg}
    v \ll k^{1/6}, \q v' \ll k^{1/3}.
\end{equation}

\subsubsection{The remaining region}
\label{compactregbkm}
In this subsection, we assume that $v,v'$ satisfy \eqref{remreg}, which defines a region $\mc R$. In this region $\mc R$ we would use the 
information on the Bergman kernel on the space of Jacobi forms. To do this, we embed the space spanned by the $V_m$-images of $\jk$ into $\jkm$ and use the results from section~\ref{bkjkmsec} on $\bkm$; it turns out to be sufficient for the region under consideration.
First, let us recall the  bound for $\bkm$:
\begin{align}
    \bkm \ll km \left( 1 +  vk^{-1/2+\epsilon} + vk^{-A}\right) \left(  1+ v^{1/2}m^{-1/2} \right).
\end{align}
The first quantity in braces above is bounded in $\mc R$. Moreover, we can rewrite the above bound as 
\begin{equation}
    \bkm \ll km^{1/2}v^{1/2} \left( 1 +  vk^{-1/2+\epsilon} + vk^{-A}\right) \left(  1+ m^{1/2}v^{-1/2} \right).
\end{equation}
Therefore in $\mc R$,
\begin{align} \label{bkmsk}
    \bkm \ll \begin{cases}  k m ^{1/2} v^{1/2} & \text{ if } m \le v; \\
     km & \text{ if } v \le m.
    \end{cases}
\end{align}
First note that
\begin{align}
  \sumn_F  \det(Y)^k |F(Z)|^2 & \le t^{k} \sumn_F \left( \sumn_{m \ge 1} v^{k/2} \exp(- 2\pi m y^2/v)|\phi_{m,F}|  \exp(- 2 \pi m t) \right)^2 \n \\
  & \le t^k \sumn_{m,m'} \, \sumn_F \widetilde{\phi_{m,F}}  \widetilde{\phi_{m',F}} \, \exp(- 2 \pi (m+m') t) \\
  & \le t^k ( \sumn_m p(m)^{1/2} \exp(- 2 \pi m t) )^2, \label{fefj}
\end{align}
where we have put
\begin{align}
    p(m) := \sumn_F \widetilde{\phi_{m,F}(\tau,z)}^2, \q t=v'-y^2/v.
\end{align}

Next, we want to check that if $F,G \in \skk$ are orthogonal, then so are $\phi_{m,F}, \phi_{m,G}$.
\begin{lem} \label{ortholem}
Let $F,G \in \skk$ be Hecke eigenforms such that $\lan F, G \ran=0$. Then for all $m \ge 1$, $\lan \phi_{m,F}, \phi_{m,G} \ran =0$.
\end{lem}

\begin{proof}
Let $F,G$ be SK lifts of $f,g \in S_{2k-2}$ respectively. We have (cf. \cite[p.~549-550]{K-S})
\begin{align}
    \lan \phi_{m,F}, \phi_{m,G} \ran & = \lan V_m \phi_{1,F}, V_m \phi_{1,G} \ran = \lan V_m^*V_m \phi_{1,F}, \phi_{1,G} \ran \n\\
    & = \lan \sumn_{d|m} \psi(d) d^{k-2} T^J(m/d) \phi_{1,F}, \phi_{1,G} \ran = \sumn_{d|m} \psi(d) d^{k-2} \lambda_f(m/d) \lan \phi_{1,F}, \phi_{1,G} \ran,
\end{align}
where $\psi(t)= t \prod_{p|t}(1+1/p)$. Here we should keep in mind the convention made before \eqref{skbk-def}.
Moreover, $\lan F, G \ran=0$ implies that $\lan \phi_{1,F}, \phi_{1,G} \ran=0$ (see \cite[Theorem~2]{K-S}).
The lemma follows.\qedhere
\end{proof}

The above proof also shows that
\begin{align}
     \lan \phi_{m,F} , \phi_{m,F} \ran \ll m^{k-1 +\epsilon} \cdot \lan \phi_1, \phi_1 \ran.
\end{align}

Note however that $\phi_{m,F}$ could be $0$ for some $m$. Therefore, when $m \ge v$,
\begin{align}
    p(m)= \sideset{}{^*}\sumn_{F} \frac{ \widetilde{\phi_{m,F}}^2 }{ \lan \phi_{m,F} , \phi_{m,F} \ran} \cdot  \lan \phi_{m,F} , \phi_{m,F} \ran\ll km \cdot m^{k-1+\epsilon} \pi^k c_k \ll \frac{ ( 4 \pi m)^k }{\Gamma(k-1)},
\end{align}
where $\sum^*$ denotes the sum over non-zero modular forms.
Thus this part of the Fourier series gives the contribution $k^{5/2 +\epsilon}$ because
\begin{align} 
    \sumn_{m \ge v} \cdots \le \frac{(4\pi t)^k}{\Gamma(k-1)} (\sumn_{m \ge 1} m^{k/2} \exp(- 2 \pi m t) )^2 \ll 2^k \frac{\Gamma(k/2+1)^2}{\Gamma(k-1)} \ll k^{5/2+\epsilon}. \label{vm}
\end{align}
For the terms where $m \le v$ we use the first inequality in \eqref{bkmsk} to get
\begin{equation}
    p(m) \ll k^{13/12} m^{1/2} (4 \pi)^k \,\Gamma(k)^{-1},
\end{equation}
and this gives the bound for $\sum_{m\le v}$
\begin{align}
&\ll  k^{13/12}  \frac{(4\pi t)^k}{\Gamma(k)} ( \sum_{m \le v} m^{k/2-1/4}\exp(- 2 \pi m t) )^2 \ll 2^k k^{13/12}  \frac{t^{1/2}}{\Gamma(k)} \cdot \Gamma(k/2+3/4)^2\ll k^{9/4+\epsilon}. \label{mv}
\end{align}
Thus combining \eqref{mv} and \eqref{vm} we get from \eqref{fefj} that
\begin{equation}
    \bksk \ll k^{5/2+\epsilon}.
\end{equation}

\subsection{Upper bound -- second method}
\label{method-2}
In this section, we provide an alternate proof of \thmref{mainthm1}, that is independent of the bounds for the geometric side of the Bergman kernel of Jacobi forms. In this direction, we start by considering the sum
\begin{align}
  A(Z):=  \sumn_{F \in  B_k^*} \frac{\det(Y)^k |F(Z)|^2}{\lan F, F \ran},
\end{align}
where $B^*_k$ denotes a Hecke basis for $\skk$ (cf. the convention before \eqref{skbk-def}).

Let $\tau=u+iv$, $z=x+iy$ and $\tau'=u'+iv'$ and $\phi\in \jk$ denote the first Fourier--Jacobi coefficient of $F$. Consider the expression
\begin{align}
 \det(Y)^{-k} A(Z)&= \sumn_F  \frac{|F(Z)|^2}{\lan F, F \ran}= \pi^k c_k\sumn_F  \frac{|F(Z)|^2}{\lan \phi,\phi\ran} \frac{1}{|L(k,f)|} \label{petratio}\\
   &\ll \pi^k c_k \sumn_{\phi}\frac{1}{\lan \phi,\phi\ran} \sumn_{r,s} \phi_r(\tau,z) \overline{\phi_s(\tau,z)} e(r \tau' - s \tau').
\end{align}
where in \eqref{petratio} we have used \eqref{petrelns} and bounded the quantity $L(k,f)^{-1}$ by an absolute constant, since $k$ falls in the region of absolute convergence.

We have $\phi_r=\phi|V_r$ with, $\phi\in J_{k,1}$. Thus, we only have  to bound the quantity:
\begin{equation} \label{a*}
  A^*(Z):= \pi^k c_k \sumn_{\phi}\frac{1}{\lan \phi,\phi\ran} \sumn_{r,s} \phi(\tau,z)|V_r\cdot \overline{\phi(\tau,z)|V_s} \,e(r \tau' - s \tau').
\end{equation}
Let us now replace $\phi/\norm{\phi}_2$ by $\phi$ -- which therefore henceforth will run over an orthonormalized basis of $\jk$. 
With this arrangement, and using the Fourier expansion of $\phi|V_r$, $A^*(Z)$ can be written as
\begin{equation}
\sum_\phi\sum_{r,s}\sum_{n_1,n_2,t_1,t_2}\sum_{d_1|(n_1,t_1,r)}\sum_{d_2|(n_2,t_2,s)}(d_1d_2)^{k-1}  c_\phi(\tfrac{n_1r}{d_1^2},\tfrac{t_1}{d_1})\overline{c_\phi(\tfrac{n_2s}{d_2^2},\tfrac{t_2}{d_2})} e(\cdots). \label{az1z2}
\end{equation}
For $i=1,2$, write $D_i=4n_ir-t_i^2$ and put  $T_i = \smat{(D_i + t_i^2)/4r}{t_i/2}{t_i/2}{r}$. Also, put 
\begin{equation}\label{SD1D2}
    S(D_1,D_2):= \sumn_\phi c_\phi(D_1,t_1)\overline{c_\phi(D_2,t_2)}.
\end{equation}
Thus we get

\begin{equation}\label{Az1z2Sl}
\begin{split}
    A^*(Z)=& \pi^k c_k \sum_{r,s}\sum_{D_1,D_2,t_1,t_2}\sum_{d_1|c(T_1)}\sum_{d_2|c(T_2)}(d_1d_2)^{k-1}S(D_1d_1^{-2},D_2d_2^{-2})e(\cdots)\\
    &=\pi^k c_k \sum_{T_1,T_2>0}\sum_{d_1|c(T_1)}\sum_{d_2|c(T_2)}(d_1d_2)^{k-1} S(D_1d_1^{-2},D_2d_2^{-2}) e(Tr(T_1Z+T_2Z)).
\end{split}
\end{equation}
Now we use the Poincar{\'e} series to estimate the sum $S(D_1,D_2)$. Since $\{\phi\}$ form an orthonormal basis, we can write
\begin{equation}
   2\pi^{k-\frac{3}{2}}D_2^{k-\frac{3}{2}} P_{D_2} = \Gamma(k- 3/2) \sumn_\phi \overline{c_\phi(D_2)} \phi.
\end{equation}
Comparing the Fourier coefficients at $D_1$ on both sides, we see that
\begin{equation} \label{slpl}
   \Gamma(k- 3/2) S(D_1,D_2)=2\pi^{k-\frac{3}{2}}D_2^{k-\frac{3}{2}} p_{D_2}(D_1).
\end{equation}
Now we estimate proceed to estimate the coefficients $p_{D_2}(D_1)$. First, we observe here that the Fourier coefficients of the Jacobi Poincar{\'e} series of index $1$ are the same as that of the half--integral weight Poincar{\'e} series in the Kohnen-plus space (see \cite[Proposition 6.2]{das2010nonvanishing}). This allows us to get better bounds for the Fourier coefficients $p_{D_2}(D_1)$ as compared to the bounds in Proposition \ref{proppkmnr}.
\begin{prop}\label{PSboundRAR}

For $k$ large, we have
\begin{align}
  p_{D_2}(D_1)\ll \delta_{D_1,D_2}+ \left(\tfrac{D_1}{D_2}\right)^{k/2-3/4}\left( \tfrac{(D_1D_2)^{1/4+\epsilon}}{k}+\tfrac{(D_1D_2)^\epsilon}{k^{1/3}}\right).
\end{align}
\end{prop}
Thus we get the following bound from \eqref{slpl};
\begin{prop}\label{propSD1D1}
\begin{align}
     S(D_1,D_2)\ll \tfrac{\pi^{k}}{\Gamma(k-3/2)}  \cdot \left(  \delta_{D_1,D_2} D_2^{k-3/2} + \tfrac{(D_1D_2)^{k/2-1/2}}{k}+\tfrac{(D_1D_2)^{k/2-3/4+\epsilon}}{k^{1/3}} \right).\label{slbound}
\end{align}
\end{prop}
Now we proceed to bound $A^*(Z)$. We write $A^*(Z)\ll A_1(Z)+A_2(Z)+A_3(Z)$ corresponding to three terms in \eqref{slbound} and we bound each $A_i(Z)$ as below.
\subsubsection{Bounding $A_i(Z)$:}
\begin{align}
      A_1(Z)\ll \frac{ (4 \pi)^k (4\pi)^{k-\frac{3}{2}}}{ \Gamma(k)  \Gamma(k-\frac{3}{2}) } \sumn_{T>0}\sumn_{d \mid \mf c(T)} d^{k-1} \mrm{det} (d^{-1}T)^{k-3/2} e^{-4\pi Tr(TY)}\ll \frac{k^{3/2}  Q(1/4, 3/4; Y)^2}{\det(Y)^{k}},
\end{align}
where $Q(\alpha, \beta; Y)$ is as in \eqref{QabY}. Therefore, from Lemma \eqref{lemQabY} we see that,
\begin{align}
    A_1(Z) &\ll k^{3/2}\cdot k^{1+\epsilon} \det(Y)^{-k} \ll  k^{5/2+\epsilon} \det(Y)^{-k}
\end{align}
and similarly
\begin{align}
    A_{2}(Z)
      & \ll k^{1/2} Q(\epsilon, 1/2;Y)^2 \det(Y)^{-k}\ll k^{5/2}\det(Y)^{-k-1/2},
\end{align}
from Lemma \ref{lemQabY} with $\alpha=\epsilon$ and $\beta=1/2$.
\begin{align}
    A_{3}(Z) 
      & \ll  k^{7/6} Q(1/2, 3/4;Y)^2 \det(Y)^{-k}\ll k^{13/6+\epsilon}\det(Y)^{-k},
\end{align}
by Lemma \ref{lemQabY} with $\alpha=1/2$ and $\beta=3/4$. This gives us $A(Z)\ll k^{5/2+\epsilon}$, which finishes the proof.

\section{Higher moments and scope for improvement via subconvexity} \label{subconvex}
In this section, we investigate the relation between the subconvexity bounds (for example, see \cite{Y}) on the twisted central $L$-values $L(1/2, f \otimes \chi_D)$ of a $\GL{}{2}$ Hecke eigenform $f$
and the higher moments of $\det (Y)^{k/2} |F(Z)|$ along an orthonormal Hecke basis. Put, for any $r>0$,
\begin{equation}
    \mathscr F_{2r} := \sup\nolimits_{Z \in \mc F_2}\sumn_F \det(Y)^{kr} |F(Z)|^{2r}.
\end{equation}
For any $r\in \mbb R$, $r>1$, and $Z \in \htwo$ consider the inequality
\begin{align} \label{infty+2}
    \sumn_F \det(Y)^{kr} |F(Z)|^{2r}\ll \left(\max_F \det(Y)^k |F(Z)|^2\right)^{r-1} \sumn_F \det(Y)^{k} |F(Z)|^2;
\end{align}
which shows immediately that $\mathscr F_{2r} \ll k^{5r/2+\epsilon}$ if we use \thmref{mainthm1}.

Next, consider the following bound for the third moment of twisted central $L$-values $L(1/2, f \otimes \chi_D)$ from \cite{Y}, where $D$ is an odd fundamental discriminant (\textsl{we assume it here for every fundamental discriminant}): for any $\epsilon>0$, one has
\begin{equation} \label{y-improve}
    \sumn_{f\in B_k} L(1/2, f \otimes \chi_D)^3\ll_\epsilon k^{1+\epsilon} D^{1- \delta +\epsilon}
\end{equation}
for some absolute constant  $\delta$ such that $0 \le\delta \le 1$. Thus, we assume an improvement the $D$ aspect of Young's result in \cite{Y}. Our goal in the rest of this section is to show that this results in an improvement of the upper bound of $k^{15/2 +\epsilon}$ of $\mathscr F_6$ from \eqref{infty+2}. It seems to us that this is the only viable avenue to improve the individual bounds for an SK lift. In \cite{blo}, assuming Generalized Lindel\"of Hypothesis for $L(1/2, f \otimes \chi_D)$ for all negative fundamental discriminants, it was shown that $\norm{F}_\infty \ll k^{3/4+\epsilon}$, which is the best possible, as there exist eigenforms with sup-norm $\gg k^{3/4}$.

We start with the inequality:
\begin{align}
    \sumn_F \big( \frac{\det(Y)^{k/2} |F(Z)|}{\petf^{\frac{1}{2}}} \big)^6 = \sumn_F \frac{\det(Y)^{3k} |F(Z)|^6}{\petf^3} \ll (\sumn_T p_6(T)^{\frac{1}{6}} \det(Y)^{\frac{k}{2}} \exp(- TY) )^6, \n
\end{align}
where we have put $p_6(T) = \sum_F |a_F(T)|^6/\petf^{3}$.

For any $T$, write $4 \det(T) = D = f^2 D_0 $ where $D_0$ is fundamental. This expression is unique. This $f$, we believe has no chance of being confused with the lifted modular form $f \in S_{2k-2}$.
Looking at the expression for $a_F(T)$ from \eqref{fcreln}, which involves the terms $\det(2T)/a^2$,
we have to determine $(D/a^2)_0$, for a given $T$ and $a$ as above, to relate them to fundamental discriminants. Clearly, if $n_o$ denotes the `odd' part of an integer $n$, then 
$  \mf c(T)_o | f_o$.
Thus $(D/a^2)_0 = D_0$ if $a$ is odd. So is the case if $D_0$ is odd.

Thus we only have to look at $p=2$. Let $n_2$ be the $2$-part of $n$.
If $2^\beta \| a$, then $2^{2 \beta} | f^2 D_0$. But we know that $\nu_2(D_0) \le 3$, so $2^{2\beta-3} | f^2$ which means $2^{2 \beta -2} | f^2$. Therefore, writing $D_0 = 4m$ ($m \equiv 2,3 \bmod 4$ and square-free)
\begin{align}
\frac{D}{a^2} = \frac{f^2 D_0}{a^2} = \frac{(f_o/a_o)^2 f_2^2 D_0}{a_2^2} = \frac{(f_o/a_o)^2 (f_2/ 2^{\beta -1})^2 D_0}{4} = A^2 \cdot m,
\end{align}
for some $A \in \N$.
Since $m \equiv 2,3 \bmod 4$ and square-free, the only way to write $A^2 \cdot m $ as $B^2 \cdot \mc D$ ($\mc D$ fundamental discriminant)--is $\mc D=4m$ and $B=A/2$. This will show that $(D/a^2)_0 = D_0$. In effect, we have also proved above that $a | f$ (in particular $\mf c(T)|f$). Therefore
\begin{align}
    \sumn_{a|\mf c(T)} a^{k-1} c_\phi(4\det(T)/a^2) = c_\phi(|D_0|) \sumn_{a|\mf c(T)} a^{k-1} \sumn_{b|\frac{f}{a}} \mu(b) \left( \frac{D}{b} \right) b^{k-3/2} a_f(\frac{f}{ab}) \n \\
    \le c_\phi(|D_0|) \sumn_{a|\mf c(T)} a^{k-1} (f/a)^{k-3/2+\epsilon}\ll c_\phi(|D_0|) f^{k-3/2+\epsilon} \mf c(T)^{1/2}.
\end{align}
Keeping in mind \eqref{y-improve}, this gives us
\begin{align*}
    p_6(T)  \le  \sumn_F \frac{c_\phi(|D_0|)^6 f^{6(k-\frac{3}{2})+\epsilon} \mf c(T)^{3}}{\petf^{3}} 
    \ll \frac{\det(T)^{3(k-\frac{3}{2})+\epsilon} \mf c(T)^3 }{k^3\Gamma(2k-2)^3}\sumn_f  \frac{L(k-1, f \otimes \chi_{D_0})^3}{L(k,f)^3L(k,\mrm{sym}^2 f)^3},
\end{align*}
where we have used the relation between Petersson norms from \eqref{petrelns}. Thus,
\begin{align}
    p_6(T)&\ll \frac{\det(T)^{3k-9/2+\epsilon} \mf c(T)^{3}}{k^3\Gamma(2k-2)^3} k^{1+\epsilon}D_0^{1-\delta +\epsilon}\ll \frac{k^{5/2+\epsilon}}{\Gamma(k)^6}\det(T)^{3k-7/2-\delta+\epsilon} \mf c(T)^{1-\epsilon}.
\end{align}
In conclusion therefore, we get (with the function $Q(\cdot, \cdot;Y)$ from \eqref{QabY})
\begin{align}
   \sumn_F \det(Y)^{3k} |F(Z)|^6/\petf^3 \ll k^{5/2+\epsilon}Q(1/6,7/12+\delta /6-\epsilon;Y)^6\ll k^{15/2-2\delta+\epsilon} \det(Y)^{\delta - 1},
\end{align}
which is bounded by $k^{15/2-2\delta+\epsilon}$ since $\delta \le 1$.
For an individual $F$, we would have
\begin{equation}
    \det(Y)^{k/2} |F(Z)|\ll  k^{5/4-\delta/3+\epsilon}.
\end{equation}

\section{\texorpdfstring{$L^\infty$}{L-inf} mass of pullback of SK lifts} \label{pullback-sec}
Here we would first restrict $F_f \in \skk$ ($f\in S_{2k-2}$) to the diagonal. The resulting object (pullback) can be written as
\begin{align} \label{pullback-exp}
   F_f^o(\tau, \tau'):= F_f |_{z=0} = \sumn_{g\in B_k} c_{g,f} g(\tau) g(\tau'),
\end{align}
and by Ichino (see \cite{ichino}), the quantities $c_{g,f}$ are given by $L$-values. More precisely,
\begin{align}
    |c_{g,f}|^2 = \frac{ |\lan F_f |_{z=0} , g \times g \ran|^2}{ \lan g,g \ran^4 } = 2^{-k-1} \frac{\Lambda(1/2, \mrm{sym}^2g \times f) , \lan h,h \ran }{\lan f,f \ran \,  \lan g,g \ran^2 }.
\end{align}
Our goal in this section is to understand the size of the space $\skk^\circ$ spanned by the pullbacks of SK lifts; the $L^2$-size of the same was studied in \cite{liu2014growth}, \cite{BKY}.
In order to quantify the concentration of $F_f$
along the diagonally embedded $\h \times \h \subset \htwo$, we introduce the following three quantities (for $f \ne 0$)

\begin{align}
    M_1(F_f) = \frac{\sup(\ff) }{\sup(F_f)}; \q   M_2(F_f) = \frac{\sup(F_f^o) }{\norm{F_f} }; \q M_3(f):=\frac{\sup(\ff) }{\norm{\ff} }.
\end{align}
We present a precise result on the size of the relevant object of this subsection – the mass of pullbacks of SK lifts measured along $\skk$. Recall that $B_{2k-2}$ denotes the set of newforms of level one on $\SL{2}{\z}$. Put
\begin{align} \label{fcircdef}
    \sup(\skk^p) := \sup\nolimits_{(\tau,\tau')\in\h \times \h}  \sumn_{f \in B_{2k-2}} (vv')^k |F_f^o (\tau, \tau')|^2/\lan F_f, F_f\ran\,\, (=\sup(\skk|_{z=0}))
\end{align}
and to quantify the \textsl{density} of the \textsl{pullbacks}, we introduce the quantity 
\begin{equation} \label{skpd}
     d:=d(\skk^p) :=  \sup(\skk^p) \big / \sup(\skk).
\end{equation}
Then $d \le 1$ and the value of $d$ may be interpreted as the density of the pullbacks inside the space of SK lifts in terms of $L^\infty$ mass – it measures the proportion of the contribution of the pullbacks to the BK coming from SK lifts. \textsl{Note however that the summation in \eqref{fcircdef} runs over a spanning set of $\skk^\circ$ but the forms $F_f^\circ$ may not be orthogonal.}

The bound on $d$ immediately gives us that $\sup(\skk^p)\ll k^{5/2+\epsilon}$. Next, we prove the lower bound  $\sup(\skk^p)\gg k^{5/2}$ (cf. \propref{skp} below), which shows that almost all of the "mass" of $\skk$ is concentrated along $z=0$. A natural way to do this would be to appeal to \eqref{pullback-exp}, and via the method of section~\ref{lbd-sk}; however, this would lead us to averages $\sum_f (\sum_g c_{g,f})^2$ -- as $\sum_g c_{g,f}$ is the $(1,1)$-th Fourier coefficient of the RHS of \eqref{pullback-exp}. Since we have at present not much control on this quantity, we take a different path.

We first consider the Fourier expansion of $F_f^\circ$ borrowed from that of $F_f$ (for convenience, we use the notation $a_F(T)= a_F(n,r,m)$ when $T = \smat{n}{r/2}{r/2}{m}$):
\begin{align}
    F_f^\circ = \sumn_{n,m} (\sumn_{r} a_{F_f}(n,r,m)) e(n \tau+ m\tau').
\end{align}
Let $v_0$, $v_0'\in \mbb H$. Then we have $\sup(\skk^p)$ is
\begin{align}
     &\gg (v_0v_0')^k e^{-4\pi (v_0+v_0')} \sumn_{f} \frac{|a_{F_f}(I_2)+a_{F_f}(1,1,1)|^2}{\lan F_f, F_f \ran} \n\\
    &\gg (v_0v_0')^k e^{-4\pi (v_0+v_0')} \Big( \sumn_{f} \frac{|a_{F_f}(I_2)|^2}{\lan F_f, F_f \ran} + \sumn_{f} \frac{|a_{F_f}(1,1,1)|^2}{\lan F_f, F_f \ran} - 2 \Big|\sumn_{f} \frac{a_F(I_2) \overline{ a_F(1,1,1) }}{\lan F_f, F_f \ran} \Big| \Big).
\end{align}
Now consider the individual terms on the LHS. Using the same arguments as in  \eqref{firstmom},
\begin{equation}
    (v_0v_0')^k e^{-4\pi (v_0+v_0')} \sumn_{f} \frac{|a_{F_f}(I_2)|^2}{\lan F_f, F_f \ran}\gg k^{5/2}.
\end{equation}
Here we have made the choice $v_0=v_0'=k/4\pi$. Similarly, since $D=-3$ in the second term, using the same arguments as in \eqref{firstmom}, we get

\begin{equation}
    (v_0v_0')^k e^{-4\pi (v_0+v_0')} \sumn_{f} \frac{|a_{F_f}(1,1,1)|^2}{\lan F_f, F_f \ran}\gg k^{5/2}  (3 \big / 4)^{k}.
\end{equation}
Now to estimate the off diagonal term, we use Poincaré series. The off diagonal term is given by
\begin{equation} \label{910}
    (v_0v_0')^k e^{-4\pi (v_0+v_0')} \sumn_{f} \frac{a_F(I_2) \overline{ a_F(1,1,1) }}{\lan F_f, F_f \ran}.
\end{equation}
Using the relation between Petersson norms from \eqref{petrelns} and \propref{propSD1D1}, \eqref{910} is
\begin{align}
    &\ll \frac{(4\pi)^k}{\Gamma(k)}(v_0v_0')^k e^{-4\pi (v_0+v_0')} S(4,3)\ll \frac{\pi^{k}}{\Gamma(k-\frac{3}{2})}\frac{(4\pi)^k}{\Gamma(k)} \frac{k^{2k}}{(4\pi)^{2k}}e^{-2k} \frac{12^{k/2}}{k}\ll k^{3/2 } (3 \big / 4)^{k/2}.
\end{align}
where $S(D_1,D_2)$ is as in \eqref{SD1D2}. Thus we have proved
\begin{prop} \label{skp}
Let $\sup(\skk^p)$ be as above. Then
\begin{equation}
    k^{5/2} \ll \sup(\skk^p) \ll k^{5/2+\epsilon}.
\end{equation}
\end{prop}
Thus the density $d \asymp 1$ -- it will be interesting to calculate it explicitly.

\begin{cor} \label{skpcor}
 For $f \in S_k$ normalized Hecke eigenform, one has the `trivial' bound on its size: $\sup_\h |F^\circ_f|/\norm{F_f} \ll k^{5/4+\epsilon}$.
\end{cor}

\subsection{The intrinsic BK}
Recall the space $\skk^\circ$ spanned by the pullbacks of all $F \in \skk$ and let $B_k^\circ \subset \skk^\circ$ be an orthogonal basis. Further, put $(S_k \otimes S_k)^\circ$ to be the $\mbb C$ span of $g \otimes g$ which is identified with $g(\tau)  g(\tau'), \, g \in S_k, \, $ a normalized Hecke eigenform. Call this set $B_k^{\circ \circ}$.
The following lemma gives a description of the space of "pullbacks" of SK lifts.
\begin{lem}\label{SK0tensor}
With the above notation, $\skk^\circ = (S_k \otimes S_k)^\circ$.
\end{lem}
\begin{proof}
From \eqref{pullback-exp}, clearly $\skk^\circ \subset (S_k \otimes S_k)^\circ$. For the other direction, we compare dimensions. $\dim (S_k \otimes S_k)^\circ = \dim S_k$. Let us put $\skk(0) := \{ F_f \mid F^\circ_f =0\}$, so that $\skk(0)$ is the kernel of the Witt map $W$ ($F \mapsto F^\circ = F|_{z=0}$) restricted to $\skk$.

Then from \cite[Proof of Thm. 1.9]{liu2014growth}, we know that $\dim \skk(0) = \dim M_{k-10} = \dim S_{k+2}$. Therefore, the dimension of the image $W(\skk)$ of the Witt map is $\dim S_{2k-2} - \dim S_{k+2} = \dim S_k$. Since $W(\skk)$ is simply $\skk^\circ$, we are done.
\end{proof}
The lemma has a curious corollary about linear independence of (square-root) central $L$-values. Recall the quantity $c_{g,f}$. Up to constants, its square equals $L(1/2, \mrm{sym}^2 g \times f)$. Let $d_1=\dim S_{2k-2}$ and $d_2=\dim S_k$. Consider $\mc L:= \mc L(d_1,d_2)$, the $d_1 \times d_2$ matrix $(c_{g,f})_{f,g}$ where $f,g$ run over normalized Hecke eigenforms in $S_{2k-2}, S_k$ respectively. Writing the elements $F^\circ$ ($F$ Hecke eigenform in $\skk$) as a column, say $C_k$, one can recast \eqref{pullback-exp} into a matrix equality: $C_k = \mc L \cdot B_k^{\circ \circ}$. If we let $\mf L$ to be the linear operator corresponding to $\mc L$ mapping the span $(S_k \otimes S_k)^\circ$ of $B_k^{\circ \circ}$ to the span
$\skk^\circ$ of $C_k$. Clearly $\mf L$ is onto. This implies that the rank of $\mc L$ equals $d_2$. We summarize this in the corollary below. 


\begin{cor} \label{maxrank}
 Let $\mc L$ be the $d_1\times d_2$ matrix as above. Then $\rk(\mc L)$ is maximal.
\end{cor}
Recall that $B_k^\circ \subset \skk^\circ$ denotes an orthogonal basis. Let us define
\begin{align} \label{pullback-bkdef}
    \sup(\skk^\circ) := \sup \nolimits_{(\tau,\tau')\in\h \times \h}  \sum \nolimits_{G \in B_k^\circ} (vv')^k |G(\tau,\tau')|^2/\langle G, G \rangle.
\end{align}

In view of \lemref{SK0tensor}, since the spaces $\skk^\circ$ and $(S_k \otimes S_k)^\circ$ are the same, we can also write \eqref{pullback-bkdef} as
\begin{equation}\label{altSK0}
    \sup(\skk^\circ)=\sup((S_k \otimes S_k)^\circ)= \sup \nolimits_{(\tau,\tau')\in\h \times \h} \sum \nolimits_{g\in B_k}(vv')^{k} |g(\tau)g(\tau')|^2/\lan g, g\ran ^2.
\end{equation}
\begin{prop}
With the above notation,
$k^2 \ll \sup(\skk^\circ) \ll k^{2+\epsilon}$.
\end{prop}

\begin{proof}
In view of the Lemma \eqref{SK0tensor}, for the basis $B_k^\circ$ we can take the set $\{ g(\tau) \times g(\tau'), \, g = h \norm{h}_2^{-1} \in S_k$, $ h$  a  normalized Hecke eigenform\}. Then $a_{g\times g}(m,n)=a_g(m) a_g(n)$. Therefore as in section \ref{seclbd}, with $m=n=1$ and by using the Cauchy--Schwarz and basic integral inequalities, we have $ \sup(\skk^\circ)$ is
\begin{align}
    &\ge (vv')^k \exp(- 4 \pi (v+v')) \sumn_h |a_{h\times h}(1,1)|^2 \lan h, h \ran^{-2} \\
    &\gg \frac{(vv')^k \exp(- 4 \pi (v+v'))}{k} \big(\sumn_h \frac{ |a_{h}(1)|^2}{\lan h, h \ran}\big)^2 
     \gg \frac{(vv')^k \exp(- 4 \pi (v+v'))}{k} \frac{(4 \pi)^{2k}}{\Gamma(k-1)^2 }  \gg k^2,
\end{align}
where we have chosen $v=v'=k/4\pi$ and used the fact that the sum over $h$ is $\asymp (4\pi)^{k-1}/\Gamma(k-1)$ (see \cite[Proposition 14.5]{iwaniec2004analytic}).

For the upper bound we use individual sup-norm bounds for $h$: from the result of Xia \cite{xia}, we know that $\norm{g}_\infty \ll k^{1/4+\epsilon}$ for $g$ as above. Thus, from \eqref{altSK0}, we see that
\begin{equation}
    \sum \nolimits_{g}(vv')^{k} \frac{|g(\tau)g(\tau')|^2}{\lan g, g\ran ^2} \ll (k^{1/2+\epsilon})^2 \cdot k\ll k^{2+\epsilon}. \qedhere
\end{equation}
\end{proof}
\begin{cor} \label{skocor}
 For $F_f \in \skk$ a normalized Hecke eigenform, one has the `trivial' bound on its size: $\sup_\h |F^\circ_f|/\norm{F^\circ_f} \ll k^{1+\epsilon}$, provided $F^\circ_f \ne 0$.
\end{cor}

Note that we of course have the bound $\sup_\h |F^\circ_f|/\norm{F_f} \ll k^{5/4} $ from \corref{skpcor} or from \thmref{mainthm1} by restriction of domain. If we assume the conjectural size $N(F_f) \sim 2$ as $k \to \infty$ (cf. \cite{liu2014growth}, \cite{BKY}), then it is immediate from \corref{skocor} that $ \sup_\h |F^\circ_f|/\norm{F_f} \ll k^{1+\epsilon} $, and $k^{1/2+\epsilon}$ if we expect that \corref{skocor} can be improved to $k^{1/2+\epsilon}$. To see this, we write and note $|F^\circ_f|/\norm{F_f} = |F^\circ_f|/\norm{F^\circ_f} \cdot \norm{F^\circ_f}/\norm{F_f} \asymp |F^\circ_f|/\norm{F^\circ_f} \cdot N(F_f)$. However, the current record for the size of $N(F_f)$ is $N(F_f) \ll k^{1-\delta}$ (cf. \cite{das-anamby-jnt}) for a very small $\delta \approx 0.0047$ and this does not help (unconditionally) towards the betterment of \corref{skpcor} that one would like to have.

\subsection{Appendix to section~\ref{pullback-sec}}
\label{appendix1}
In this section, we calculate the $L^\infty$ size of the image of the Witt operator on $S^2_k$ which we define by (so $W(F)=F^\circ$ from the previous section)
\begin{align}
    W(F)(\tau,\tau') := F_{z=0}.
\end{align}
Since the arguments here are closely related to our ideas in this paper, and especially to this section, we feel it is apt to include this here as an appendix. Assume that $k$ is even. Recall that our point of interest is 
\begin{align} \label{witt-sup}
\sup(W(S^2_k)) = \sup\nolimits_{\tau_1,\tau_2 \in \h} \sumn_{F \in \mc{BW}(S^2_k)} \frac{(v_1 v_2)^{k/2} |W(F)|^2}{\lan W(F), W(F) \ran},
\end{align}
where $\mc{BW}(S^2_k)$ is an orthonormal basis for $W(S^2_k)$.
The main idea is to find a convenient basis for $W(S^2_k)$ and work with it, owing to the invariance of the Bergman kernel with respect to the choice of a basis.

First, note that
the size of $S_k \otimes S_k$ is clearly $k^3$ as its BK is the product of that of $S_k$ with itself. Secondly, note that since $F(\smat{\tau}{z}{z}{\tau'}) = F(\smat{\tau'}{z}{z}{\tau})$ for all $F \in S^2_k$, one has $W(F)   \in (S_k \otimes S_k)^{\mrm{sym}}$, where $(S_k \otimes S_k)^{\mrm{sym}}$ be the space of all symmetric tensors in $S_k \otimes S_k$. Its dimension is $s_k(s_k+1)/2$, where we have put $s_k = \dim S_k$. From \cite[1.5~Hilfssatz, Folgerung, p.150]{Fr} we know that when $k$ is even, $W$ is surjective. Therefore, if we want to find the size of the space of pullbacks $W(S^2_k)$ for $S^2_k$ -- we merely have to choose an orthonormal basis of  $(S_k \otimes S_k)^{\mrm{sym}}$.
Since $2$ is invertible for us, it is clear that $(S_k \otimes S_k)^{\mrm{sym}} = (S_k \otimes S_k)^{\mrm{diag}}$ where $(S_k \otimes S_k)^{\mrm{diag}} = \{ \sum_{i=1}^s c_i g_i \otimes g_i, \, g_i \in S_k, \, s \ge 1 \}$.
This will indicate how much of the $L^\infty$ mass of $S^2_k$ is supported by the diagonal.

We will prove that its size is $k^3$. To see this, first put $s(f,g)= f\otimes g + g \otimes f$.
Further, let $\{g_j\}_j$ be the normalized Hecke basis for $S_k$.
We note that $s(g_i, g_j)$, $i \le j$, $1 \le j \le s_k$ is a basis for 
$(S_k \otimes S_k)^{\mrm{sym}}$. Further, the elements of this basis are pairwise orthogonal. To see this, note that 
\begin{align} \label{pullb-ortho}
\lan s(g_1,g_2), s(g_3,g_4) \ran = 2\lan g_1, g_3 \ran \lan g_2, g_4 \ran + 2\lan g_1, g_4 \ran \lan g_2, g_3 \ran.
\end{align}
Suppose that (without loss of generality) $g_1,g_3$ are unequal, hence orthogonal. Then \eqref{pullb-ortho}  reduces to $2\lan g_1, g_4 \ran \lan g_2, g_3 \ran$ which can be non-zero if and only if $g_1=g_4, g_2=g_3$.

As before, the quantity which measures the size is the BK: we sum over the modulus square of the elements in an orthonormal basis. Since the BK does not depend on the choice of the orthonormal basis, we can write
\begin{align} \label{pullb-bk}
    \sup((S_k \otimes S_k)^{\mrm{sym}}) = \sumn_{i \le j} \frac{ (vv')^k |s(g_i,g_j)|^2 }{\lan  s(g_i,g_j),  s(g_i,g_j) \ran}.
\end{align}

Let $P_n$ be the $n$-th Poincar\'e series in $S_k$. We can write 
\begin{align}
s(P_m , P_n) = \frac{\Gamma(k-1)^2}{(4 \pi m)^{k-1} (4 \pi n)^{k-1}} \sumn_{i \le j} \frac{ ( \overline{a_i(m)} \overline{a_j(n)} + \overline{a_i(n)} \overline{a_j(m)} ) }{\lan  s(g_i,g_j),  s(g_i,g_j) \ran} s(g_i , g_j);
\end{align}
from which we get a Petersson formula: 
\begin{align}
\frac{\Gamma(k-1)^2}{(4 \pi m)^{k-1} (4 \pi n)^{k-1}} \sumn_{i \le j} |a_i(m) a_j(n) + a_i(m) a_j(n)|^2 = p_m(m)p_n(n).
\end{align}
Now putting $m=n=1$ and following the method in section~\ref{lbd-sk}, the lower bound $k^3$ is immediate.

For the upper bound, let $h_i$ denote the orthonormalized $g_i$. Then the RHS of \eqref{pullb-bk} can be bounded (up to absolute constants) by
\begin{align}
    \sumn_{i \le j} (v^k |h_i(\tau)|^2 \, v'^k |h_j(\tau')|^2 + v^k |h_j(\tau)|^2 \, v'^k |h_i(\tau')|^2),
\end{align}
which is bounded by $(\sum_i v^k |h_i(\tau)|^2) (\sum_j v^k |h_j(\tau')|^2) \ll k^3$ by \cite[Section 7.2]{sd-hk}.

Summarizing the above calculations, we can now state the following. 
\begin{thm} \label{witt-thm}
With the above notation and setting, if $k$ is even, $\sup(W(S^2_k)) \asymp k^3.$
\end{thm}

\printbibliography
\end{document}